\documentclass{article}
\usepackage[margin = 1in,margin=2.5cm]{geometry}
\usepackage[all]{xy}
\usepackage{tikz}
\usepackage{tikz-cd}
\usetikzlibrary{3d,calc}
\usepackage{amsmath}
\usetikzlibrary{positioning, shapes, calc}
\usepackage{graphicx}
\usepackage{tikz-cd}
\usetikzlibrary{cd, arrows.meta, positioning}
\usepackage[margin=2.5cm]{geometry}
\usetikzlibrary{patterns}

\usepackage{graphicx}
\usepackage{amssymb}
\usepackage{amsmath}
\usepackage{epstopdf}
\usepackage{amsthm}
\usepackage{hyperref}
\usepackage{amscd}
\usepackage{mathrsfs}
\usepackage{mathabx}
\usepackage{tikz}
\usepackage{amssymb}
\usepackage{xcolor}
\usepackage{textcomp}

\definecolor{ForestGreen}{RGB}{34,139,34}
\newcommand{\greencheck}{\textcolor{ForestGreen}{\textbf{\checkmark}}}
\newcommand{\notcheckmark}{\textcolor{ForestGreen}{\textbf{\checkmark\llap{\textcolor{blue}{\raisebox{0.15ex}{\textbf{\textbackslash}}}}}}}
\newcommand{\redcross}{\textcolor{red}{\textbf{\texttimes}}}
\hypersetup{
	colorlinks=true,
	linkcolor=black,
	filecolor=black,      
	urlcolor=black,
	citecolor=blue,
}

\usepackage{pict2e}

\makeatletter
\newcommand{\getfontdimen}[3]{
  \fontdimen#1
  \ifx#2\displaystyle\textfont\else
  \ifx#2\textstyle\textfont\else
  \ifx#2\scriptstyle\scriptfont\else
  \scriptscriptfont\fi\fi\fi #3%
}
\NewDocumentCommand{\barred}{O{}m}{#1{\mathpalette\barred@{#2}}}
\newcommand{\barred@}[2]{%
  \begingroup
  \sbox\z@{$\m@th#1#2$}
  \setlength{\dimen@}{\dimeval{\ht\z@+\dp\z@+0.4\getfontdimen{5}{#1}{2}}}
  \sbox\tw@{%
    \raisebox{\dimeval{-\dp\z@-0.2\getfontdimen{5}{#1}{2}}}{%
      \begin{picture}(0,\dimen@)
      \roundcap
      \linethickness{1\getfontdimen{8}{#1}{3}}
      \Line(0,0)(0,\dimen@)
      \end{picture}%
    }%
  }%
  \vphantom{\copy\tw@}
  \ooalign{%
    \box\z@\cr 
    \hidewidth\box\tw@\hidewidth\cr 
  }%
  \endgroup
}

\newcommand{\bfa}{\mathbf{a}}

\newcommand{\bfe}{\mathbf{e}}
\newcommand{\bfv}{\mathbf{v}}

\newcommand{\Garsia}{\mathscr{G}}

\newcommand{\ZZ}{\mathbb{Z}}
\newcommand{\NN}{\mathbb{N}}
\newcommand{\QQ}{\mathbb{Q}}
\newcommand{\kk}{\Bbbk}

\newcommand{\Aut}{\operatorname{Aut}}
\newcommand{\Sd}{\operatorname{Sd}}
\newcommand{\rk}{\operatorname{rk}}
\newcommand{\Part}{\mathscr{P}}
\newcommand{\shape}{\operatorname{shape}}

\newcommand{\Rep}{\operatorname{Rep}}
\newcommand{\Hilb}{\operatorname{Hilb}}
\newcommand{\lb}{\operatorname{lb}}

\theoremstyle{plain}
\newtheorem{theorem}{Theorem}[section]
\newtheorem{prop}[theorem]{Proposition}
\newtheorem{lemma}[theorem]{Lemma}
\newtheorem{cor}[theorem]{Corollary}
\newtheorem{question}[theorem]{Question}
\newtheorem{observation}[theorem]{Observation}

\theoremstyle{definition}
\newtheorem{definition}[theorem]{Definition}
\newtheorem{example}[theorem]{Example}
\newtheorem*{remark}{Remark}
\newtheorem{setup}[theorem]{Setup}
\newtheorem{notation}[theorem]{Notation}
\newtheorem{convention}[theorem]{Convention}
\newtheorem{algorithm}[theorem]{Algorithm}

\newcommand{\Char}{\operatorname{char}}

\makeatletter
\newcommand{\subjclass}[2][1991]{
  \let\@oldtitle\@title
  \gdef\@title{\@oldtitle\footnotetext{#1 \emph{Mathematics Subject Classification.} #2}}
}
\newcommand{\keywords}[1]{%
  \let\@@oldtitle\@title%
  \gdef\@title{\@@oldtitle\footnotetext{\emph{Key words and phrases.} #1.}}%
}
\makeatother

\title{Comparing the face rings of a boolean complex and its barycentric subdivision}
\author{Ben Blum-Smith and Sophie Marques}
\date{} 


\begin{document}

\maketitle

\begin{center}
    {\em Dedicated to the memory of Adriano Garsia.}
\end{center}

\begin{abstract}
We consider the relationship between the Stanley--Reisner ring (a.k.a. face ring) of a simplicial or boolean complex $\Delta$ and that of its barycentric subdivision. These rings share a distinguished parameter subring. S. Murai asked if they are isomorphic, equivariantly with respect to the automorphism group $\Aut(\Delta)$, as modules over this parameter subring. We show that, in general, the answer is no, but for Cohen--Macaulay complexes in characteristic coprime to $|\Aut(\Delta)|$, it is yes, and we give an explicit construction of an isomorphism. To give this construction, we adapt a pair of tools introduced by A. Garsia in 1980. The first one transfers bases from a Stanley--Reisner ring to closely related rings of which it is a Gr\"obner degeneration, and the second identifies bases to transfer.
\end{abstract}

\tableofcontents

\section{Introduction}

The algebraic structure of the  Stanley--Reisner ring (or {\em face ring}) $\kk[\Delta]$ of a simplicial complex $\Delta$  reflects the topology of $\Delta$; for example, the Cohen--Macaulay and Gorenstein properties of $\kk[\Delta]$ are detectable in the (reduced and relative) homology of the geometric realization of $\Delta$. On the other hand, $\kk[\Delta]$ also reflects combinatorial information about $\Delta$ not visible in the topology: two nonisomorphic simplicial complexes will in general  have nonisomorphic Stanley--Reisner rings, even if their geometric realizations are homeomorphic. A key example is the barycentric subdivision $\Sd\Delta$ of $\Delta$: $\kk[\Delta]$ and $\kk[\Sd\Delta]$ are not isomorphic as rings, although $\Delta$ and $\Sd\Delta$ have homeomorphic geometric realizations. Nonetheless, these rings are closely related, and a natural question is: {\em how close is the relationship?}

We consider this question at the generality of boolean complexes, a generalization of simplicial complexes (see Section~\ref{sec:background} for  definitions and notation). If $\Delta$ is a boolean complex, then $\kk[\Delta]$ is an {\em algebra with straightening law (ASL)} \cite{eisenbud1980introduction, hodge-algebras}, and $\kk[\Sd\Delta]$  is the associated {\em discrete ASL}. R. Stanley \cite{stanley1991f} observed that this implies that the depth of $\kk[\Delta]$ is at least that of $\kk[\Sd\Delta]$. Then, A. Duval \cite{duval1997free} showed that in fact the depths are {\em equal}. Much more recently, A. Conca and M. Varbaro demonstrated that this is an example of a general phenomenon tying the rings together closely: the discrete ASL associated to any ASL is a squarefree Gr\"obner degeneration of it, and Conca and Varbaro's spectacular result \cite{conca2020square} then implies they have all the same extremal Betti numbers when resolved over the polynomial ring whose indeterminates index the ASL generators (and therefore, they have the same depth).

In the special case of $\kk[\Delta]$ and $\kk[\Sd\Delta]$, recent work of A. Adams and V. Reiner \cite{adams-reiner} conjectured a further close connection. These rings share a common parameter subring $\kk[\Theta]$ (see Section~\ref{sec:background} for notation). Adams and Reiner conjectured that, when resolved over $\kk[\Theta]$, {\em all} the Betti numbers of $\kk[\Delta]$ and $\kk[\Sd\Delta]$  are equal. In fact, they conjectured something stronger \cite[Conjecture~6.1]{adams-reiner}. If a group $G$ of automorphisms acts on $\Delta$ (and therefore also on $\Sd \Delta$), the parameter subring $\kk[\Theta]$ is pointwise-fixed, and Adams and Reiner's conjecture states that the equivariant Betti numbers, which are refinements of the Betti numbers taking values in the Grothendieck ring of $G$ over $\kk$, are then equal (see Section~\ref{sec:nonconstructive} for more detail on the Grothendieck ring).  These conjectures carry no hypothesis on the characteristic of $\kk$, or on the boolean complex $\Delta$ (beyond finiteness).

After a version of Adams and Reiner's preprint appeared on the arXiv, S. Murai posed the following question, upgraded to a conjecture by Adams \cite[Conjecture~3.3.4]{adams2023further}, about a further strengthening of this conjecture.

\begin{question}[Murai]\label{q:murai}
Are the Stanley--Reisner rings $\kk[\Delta]$ and $\kk[\Sd\Delta]$ isomorphic as modules over $\kk[\Theta]$? Are they $G$-equivariantly isomorphic?
\end{question}

We study the existence of an equivariant isomorphism. We give both a negative and a positive result. In arbitrary characteristic, we show there may fail to be an equivariant isomorphism. On the other hand, we prove that in the Cohen--Macaulay, coprime characteristic case, an equivariant isomorphism does exist, and we give an explicit construction of such an isomorphism.

\begin{theorem}[Negative result]\label{thm:no-equivariant-iso}
Let $d\geq 2$, let $\Delta_d$ be a $d$-simplex, and let $G$ be its automorphism group. Let $\kk$ be a field of characteristic 2. Then there is no $G$-equivariant $\kk[\Theta]$-module isomorphism $ \kk[\Sd\Delta_d] \rightarrow \kk[\Delta_d]$.
\end{theorem}

\begin{theorem}[Positive result]\label{thm:yes-CM-in-coprime}
    Let $\Delta$ be a finite boolean complex that is Cohen--Macaulay over a field $\kk$, and suppose $G$ is a group of automorphisms of $\Delta$ whose order is a unit in $\kk$. Then there exists a graded $G$-equivariant $\kk[\Theta]$-module isomorphism $\kk[\Sd\Delta]\rightarrow\kk[\Delta]$, and an algorithm to compute it explicitly.
\end{theorem}

Cohen--Macaulayness implies that both $\kk[\Delta]$ and $\kk[\Sd\Delta]$ are module-free over $\kk[\Theta]$, thus they are certainly $\kk[\Theta]$-module isomorphic; general theory implies that the isomorphism can be taken to be graded. Therefore, the key points in Theorem~\ref{thm:yes-CM-in-coprime} are the existence of a {\em $G$-equivariant} isomorphism, and its explicit construction. And because the $d$-simplex is Cohen--Macaulay in any characteristic, $\kk[\Delta]$ and $\kk[\Sd\Delta]$ are non-equivariantly $\kk[\Theta]$-module isomorphic in the example in Theorem~\ref{thm:no-equivariant-iso}, and the key point is the impossibility of $G$-equivariance. It remains plausible that Adams and Reiner's original conjecture on the equivariant Betti numbers \cite[Conjecture~6.1]{adams-reiner} holds, and also that the weaker (non-equivariant) form of Murai's question / Adams' conjecture has a positive answer. 

Theorem~\ref{thm:no-equivariant-iso} is based on a hands-on analysis of what the existence of an equivariant isomorphism would force upon subrings. In particular, in the situation of the theorem, the automorphism group is $\mathfrak{S}_n$, the symmetric group on $n$ points with $n=d+1$, and a $G$-equivariant $\kk[\Theta]$-module isomorphism would also imply the existence of a $C_2 \cong \mathfrak{S}_n/\mathfrak{A}_n$-equivariant $\kk[\Theta]$-module isomorphism between the $\mathfrak{A}_n$-invariant subrings (where $\mathfrak{A}_n$ is the alternating subgroup). These have a simple description as free $\kk[\Theta]$-modules of rank two, and we find a contradiction by working explicitly with bases.

The existence part of Theorem~\ref{thm:yes-CM-in-coprime} is proven in two different ways. One is via the explicit construction of an isomorphism. The other is a nonconstructive proof that hews closely to ideas in \cite{adams-reiner}, and was developed in conversation with Victor Reiner.

The main work of this paper is the proof of Theorem~\ref{thm:yes-CM-in-coprime} via the explicit construction of a $G$-equivariant $\kk[\Theta]$-module isomorphism. It is based on methods developed by A. Garsia \cite{garsia}, which we adapt to the present context. When $\Delta$ is Cohen--Macaulay, Garsia's techniques allow to transfer a $\kk[\Theta]$-module basis for $\kk[\Sd\Delta]$ to a $\kk[\Theta]$-module basis for $\kk[\Delta]$, from which can be constructed a non-equivariant isomorphism $\Phi$. Further, the same ideas used to prove Garsia's basis transfer theorem also allow us to show that, in the coprime characteristic situation, the equivariant map obtained by averaging $\Phi$ over the group $G$ remains an isomorphism. Finally, a different circle of ideas from \cite{garsia} allows to construct a $\kk[\Theta]$-module basis for $\kk[\Sd\Delta]$ in the first place.

Beyond the proofs of Theorems~\ref{thm:no-equivariant-iso} and \ref{thm:yes-CM-in-coprime}, an important contribution of the present work is the reformulation and re-presentation of the ideas from \cite{garsia} that we use. In the context of adapting them, we do a significant reorganization of these ideas to draw out and foreground what we view as the underlying conceptual picture of $\kk[\Sd\Delta]$ and $\kk[\Delta]$ that they provide. Specifically:
\begin{itemize}
    \item The method for transferring bases presented in \cite[Section~6]{garsia} in the context of partition rings (and later generalized in various directions in \cite{baclawski1981combinatorial, baclawski1981rings, garsia-stanton}), is formulated here (for an arbitrary boolean complex $\Delta$) as resulting from the underlying fact that $\kk[\Delta]$ is  {\em filtered over the poset of partitions with respect to dominance order}, and $\kk[\Sd\Delta]$ is the associated graded algebra; see Section~\ref{sec:garsia} (and especially Section~\ref{sec:filtering-garsia} and \ref{sec:garsia-transfer}) below.
    \item The linear-algebraic tests of Cohen--Macaulayness of a ranked poset given in \cite[Section~3]{garsia}, are here formulated, at the generality of an arbitrary balanced boolean complex, as springing from a beautiful characterization of Cohen--Macaulayness in terms of a certain subspace arrangement in a single finite-dimensional vector space over $\kk$; see Section~\ref{sec:garsia-LA-CM} (and especially Theorem~\ref{thm:garsia-LA-CM}) below.
\end{itemize} 

The structure of the paper is as follows. In Section~\ref{sec:background}, we give background on $\kk[\Delta]$ and $\kk[\Sd\Delta]$, and fix the notation used throughout. Sections~\ref{sec:garsia} and \ref{sec:garsia-LA-CM} are explications and generalizations of tools from \cite{garsia}, as follows. Section~\ref{sec:garsia} concerns the grading of $\kk[\Sd\Delta]$ and filtering of $\kk[\Delta]$ by partitions ordered by dominance, and, using this, explicates Garsia's method for transferring bases. Section~\ref{sec:garsia-LA-CM} gives Garsia's linear-algebraic characterization of Cohen--Macaulayness in terms of a certain subspace arrangement. Section~\ref{sec:counterex} proves Theorem~\ref{thm:no-equivariant-iso} (the counterexample to Question~\ref{q:murai}). Section~\ref{sec:positive} proves Theorem~\ref{thm:yes-CM-in-coprime}, using the tools developed in Sections~\ref{sec:garsia} and \ref{sec:garsia-LA-CM}. 

\section{Setup and background}\label{sec:background}

\subsection{Boolean complexes, barycentric subdivisions, and Stanley--Reisner rings}

We assume the reader is familiar with the notion of a finite {\em simplicial complex} $\Delta$, its associated {\em Stanley--Reisner ring} or {\em face ring} $\kk[\Delta]$ over a given field $\kk$, and its {\em geometric realization} $|\Delta|$---which we view either as a bare topological space or, more richly, as a CW complex. We also assume familiarity with the {\em face poset} $P(\Delta)$, and the {\em barycentric subdivision} $\Sd\Delta$ (although they are described in passing below). References on the Stanley--Reisner ring include \cite{bruns-herzog, stanley1996combinatorics, miller2005combinatorial}. See \cite{wachs2006poset} for the face poset and barycentric subdivision.

A {\em boolean complex} $\Delta$  (name introduced in \cite{garsia-stanton}; also known as a {\em simplicial cell complex} \cite[Section~2.8]{buchstaber-panov} or a {\em generalized simplicial complex} \cite[Section~2.2]{kozlov}), is a finite regular CW complex in which every cell is a simplex with its standard regular CW structure, and the attaching maps are homeomorphisms sending cells homeomorphically to cells.\footnote{The definition is written out carefully in \cite[Definition~2.41]{kozlov} (under the name ``generalized simplicial complex"). We will only be viewing boolean complexes, and maps between them, as objects that are fully specified by combinatorial data. Thus, we view each cell as parametrized by the set of convex combinations of its vertices, and all maps between boolean complexes must send cells to cells and preserve convex combination. We also require this of the attaching maps involved in the construction of the CW complex in the first place.} It is a generalization of a simplicial complex in which a pair of faces can meet along an arbitrary subcomplex rather than necessarily a single face (for example, two faces may meet at all of their vertices without being identical). One natural way they arise is as quotients of the Coxeter complex of a reflection group by a subgroup of that reflection group---this was the motivation in \cite{garsia-stanton}---or, more generally, as  quotients of balanced simplicial complexes by groups of label-preserving automorphisms.

The usual language of simplicial complexes is readily imported into the context of boolean complexes. Cells are {\em faces}. The $0$-cells are {\em vertices}. The $1$-cells are {\em edges}. One defines the {\em face poset} as for any CW complex: the elements are the cells, and for cells $\alpha,\beta$, $\alpha \preceq \beta$ means that $\alpha$ is contained in $\beta$'s closure, and we say in this case that {\em $\alpha$ is a face of $\beta$}, or, more briefly, {\em $\alpha$ belongs to $\beta$}. (One also says in this case that $\alpha$ and $\beta$ are {\em incident}, although this does not specify the direction of containment.) The maximal elements in the face poset are {\em facets}. If all facets have the same dimension, then the complex is {\em pure}.

A simple example of a boolean complex that is not a simplicial complex is a pair of vertices $v,w$ connected by a pair of distinct edges $\alpha,\beta$. See Figure~\ref{fig:running-ex}. We use this as a running example in the below.

\begin{figure}
\begin{center}
\includegraphics{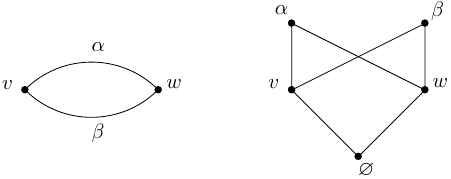}
\end{center}
\caption{Left: a boolean complex $\Delta$ that is not a simplicial complex. Right: its augmented face poset $\widehat P(\Delta)$ (including the minimal element $\varnothing$), per Definition~\ref{def:SR-ring-of-boolean}.}\label{fig:running-ex}
\end{figure}

The face poset of a boolean complex, with a minimal (``empty") face appended, is called a {\em simplicial poset} \cite{stanley1991f} (or a {\em poset of boolean type} \cite{bjorner1984posets}). Simplicial posets can be recognized by the fact that they have a unique minimal element and every lower interval is a finite boolean lattice (i.e., isomorphic to the poset of subsets of a finite set, ordered by inclusion).\footnote{The combinatorial literature on boolean complexes and simplicial posets tends, as a generality, to elide the distinction between the poset and the CW complex, since either can be reconstructed from the other. For example, \cite{garsia-stanton}, which introduced the term {\em boolean complex}, actually defined it as a kind of poset, in spite of the word ``complex". We make an effort to be careful to keep the two notions separate in our language, although we probably have not fully succeeded.}

The Stanley--Reisner ring of a boolean complex $\Delta$ (equivalently, of its associated simplicial poset) was defined in \cite{stanley1991f} and studied further in \cite{reiner, duval1997free}. In hindsight, the construction was already implicit in \cite{garsia-stanton}. We review the definition, and discuss pertinent properties. The discussion is formatted as a sequence of numbered paragraphs labeled ``Setup" for later cross-referencing.

\begin{definition}[Stanley--Reisner ring of a boolean complex]\label{def:SR-ring-of-boolean}
    Let $\Delta$ be a boolean complex. Let $P(\Delta)$ be its face poset, and let $\widehat P(\Delta)$ be its {\em augmented face poset}, constructed from $P(\Delta)$ by appending a minimal face $\varnothing$, i.e., 
    \[
    \widehat P(\Delta):= P(\Delta) \cup \{\varnothing\},
    \] with $\varnothing \preceq\alpha$ for all $\alpha\in P(\Delta)$. Let $\kk$ be a field, and let $S$ be a polynomial ring over $\kk$ with indeterminates $x_\alpha$ indexed by the elements $\alpha$ of $\widehat P(\Delta)$. The {\em Stanley--Reisner ideal} $I_\Delta$ of $\Delta$ is the ideal in $S$ generated by the following elements:
    \begin{enumerate}
        \item $x_\varnothing - 1$\label{item:empty-is-one}
        \item $x_\alpha x_\beta$ for every pair $\alpha,\beta \in \widehat P(\Delta)$ lacking a common upper bound in $\widehat P(\Delta)$\label{item:no-common-upper-bound}
        \item $x_\alpha x_\beta - x_{\alpha \wedge \beta}\sum_{\gamma\in \operatorname{lub}(\alpha,\beta)} x_\gamma$ for every pair $\alpha,\beta \in \widehat P(\Delta)$ possessing a common upper bound in $\widehat P(\Delta)$, where the sum is over the set $\operatorname{lub}(\alpha,\beta)$ of minimal common upper bounds for $\alpha, \beta$\label{item:common-upper-bound}
    \end{enumerate}
    In \ref{item:common-upper-bound}, the meet $\alpha\wedge \beta$ is well defined in $\widehat P(\Delta)$ because, having a common upper bound, $\alpha$ and $\beta$ are in a lower interval of $\widehat P(\Delta)$ together, and, $\widehat P(\Delta)$ being a simplicial poset, every lower interval is a boolean lattice (and therefore a lattice).

    This all established, the {\em Stanley--Reisner ring} (or {\em face ring}) of $\Delta$ is the ring
    \[
    \kk[\Delta]:= S/I_\Delta.
    \]
\end{definition}

\begin{example}\label{ex:x-ring-running-ex}
    In our running example from Figure~\ref{fig:running-ex}, $S$ is the polynomial ring generated by $x_\varnothing$, $x_u$, $x_v$, $x_\alpha$, and $x_\beta$, and the nontrivial generators of the Stanley-Reisner ideal $I_\Delta$ are
    \[
    x_\varnothing-1,\text{ }x_\alpha x_\beta,\text{ and }x_ux_v - x_\varnothing(x_\alpha+x_\beta),
    \]
    respectively of types~\ref{item:empty-is-one}, \ref{item:no-common-upper-bound}, and \ref{item:common-upper-bound} of Definition~\ref{def:SR-ring-of-boolean}. (These are the only nontrivial generators because all pairs of faces besides $u,v$ and $\alpha,\beta$ are comparable; then the corresponding generator of type~\ref{item:common-upper-bound} reduces to zero, as noted in Setup~\ref{set:asl} below.)
\end{example}

\begin{notation}
    In Definition~\ref{def:SR-ring-of-boolean} and going forward, we engage in mild abuse of notation by using the same symbols $x_\alpha$ to denote both the indeterminates of the parent polynomial ring $S$, and their images in the quotient $\kk[\Delta] = S/I_{\Delta}$. We will do the same with the generators $y_\alpha$ discussed below for the Stanley--Reisner ring $\kk[\Sd\Delta]$ of the barycentric subdivision.

    By the same token, the elements \ref{item:empty-is-one}, \ref{item:no-common-upper-bound} and \ref{item:common-upper-bound} above are, prima facie, elements of the ring $S$ that are contained in (and generate) the ideal $I_\Delta$, but in the ring $\kk[\Delta]$ they become {\em equations} $x_\varnothing = 1$,
    $x_\alpha x_\beta = 0$, and
        $x_\alpha x_\beta = x_{\alpha \wedge \beta}\sum_{\gamma\in \operatorname{lub}(\alpha,\beta)} x_\gamma$, and we will refer to them (especially \ref{item:no-common-upper-bound} and \ref{item:common-upper-bound}) whether we mean elements of $I_\Delta \subset S$ or equations in $\kk[\Delta]$.
\end{notation}

\begin{setup}[Simplicial complexes as boolean complexes; relation between the rings]\label{set:boolean-simplicial-identification}
    If $\Delta$ is an abstract simplicial complex on a vertex set $V(\Delta)$, its augmented face poset (including the minimal empty face) is a simplicial poset. (This is the etymology of ``simplicial poset".) So its geometric realization, including the CW structure, is a boolean complex. It is in this sense that boolean complexes generalize simplicial complexes.
    
    With respect to the boolean complex structure, the ring constructed in Definition~\ref{def:SR-ring-of-boolean} is an alternative description of the Stanley--Reisner ring $\kk[\Delta]$ of the simplicial complex $\Delta$, defined in the usual way (i.e., generated by indeterminates $x_v$ indexed by the vertices, mod the squarefree monomials corresponding to non-faces). It is for this reason that the ring of Definition~\ref{def:SR-ring-of-boolean} is reasonably called a Stanley--Reisner ring, and that the notation $\kk[\Delta]$ may be regarded as unambiguous whether $\Delta$ is viewed as a simplicial or boolean complex. The identification is given by beginning with the usual Stanley--Reisner ring $\kk[\{x_v\}_{v\in V(\Delta)}] / (\text{non-faces})$, where $V(\Delta)$ is the vertex set of $\Delta$, and then expanding the set of generators to include one for every face of $\Delta$: 
    \[
    x_\alpha := \prod_{v\in \alpha} x_v,
    \]
    where $\alpha \in \Delta$ is an arbitrary face (including possibly the empty face  $\varnothing$), viewed as a subset of the vertex set $V(\Delta)$, and $x_v$ is the standard generator associated with the vertex $v\in V(\Delta)$.
    
    The need for this identification explains the design of the ideal $I_\Delta$ in Definition~\ref{def:SR-ring-of-boolean}. For $\Delta$ a (true) simplicial complex, the relation \ref{item:empty-is-one} expresses that $x_\varnothing$ is sent to the empty product. The relation \ref{item:no-common-upper-bound} expresses that products of $x_v$'s are zero if they are not supported on a (common) face of $\Delta$. And because for $\Delta$ a (true) simplicial complex the set $\operatorname{lub}(\alpha,\beta)$ of least common upper bounds for faces $\alpha,\beta$ always has at most one element,  the relation \ref{item:common-upper-bound} reduces to the ``diamond relation"
    \[
    \prod_{v\in \alpha}x_v\prod_{v\in \beta}x_v = \prod_{v\in \alpha\wedge\beta}x_v\prod_{v\in \alpha\vee\beta}x_v
    \]
    when $\alpha$ and $\beta$ do belong to a common face of $\Delta$.
\end{setup}

\begin{setup}[ASL structure]\label{set:asl}  
    Note that, if $\alpha$ and $\beta$ are comparable in $\widehat P(\Delta)$, the corresponding generator \ref{item:common-upper-bound} of the Stanley--Reisner ideal reduces to zero. Hence pairwise products $x_\alpha x_\beta$ of the generators for $\kk[\Delta]$ only appear nontrivially as leading terms in the relations \ref{item:no-common-upper-bound}, \ref{item:common-upper-bound} when $\alpha, \beta$ are incomparable. In fact, $\kk[\Delta]$ is an {\em algebra with straightening law (ASL)} on the order dual of the face poset $P(\Delta)$ (with no minimal empty face attached, as $x_\varnothing$ has been identified with $1$ per the relation \ref{item:empty-is-one}). Furthermore, it is graded---see Setup~\ref{set:N-grading} below for the grading. It follows from general theory of graded ASLs that the monomials supported on chains (i.e., totally ordered subsets) in $ P(\Delta)$ actually form a $\kk$-basis, and the relations \ref{item:no-common-upper-bound}, \ref{item:common-upper-bound} allow to systematically write any monomial as a linear combination of monomials that are so supported. See \cite{eisenbud1980introduction, hodge-algebras} or \cite[Chapter~7]{bruns-herzog} for orientation to the theory of ASLs (also sometimes called {\em ordinal Hodge algebras}).

    A different theory of algebras with straightening law was developed by Kenneth Baclawski \cite{baclawski1981rings} at around the same time as De Concini, Eisenbud and Procesi's theory, and $\kk[\Delta]$ fits into Baclawski's framework as well (See Lemma~\ref{lem:filter-main-lemma} below). Although it is the De Concini--Eisenbud--Procesi definition that has been canonized, in some ways Baclawski's development more directly highlights the aspects of the situation that we will make use of. Throughout, when we need to refer to Baclawski's definition, to avoid ambiguity we will speak of a {\em ring with [lexicographic] straightening law in the sense of Baclawski}.
\end{setup}

\begin{definition}[Standard monomials]\label{def:standard-monomial}
    A monomial in the generators of an ASL that is supported on a chain of the underlying poset is called a {\em standard monomial}.  
\end{definition}

Thus, Setup~\ref{set:asl} can be summarized as saying that the standard monomials form a basis for $\kk[\Delta]$, and systematic application of the relations \ref{item:no-common-upper-bound}, \ref{item:common-upper-bound} allows any monomial in the $x_\alpha$'s, $\alpha \in P(\Delta)$, to be rewritten on this basis.

\begin{setup}[The barycentric subdivision; a $\kk$-linear isomorphism]\label{set:garsia}
    The barycentric subdivision $\Sd\Delta$ of a boolean complex $\Delta$ is a (true) simplicial complex, whose vertices are in bijection with the elements in the face poset $P(\Delta)$, and whose faces are in bijection with the chains in $P(\Delta)$. In other words, it is the order complex of the poset $P(\Delta)$. Thus the Stanley--Reisner ring $\kk[\Sd\Delta]$ is nothing but the Stanley--Reisner ring $\kk[P(\Delta)]$ of the poset $P(\Delta)$; it has generators $y_\alpha$, $\alpha\in P(\Delta)$, such that a monomial in the $y_\alpha$'s is nonzero if and only if it is supported on a chain in $P(\Delta)$. Thus $\kk[\Sd\Delta]$ has a $\kk$-basis consisting of monomials in the $y_\alpha$'s that are supported on chains (i.e., standard monomials). It follows, in view of Setup~\ref{set:asl}, that there is a $\kk$-linear isomorphism
    \[
    \Garsia:\kk[\Sd\Delta]\rightarrow \kk[\Delta]
    \]
    given by mapping
    \[
    y_\alpha \mapsto x_\alpha
    \]
    for each $\alpha \in P(\Delta)$, multiplicatively extending to standard monomials, and then linearly extending to all of $\kk[\Sd\Delta]$. We will have much more to say about this map in Section~\ref{sec:garsia}. We here note only that it is not a ring homomorphism: we have $y_\alpha y_\beta = 0$ in $\kk[\Sd\Delta]$ whenever $\alpha,\beta$ are incomparable in $P(\Delta)$ since in this case $y_\alpha y_\beta$ is not supported on a chain; but the corresponding $x_\alpha x_\beta \in \kk[\Delta]$ may be nonzero if $\alpha,\beta$ have a common upper bound, per relation \ref{item:common-upper-bound} of Definition~\ref{def:SR-ring-of-boolean}.
\end{setup}

\begin{example}\label{ex:y-ring-running-ex}
    In our running example from Figure~\ref{fig:running-ex}, we have
    \[
    \kk[\Sd\Delta]=\kk[y_u,y_v,y_\alpha,y_\beta]/(y_uy_v,y_\alpha y_\beta).
    \]
    Then the failure of $\Garsia$ to be a ring homomorphism is illustrated by the fact that $y_uy_v=0$ in $\kk[\Sd\Delta]$, so $\Garsia(y_uy_v)=0$, but on the other hand $\Garsia(y_u)\Garsia(y_v)=x_ux_v=x_\varnothing(x_\alpha+x_\beta) = x_\alpha+x_\beta \neq 0$ (see Example~\ref{ex:x-ring-running-ex}).
\end{example}

\begin{setup}[$G$-action]\label{set:garsia-is-equivariant}
    If $G\subseteq\Aut(\Delta)$ is a group of automorphisms of $\Delta$, then $G$ acts in the natural way on $P(\Delta)$, $\kk[\Delta]$, and $\kk[\Sd\Delta]$. We denote the action of $\sigma \in G$ on $f\in \kk[\Delta]$ or $\kk[\Sd\Delta]$, or $\alpha\in P(\Delta)$, by $\sigma\cdot f$, respectively $\sigma\cdot \alpha$. 
\end{setup}

\begin{setup}[The barycentric subdivision is the discrete ASL]\label{set:discrete-asl}
    Every ASL has a corresponding {\em discrete ASL}, in which the product of any pair of  generators corresponding to incomparable elements in the underlying poset is zero. (In other terms, the discrete ASL is the Stanley--Reisner ring of the [order complex of the] underlying poset.) Setup~\ref{set:garsia}, shows, in view of Setup~\ref{set:asl}, that $\kk[\Sd\Delta]$ is the discrete ASL associated with the ASL $\kk[\Delta]$.
\end{setup}

\begin{setup}[$\NN$-grading; ranked poset]\label{set:N-grading}
    The rings $\kk[\Delta]$ and $\kk[\Sd\Delta]$ carry natural $\NN$-gradings, with respect to which the $\kk$-linear isomorphism $\Garsia$ of Setup~\ref{set:garsia} is a graded map, as follows. 
    
    A poset is {\em ranked} if for each element $\alpha$, the lengths of all saturated chains connecting it to any minimal element are equal; the common length of these chains is denoted $\rk(\alpha)$. The poset $\widehat P(\Delta)$ is ranked: for any $\alpha\in \widehat P(\Delta)$, any saturated chain from $\varnothing$ to $\alpha$ in $\widehat P(\Delta)$ corresponds to a full flag in the closure of the cell $\alpha$ in the regular CW complex $\Delta$, so its length is determined by the dimension of $\alpha$; specifically, $\rk(\alpha) = \dim \alpha + 1$. Set
    \[
    \deg x_\alpha, \deg y_\alpha := \rk(\alpha).
    \]
    By convention, whenever we write $\rk(\alpha)$, the rank is computed in $\widehat P(\Delta)$ (not $P(\Delta)$).

    This assignment (together with the usual convention that the ground field lives in degree $0$) induces $\NN$-gradings on both $\kk[\Delta]$ and $\kk[\Sd\Delta]$. 
\end{setup}

\begin{example}
    In the running example from Figure~\ref{fig:running-ex}, we have 
    \[
    \deg x_v, \deg x_w, \deg y_v, \deg y_w = 1
    \]
    and
    \[
    \deg x_\alpha, \deg x_\beta, \deg y_\alpha, \deg y_\beta = 2.
    \]
    So, for example, 
    \[
    \deg x_w^2 x_\alpha^3 = \deg y_w^2 y_\alpha^3 = 2\cdot 1 + 3\cdot 2 = 8.
    \]
\end{example}

\begin{remark}
    The $\NN$-grading defined in Setup~\ref{set:N-grading} is the one that coincides with the standard grading in the case that $\Delta$ is a (true) simplicial complex (i.e., the grading assigning degree $1$ to each of the generators $x_v$, $v\in V(\Delta)$). However, if $\Delta$ is a boolean complex that is not isomorphic to a simplicial complex, it may not necessarily be a standard grading, i.e., $\kk[\Delta]$ may not necessarily be generated as a $\kk$-algebra by its degree-1 component. For instance, in our running example from Figure~\ref{fig:running-ex}, we have 
    $
    x_vx_w = x_\alpha + x_\beta 
    $
    as in Example~\ref{ex:y-ring-running-ex}. However, neither $x_\alpha$ nor $x_\beta$ is individually in the subalgebra $\kk[x_v,x_w]\subset \kk[\Delta]$ generated by the degree-1 elements.

    We also highlight that the $\NN$-grading on the barycentric subdivision ring $\kk[\Sd\Delta]$ given in Setup~\ref{set:N-grading} is different from the grading on $\kk[\Sd\Delta]$ obtained from the simplicial complex structure of $\Sd\Delta$ by assigning all the generators $y_\alpha$, $\alpha \in P(\Delta)=V(\Sd\Delta)$ the degree $1$. We will have no use for this latter grading in the present work.
\end{remark}

\begin{definition}[Balanced boolean complexes]\label{def:balanced}
    A simplicial complex, and more generally a boolean complex $\Delta$, is {\em balanced} if there is a labeling (aka coloring) of the vertex set by $\dim \Delta + 1$ labels (colors), so that all the vertices belonging to any one facet have distinct labels. A specific such labeling/coloring is a {\em balancing} of $\Delta$.
\end{definition}

\begin{remark}
    In the running example of Figure~\ref{fig:running-ex}, the dimension is $1$, so a balancing requires $2=\dim \Delta+1$ labels. It is achieved by labeling $v$ with one label and $w$ with another.
\end{remark}

\begin{setup}[Barycentric subdivision is balanced]\label{set:sd-balanced}
    For any boolean complex $\Delta$, the barycentric subdivision $\Sd\Delta$  is automatically balanced by the labeling that assigns to each vertex in $\Sd\Delta$ corresponding to the face $\alpha\in \Delta$ the label $\rk(\alpha) = \dim \alpha + 1$.
\end{setup}

\subsection{The parameter subring}\label{sec:parameter-subring}

Let $R$ be a $\kk$-algebra that is finitely generated, $\NN$-graded, and connected (i.e., $R_0=\kk$), and let  $\vartheta_1,\dots,\vartheta_n\in R$  be a homogeneous system of parameters. Then $\kk[\vartheta_1,\dots,\vartheta_n]$ is called a {\em parameter subring};  it is $\NN$-graded because the $\vartheta_j$'s are homogeneous, and with respect to it, $R$ is a graded module. Question~\ref{q:murai} is formulated with respect to a specific polynomial ring that occurs as a parameter subring in both $\kk[\Delta]$ and $\kk[\Sd\Delta]$. We introduce that ring here, following the notation in \cite{adams-reiner}. 

Let $d :=\dim \Delta$. Then the length of the poset $\widehat P(\Delta)$ is $n := d+1$. For $j=1,\dots,n$, define
\[
\theta_j := \sum_{\substack{\alpha\in P(\Delta) \\ \rk(\alpha)=j}} x_\alpha \in \kk[\Delta]
\]
and
\[
\gamma_j := \sum_{\substack{\alpha\in P(\Delta) \\ \rk(\alpha)=j}} y_\alpha \in \kk[\Sd\Delta].
\]
These are known as the {\em rank-row polynomials} \cite{garsia, garsia-stanton} or the {\em universal parameters} \cite{herzog2021systems}. The $\gamma$'s are also referred to in \cite{adams-reiner} as the {\em colorful parameters} because they are sums across the label classes (aka color classes) of the balancing of $\Sd\Delta$ described in Setup~\ref{set:sd-balanced}. As the names indicate, they form homogeneous (with respect to the gradings described in Setup~\ref{set:N-grading}) systems of parameters for $\kk[\Delta]$ and $\kk[\Sd\Delta]$ respectively \cite[Theorem~6.3]{hodge-algebras}. Therefore, the subrings
\[
\kk[\Theta]:=\kk[\theta_1,\dots,\theta_n]\subset \kk[\Delta]
\]
and
\[
\kk[\Gamma]:=\kk[\gamma_1,\dots,\gamma_n]\subset \kk[\Sd\Delta]
\]
are polynomial rings in the same number of indeterminates, thus isomorphic. 

\begin{notation}
Throughout the paper, we use $\Theta$ and $\Gamma$ as abbreviations for the sequences $\theta_1,\dots,\theta_n$ and $\gamma_1,\dots,\gamma_n$ respectively, in contexts where the latter are generating something. But we rely on context to communicate whether they are generating a $\kk$-algebra or an ideal. In Sections~\ref{sec:garsia-LA-CM} and \ref{sec:construct-a-basis}, we do the same with a ring $\kk[\Lambda]$ that generalizes $\kk[\Sd\Delta]$, and its homogeneous system of parameters $\Omega:=\omega_1,\dots,\omega_n$.
\end{notation}

We denote by $\Psi$ the graded $\kk$-algebra isomorphism
\[
\Psi: \kk[\Gamma]\rightarrow\kk[\Theta]
\]
that extends
\[
\gamma_j \mapsto \theta_j,\; j=1,\dots,n.
\]
Note that, while $\Psi$ coincides with the map $\Garsia$ of Setup~\ref{set:garsia} on $\gamma_1,\dots,\gamma_n$, it does not coincide with $\Garsia$ on all of $\kk[\Gamma]$; see Example~\ref{ex:parameters-under-garsia} below.

We use $\Psi$ (or rather, $\Psi^{-1}$) to view $\kk[\Sd\Delta]$ as a $\kk[\Theta]$-module for the sake of Question~\ref{q:murai} and Theorems~\ref{thm:no-equivariant-iso} and \ref{thm:yes-CM-in-coprime}. More precisely:

\begin{setup}[{$\kk[\Theta]$}-module structure]\label{set:theta-module}
    The $\kk[\Theta]$-module structure of $\kk[\Sd\Delta]$ is given by the composed map
\[
\kk[\Theta]\xrightarrow{\Psi^{-1}} \kk[\Gamma]\hookrightarrow \kk[\Sd\Delta].
\]
The $\kk[\Theta]$-module structure of $\kk[\Delta]$ is given by the canonical inclusion
\[
\kk[\Theta]\hookrightarrow \kk[\Delta].
\]
\end{setup}

\begin{remark}
We tend to think of $\kk[\Theta]$ and $\kk[\Gamma]$ as identified along $\Psi$, although we retain the notational distinction because the map $\Garsia$ will be important to us and is not equal to $\Psi$ when restricted to $\kk[\Gamma]$, as noted above.
\end{remark}

\subsection{Cohen--Macaulayness}\label{sec:Cohen--Macaulayness}

Theorem~\ref{thm:yes-CM-in-coprime} requires the hypothesis that the boolean complex $\Delta$ be Cohen--Macaulay over the field $\kk$, so we review some background information about Cohen--Macaulayness of rings and simplicial and boolean complexes. For a comprehensive treatment of the theory of Cohen--Macaulay rings, see \cite{bruns-herzog}. For more on Cohen--Macaulay complexes, see \cite[Chapter~4]{wachs2006poset}.

All the rings $R$ of relevance to this paper are finitely generated $\kk$-algebras that are $\NN$-graded and connected (i.e., $R_0=\kk$). For such rings, Cohen--Macaulayness takes a particularly clean form.

\begin{lemma}\label{lem:hironaka-decomposition}
Let $R$ be a finitely generated, $\NN$-graded, connected $\kk$-algebra. Then the following are equivalent:
\begin{enumerate}
\item $R$ is Cohen--Macaulay.\label{item:CM}
\item There exists a homogeneous system of parameters $\vartheta_1,\dots,\vartheta_n\in R$ such that $R$ is a free $\kk[\vartheta_1,\dots,\vartheta_n]$ module.\label{item:free-over-one-parameter-ring}
\item For any homogeneous system of parameters $\vartheta_1,\dots,\vartheta_n\in R$, $R$ is a free $\kk[\vartheta_1,\dots,\vartheta_n]$-module.\label{item:free-over-any-parameter-ring}
\item For any homogeneous system of parameters $\vartheta_1,\dots,\vartheta_n\in R$, any homogeneous $\kk$-vector space basis of the quotient $R/(\vartheta_1,\dots,\vartheta_n)R$ lifts to a $\kk[\vartheta_1,\dots,\vartheta_n]$-module basis of $R$.\label{item:any-basis-lifts}
\end{enumerate}
\end{lemma}

\begin{remark}
    A given expression of a Cohen--Macaulay $\NN$-graded $\kk$-algebra as a direct sum
    \[
    R = \bigoplus_{i=1}^m \eta_i \kk[\vartheta_1,\dots,\vartheta_n],
    \]
    where $\vartheta_1,\dots,\vartheta_n$ is a homogeneous system of parameters, and $\eta_1,\dots,\eta_m$ is a homogeneous $\kk[\vartheta_1,\dots,\vartheta_n]$-module basis of $R$, is called a {\em Hironaka decomposition} of $R$.
\end{remark}

\begin{proof}[Proof of Lemma~\ref{lem:hironaka-decomposition}]
The equivalence of conditions \ref{item:CM}, \ref{item:free-over-one-parameter-ring}, and \ref{item:free-over-any-parameter-ring} is a special case of \cite[Theorem~4.3.5]{benson}.  Condition~\ref{item:any-basis-lifts} implies condition~\ref{item:free-over-any-parameter-ring} because any graded $\kk$-vector space has a homogeneous basis, and in particular, $R/(\vartheta_1,\dots,\vartheta_n)R$ has one. Condition~\ref{item:free-over-any-parameter-ring} implies condition~\ref{item:any-basis-lifts} by the following standard argument. A homogeneous $\kk$-basis for $R/(\vartheta_1,\dots,\vartheta_n)R$ lifts to a $\kk[\vartheta_1,\dots,\vartheta_n]$-module generating set for $R/(\vartheta_1,\dots,\vartheta_n)R$ by the graded Nakayama lemma. Meanwhile, because $R$ is assumed free as a $\kk[\vartheta_1,\dots,\vartheta_n]$-module (by condition~\ref{item:free-over-any-parameter-ring}), and is of finite rank, say rank $r$ (since $\vartheta_1,\dots,\vartheta_n$ is a system of parameters for $R$), it follows that $R/(\vartheta_1,\dots,\vartheta_n)R$ has dimension $r$ as a $\kk$-vector space. Thus the $\kk[\vartheta_1,\dots,\vartheta_n]$-module map
\[
\kk[\vartheta_1,\dots,\vartheta_n]^r \rightarrow R
\]
that sends a basis for $\kk[\vartheta_1,\dots,\vartheta_n]^r$ to the lifts in $R$ of the given homogeneous $\kk$-basis for $R/(\vartheta_1,\dots,\vartheta_n)R$ is a surjection between isomorphic finitely generated modules. By Vasconcelos' Theorem \cite[Proposition~1.2]{vasconcelos1969finitely}, it is an isomorphism. Thus the lifts in fact form a $\kk[\vartheta_1,\dots,\vartheta_n]$-basis for $R$, confirming condition~\ref{item:any-basis-lifts}.
\end{proof}

For later use, we record a (well-known) relaxation of condition~\ref{item:any-basis-lifts} of Lemma~\ref{lem:hironaka-decomposition} that holds even when $R$ is not Cohen--Macaulay:

\begin{lemma}\label{lem:minimal-module-generating-set}
    Let $R$ be a finitely generated, $\NN$-graded, and connected, but not necessarily Cohen--Macaulay, $\kk$-algebra. Let $\vartheta_1,\dots,\vartheta_n$ be a homogeneous system of parameters. Then a set of homogeneous elements $\eta_1,\dots,\eta_m\in R$ forms a minimal $\kk[\vartheta_1,\dots,\vartheta_n]$-module generating set if and only if the images of $\eta_1,\dots,\eta_m$ in $R/(\vartheta_1,\dots,\vartheta_n)R$ constitute a $\kk$-basis.
\end{lemma} 

\begin{proof}
    Viewing $\kk$ as the $\kk[\vartheta_1,\dots,\vartheta_n]$-module $\kk[\vartheta_1,\dots,\vartheta_n]/(\vartheta_1,\dots,\vartheta_n)$, right-exactness of the tensor product by $\kk$ over $\kk[\vartheta_1,\dots,\vartheta_n]$ implies that if $\eta_1,\dots,\eta_m$ generate $R$ over $\kk[\vartheta_1,\dots,\vartheta_n]$, then their images in $R/(\vartheta_1,\dots,\vartheta_n)R$ span it over $\kk$. The converse statement is supplied by the graded Nakayama lemma (as in Lemma~\ref{lem:hironaka-decomposition}). Thus for $\eta_1,\dots,\eta_m$, generation of $R$ over $\kk[\vartheta_1,\dots,\vartheta_n]$ is equivalent to generation of $R/(\theta_1,\dots,\theta_n)R$ over $\kk$. It follows that minimality with respect to generation is also equivalent.
\end{proof}

A simplicial or boolean complex $\Delta$ is said to be Cohen--Macaulay over a field $\kk$ if its Stanley--Reisner ring $\kk[\Delta]$ is Cohen--Macaulay. Foundational work of G. A. Reisner \cite{reisner}, sharpened by J. Munkres \cite{munkres} and generalized to boolean complexes by A. Duval \cite{duval1997free}, shows that Cohen--Macaulayness of $\Delta$ depends only on the homeomorphism class of the total space $|\Delta|$, and in fact, there is a beautiful characterization in terms of the homology of $|\Delta|$, see  \cite[Corollary~5.4.6]{bruns-herzog}.

In particular, if $\Delta$ is Cohen--Macaulay over a given field $\kk$, then so is $\Sd \Delta$ (and vice versa). An $\NN$-graded free module over an $\NN$-graded ring has a basis homogeneous with respect to this grading, so in this situation, both $\kk[\Delta]$ and $\kk[\Sd\Delta]$ have homogeneous bases over $\kk[\Theta]$. But actually, more is true. In Section~\ref{sec:grading-by-shape}, we will show that $\kk[\Sd\Delta]$ has a grading over the monoid $\Part_n$ of partitions with at most $n$ parts, where $n=\dim \Delta + 1$, that refines the $\NN$-grading. This will be called the {\em shape grading}. It follows from general theory (but also from the explicit construction in Section~\ref{sec:construct-a-basis}) that when $\Delta$ is Cohen--Macaulay, $\kk[\Sd\Delta]$ has a module basis over $\kk[\Theta]$ that is homogeneous with respect to the shape grading. In Section~\ref{sec:filtering-garsia}, it will be shown that $\kk[\Delta]$ is (not graded but) filtered over $\Part_n$, in an appropriate sense, and then Theorem~\ref{thm:transfer-bases} will show that shape-homogeneous bases for $\kk[\Sd\Delta]$ over $\kk[\Theta]$ can be used to build bases for $\kk[\Delta]$ over $\kk[\Theta]$. These will then be used in Section~\ref{sec:construction-mod-basis} to construct a non-equivariant $\kk[\Theta]$-module isomorphism between $\kk[\Sd\Delta]$ and $\kk[\Delta]$ that can be deformed into an equivariant isomorphism.

\section{Grading and filtering by shape; the Garsia transfer map}\label{sec:garsia}

In this section, we adapt to the present setting a tool originally introduced by Garsia \cite{garsia} for transferring bases from Stanley--Reisner rings to {\em partition rings}, which are certain combinatorially-defined subrings of polynomial rings. It was later adapted in various directions by Baclawski, Garsia, and Stanton \cite{baclawski1981rings, baclawski1981combinatorial, garsia-stanton}. The present adaptation fits into the framework developed by Baclawski \cite{baclawski1981rings}, and where possible we take advantage of that paper's results to expedite our exposition. However, in order to draw out what we view as the underlying conceptual picture that the method is exploiting---namely, that $\kk[\Delta]$ is {\em filtered by shape}, to be explained below, and $\kk[\Sd \Delta]$ is the associated graded ring---our presentation is structured differently than in these sources, and is mostly self-contained.

We use this tool in Section~\ref{sec:positive} to prove that an appropriately chosen $\kk[\Theta]$-module isomorphism $\kk[\Sd\Delta]\rightarrow \kk[\Delta]$, if it exists, can be made equivariant by averaging. In the Cohen--Macaulay case, it is also used to construct this non-equivariant isomorphism in the first place.

In the context that the ring playing the role of $\kk[\Delta]$ in the theory is an invariant ring of a permutation group acting on a polynomial ring (the setting of \cite{garsia-stanton}), the Garsia transfer method has been used previously in \cite{reiner, hersh} to find bases for certain rings of polynomial permutation invariants, and in \cite{pevzner2024symmetric} to study the module structure of the {\em fixed quotients} of the action of the symmetric group $\mathfrak{S}_n$ on $\kk[\Delta_d]$ and $\kk[\Sd\Delta_d]$ (where $\Delta_d$ is the $d$-simplex, and $n=d+1$).\footnote{Following \cite{pevzner2024symmetric}, given a ring $R$ with an action by a group $G$, the {\em fixed quotient} $R_G$ is the $R^G$-module universal with respect to receiving a $G$-invariant map from $R$. It can be constructed as the quotient of $R$ by the $R^G$-module generated by elements $gr - r$ with $g\in G$ and $r\in R$. It is sometimes called the {\em module of coinvariants}, but the name in \cite{pevzner2024symmetric} avoids  confusion, in the main case when $R$ is a graded, connected algebra and the action by $G$ is graded, with the {\em coinvariant algebra} $R/R_+^GR$; here $R_+^GR$ is the {\em Hilbert ideal}, i.e., the ideal in $R$ generated by the homogeneous invariants of positive degree. Unfortunately, the coinvariant algebra is also often denoted $R_G$.\label{note:fixed-quotient-coinvariant}} It was also used in \cite{huang20130} to analogize the actions of the $0$-Hecke algebra of $\mathfrak{S}_n$ on $\kk[\Delta_d]$ and on $\kk[\Sd\Delta_d]$, and in \cite{thiery2000algebraic} to speed up computation of certain rings of permutation invariants associated to the study of graphs. The method is exposited carefully in \cite[Section~2.5.3]{blum-smith} and \cite[Sections~5.2 and 5.3]{pevzner2024symmetric}.

Garsia's method may be viewed from a certain point of view as a way of formalizing and generalizing the insight behind the classical lexicographic proof of the Fundamental Theorem on Symmetric Polynomials, due to C. F. Gauss \cite[Paragraphs~3--5]{gauss}. This proof was generalized in \cite{gobel} and \cite{reiner1995gobel} to study invariant rings, respectively, of permutation groups acting on  polynomial algebras, and subgroups of Weyl groups acting on the group algebras of weight lattices. The latter generalization was inspired explicitly by \cite{garsia-stanton}. The connection between the Gauss proof and the method of Garsia is made explicit in \cite[Section~2.5.3.3]{blum-smith}.

Essentially the same method (in particular, the underlying filtration of $\kk[\Delta_d]$ drawn out below in Section~\ref{sec:filtering-garsia}) was used in \cite[Section~3]{adin2005descent} to realize the {\em descent representations} (see \cite{solomon1968decomposition}) of the symmetric group $\mathfrak{S}_n$ as subquotients of the coinvariant algebra of $\mathfrak{S}_n$'s canonical action on $\kk[\Delta_d]$. (The descent representations were also realized in \cite[Section~2]{garsia-stanton}, as homogeneous components of the coinvariant algebra of the $\mathfrak{S}_n$ action on $\kk[\Sd\Delta_d]$. The definition of the coinvariant algebra is recalled in  footnote~\ref{note:fixed-quotient-coinvariant}.) 

\subsection{Grading $\kk[\Sd\Delta]$ by shape}\label{sec:grading-by-shape}
The $\NN$-grading on $\kk[\Sd\Delta]$ described in Setup~\ref{set:N-grading} can be refined into an $\NN^n$-grading, where $n=\dim \Delta +1$ (as in Section~\ref{sec:parameter-subring} above), by assigning
\[
\deg_\mathrm{md} y_\alpha := \bfe_{\rk(\alpha)},
\]
where $\alpha \in P(\Delta)$, the rank $\rk(\alpha)$ is calculated in $\widehat P(\Delta)$ as in Setup~\ref{set:N-grading}, and $\bfe_1,\dots,\bfe_n$ form the standard basis for $\NN^n$. The degree of a monomial with respect to this $\NN^n$-grading is called a {\em multidegree}, to distinguish it from the  $\NN$-grading defined in Setup~\ref{set:N-grading}, and the subscript $\mathrm{md}$ (for {\em multidegree}) is used accordingly.

The fact that this assignment induces a grading on $\kk[\Sd\Delta]$ follows, as in Setup~\ref{set:N-grading}, from the fact that $\kk[\Sd\Delta]$ is the quotient of a polynomial ring in the $y_\alpha$'s by a monomial ideal. One recovers the $\NN$-grading of Setup~\ref{set:N-grading} from the present $\NN^n$-grading via the monoid map
\[
\NN^n\rightarrow \NN
\]
extended from
\[
\bfe_j \mapsto j
\]
(for $j=1,\dots,n$). Note that throughout, whenever we refer to $\NN^n$ as a monoid, we mean with respect to the addition structure.

\begin{example}\label{ex:y-ring-multidegree}
    In our running example from Figure~\ref{fig:running-ex}, we have
    \[
    \deg_\mathrm{md} y_v, \deg_\mathrm{md} y_w = \bfe_1
    \]
    and
    \[
    \deg_\mathrm{md} y_\alpha, \deg_\mathrm{md} y_\beta = \bfe_2.
    \]
    So
    \[
    \deg_\mathrm{md} y_w^2y_\alpha^3 = 2\bfe_1 + 3\bfe_2.
    \]
\end{example}

We now change points of view on the grading monoid $\NN^n$. Let  $\Part_n$ be the set of partitions with at most $n$ parts. We make $\Part_n$ into a monoid by adding partitions part-by-part, zero-padding the shorter one if the numbers of parts are different; for example, $(3,2,1) + (5,5,5,3) = (8,7,6,3)$. Then, we  assign to each multidegree $\bfa\in \NN^n$ a partition $\lambda \in \Part_n$, its {\em shape}. The assignment will yield an isomorphism of monoids; our goal is to view $\kk[\Sd\Delta]$ as graded by $\Part_n$. Algebraically, this changes nothing, but it will make considerations of order structure, which become important in the following section on filtering $\kk[\Delta]$, more transparent.

\begin{notation}
    If a partition contains a part more than once, this can be indicated with an exponent. Thus, $(5,5,5,3)=(5^3,3)$, for example.
\end{notation}

\begin{lemma}\label{lem:monoid-iso}
    The monoid map
    \[
    \operatorname{sh}:\NN^n \rightarrow \Part_n
    \]
    given on the free commuting generators $\bfe_1,\dots,\bfe_n$ by
    \[
    \bfe_j \mapsto (1^j)
    \]
    is an isomorphism of monoids.
\end{lemma}

\begin{proof}
The inverse map is
\[
(\lambda_1,\dots,\lambda_n) \mapsto (\lambda_1 - \lambda_2)\bfe_1 + \dots + (\lambda_{n-1}-\lambda_n)\bfe_{n-1} + \lambda_n \bfe_n,
\]
where $(\lambda_1,\dots,\lambda_n)$ is an arbitrary partition with at most $n$ parts (allowing some of the $\lambda_j$'s to be zero).
\end{proof}

\begin{definition}[Shape in {$\kk[\Sd\Delta]$}]\label{def:shape}
    For a monomial $m\in \kk[\Sd\Delta]$, define
    \[
    \shape(m) = \operatorname{sh}(\deg_\mathrm{md} m).
    \]
    Since $\operatorname{sh}$ is a monoid isomorphism by Lemma~\ref{lem:monoid-iso}, and $\deg_\mathrm{md}$ is an $\NN^n$-grading as discussed at the beginning of the section, this assignment gives $\kk[\Sd\Delta]$ the structure of a $\Part_n$-graded $\kk$-algebra. Given a partition $\lambda \in \Part_n$, we denote by $\kk[\Sd\Delta]_\lambda$ the $\kk$-subspace of $\kk[\Sd\Delta]$ spanned by monomials of shape $\lambda$.
\end{definition}

Note that the shape $\shape(m)$ of a monomial is a partition of its degree $\deg(m)$ as defined in Setup~\ref{set:N-grading}. Also, because shape is defined in terms of ranks in $\widehat{P}(\Delta)$, and automorphisms of $\Delta$ preserve these ranks, they also preserve shape. I.e., if $\sigma \in \Aut(\Delta)$ is an automorphism, and $m\in \kk[\Sd\Delta]$ is a standard monomial, then
\begin{equation}\label{eq:autos-preserve-shape}
\shape(m) = \shape(\sigma\cdot m).
\end{equation}

\begin{example}
    In the situation of Example~\ref{ex:y-ring-multidegree}, the monomial $y_w^2y_\alpha^3$ has shape $(5,3) = 2\cdot(1) + 3\cdot(1,1)$.
\end{example}

\begin{remark}
    In \cite{garsia, garsia-stanton, pevzner2024symmetric}, it is the conjugate partition to the one given in Definition~\ref{def:shape} that is called the shape. The present convention follows \cite{blum-smith} (and, implicitly, \cite{reiner2}) and is motivated by the fact that when one uses the corresponding definition of shape in $\kk[\Delta]$ (as we will below), and $\Delta$ is a simplex, so that $\kk[\Delta]$ may be viewed as a standard-graded polynomial ring, then the shape as defined here coincides with the usual notion of the shape of a monomial in a standard-graded polynomial ring, i.e., its exponent vector taken in nonincreasing order. This will be illustrated below in Example~\ref{ex:shape-motivation}.
\end{remark}

We have
\[
\kk[\Sd\Delta] = \bigoplus_{\lambda \in \Part_n} \kk[\Sd\Delta]_\lambda,
\]
the decomposition of $\kk[\Sd\Delta]$ into the homogeneous components of the grading of Definition~\ref{def:shape}. Because the parameters $\gamma_1,\dots,\gamma_n$ are themselves homogeneous, the subring $\kk[\Gamma]$ is similarly $\Part_n$-graded. Furthermore, monomials in $\gamma_1,\dots,\gamma_n$ have a particularly nice description in terms of this grading:

\begin{prop}\label{prop:parameter-monom-every-term}
    For any natural numbers $a_1,\dots,a_n$, the expansion of the product
    \[
    \gamma_1^{a_1}\cdots \gamma_n^{a_n}
    \]
    on the basis of standard monomials for $\kk[\Sd\Delta]$ consists precisely of the sum of all standard monomials of shape
    \[
    a_1(1^1) + \dots + a_n(1^n) \in \Part_n.
    \]
\end{prop}

\begin{proof}
    Because of the definition of multiplication in $\kk[\Sd\Delta]$, every term in the expansion of $\gamma_1^{a_1}\cdots \gamma_n^{a_n}$  either is supported on a chain or computes to zero.  Because each $\gamma_j$ is the sum of $y_\alpha$ for {\em every} $\alpha\in P(\Delta)$ of a fixed rank, the nonzero terms in the expansion of $\gamma_1^{a_1}\cdots \gamma_n^{a_n}$ biject with the set of all multichains (i.e., multisets supported on chains) in $P(\Delta)$ in which the multiplicity of the element of rank $j$ (with rank computed in $\widehat P(\Delta)$) is $a_j$. Thus the nonzero terms are exactly the set of standard monomials of multidegree $a_1\bfe_1 + \dots + a_n\bfe_n$. By Lemma~\ref{lem:monoid-iso} and Definition~\ref{def:shape}, this is equivalently the set of standard monomials of shape
    \[
    a_1(1^1) + \dots + a_n(1^n),
    \]
    as claimed.
\end{proof}

\begin{example}
    To illustrate, we compute $\gamma_1^2\gamma_2$ for our running example from Figure~\ref{fig:running-ex}:
    \begin{align*}
        \gamma_1^2\gamma_2 &= (y_v+y_w)^2(y_\alpha + y_\beta)\\
        &=(y_v^2 + y_w^2)(y_\alpha + y_\beta)\\
        &=y_v^2y_\alpha + y_v^2y_\beta + y_w^2y_\alpha + y_w^2y_\beta,
    \end{align*}
    the sum of all standard monomials of shape $2(1)+(1,1)=(3,1)$.
\end{example}

\subsection{Filtering {$\kk[\Delta]$} by shape}\label{sec:filtering-garsia}

The attempt to copy Section~\ref{sec:grading-by-shape} for $\kk[\Delta]$, hoping to induce a multigrading via the assignment
\[
\deg_\mathrm{md} x_\alpha := \bfe_{\rk(\alpha)},
\]
does not succeed. The defining relations for $\kk[\Delta]$ given in Definition~\ref{def:SR-ring-of-boolean} are not homogeneous with respect to this assignment, as illustrated by the following calculation.

\begin{example}
    We return to our running example from Figure~\ref{fig:running-ex}. It was computed above that
    \[
    x_vx_w = x_\alpha + x_\beta
    \]
    in $\kk[\Delta]$. But the assignment above gives $\deg_\mathrm{md} x_v, \deg_\mathrm{md} x_w = \bfe_1$, and $\deg_\mathrm{md} (x_\alpha + x_\beta) = \bfe_2$, while 
    \[
    \bfe_1 + \bfe_1 \neq \bfe_2
    \]
    in $\NN^n$.
\end{example}

On the other hand, this same calculation illustrates the virtue of thinking of the grading monoid in the previous section as $\Part_n$ rather than $\NN^n$, as we now illustrate (and as this section will explain thoroughly). The shapes (obtained via the isomorphism in Lemma~\ref{lem:monoid-iso}) are
\[
\shape(x_v) = \shape(x_w) = (1)
\]
and
\[
\shape(x_\alpha + x_\beta) = (1,1).
\]
We have $(1) + (1) = (2)$ in $\Part_n$. This is unequal to $(1,1)$ of course. However, $(2)$ is above $(1,1)$ in the {\em dominance order} on partitions (whose definition is recalled below), and this manifests a general phenomenon which will allow us to view $\kk[\Delta]$ as {\em filtered} by shape. To state the precise result, we begin by defining shape on the polynomial ring over $P(\Delta)$ that surjects onto $\kk[\Delta]$, where it unproblematically gives a grading. We recall the notations $S$, $I_\Delta$ of Definition~\ref{def:SR-ring-of-boolean}, and work in the polynomial ring 
\[
\overline S := S/(x_\varnothing - 1) \cong \kk[\{x_\alpha\}_{\alpha\in  P(\Delta)}],
\]
with polynomial generators corresponding to the ASL generators of $\kk[\Delta]$. Let $\overline I_\Delta$ be the image of $I_\Delta$ in $\overline S$; it is the ideal generated by the (images in $\overline S$ of the) elements (of $I_\Delta$) of the form \ref{item:no-common-upper-bound} and \ref{item:common-upper-bound} of Definition~\ref{def:SR-ring-of-boolean}. Note that 
\[
\kk[\Delta] = \overline S / \overline I_\Delta.
\]
We continue to abuse notation by using the same symbols $x_\alpha$ for the generators of all three of the rings $S$, $\overline S$, and $\kk[\Delta]$, and we indicate to the reader via the context which ring is meant.

Because $\overline S$ is a polynomial ring, we can grade it over any commutative monoid by specifying degrees for the generators; we do so in parallel with Definition~\ref{def:shape}.

\begin{definition}[Shape in $\overline S$]\label{def:x-shape-upstairs}
    For $\alpha \in P(\Delta)$ and $x_\alpha \in \overline S$, define
    \[
    \shape(x_\alpha) := (1^{\rk(\alpha)})\in \Part_n,
    \]
    where $\rk(\alpha)$ is calculated in $\widehat P(\Delta)$ as in Setup~\ref{set:N-grading}. This determines a grading of $\overline S$ over $\Part_n$. 
\end{definition}

\begin{remark}
It is convenient to state Definition~\ref{def:x-shape-upstairs} in terms of $\overline S$ because we want to work in the polynomial ring whose generators correspond to the ASL generators of $\kk[\Delta]$. However, the same definition also works with $S$ in place of $\overline S$ and $\widehat P(\Delta)$ in place of $P(\Delta)$, if (as is natural) we interpret $(1^0)$ as the empty partition, i.e., the identity of the monoid $\Part_n$. Indeed, the ideal $(x_\varnothing - 1)$ by which we pass from $S$ to $\overline S$ is then homogeneous with respect to this grading. Thus, even with respect to the $\Part_n$ grading, we can think of $\overline S$ as a ring in which $x_\varnothing$ is another name for $1$. This is useful in the proof of Lemma~\ref{lem:filter-main-lemma} below.
\end{remark}

\begin{definition}[Lifting to $\overline S$; shape in {$\kk[\Delta]$}]\label{def:x-shape}
As discussed in Setup~\ref{set:asl}, $\kk[\Delta]$ has a $\kk$-basis consisting of standard monomials (Definition~\ref{def:standard-monomial}). A standard monomial
\[
m = \prod_{\alpha\in C} x_\alpha^{c_\alpha},
\]
where $C\subset P(\Delta)$ is a chain, can be viewed as an element either of $\overline S$ or of $\kk[\Delta]$---we refer to the former interpretation as the {\em lift} of the latter interpretation. Although the map $\overline S\rightarrow \kk[\Delta]$ is not bijective, the lift of a standard monomial in $\kk[\Delta]$ {\em is} its unique preimage in $\overline S$ that is a standard monomial; thus our use of the definite article in this definition is justified.

Now, define the {\em shape} on the standard monomials $m \in \kk[\Delta]$, written $\shape(m)$, by applying Definition~\ref{def:x-shape-upstairs} to their lifts in $\overline S$. For $\lambda \in \Part_n$, we denote by $\kk[\Delta]_\lambda$ the $\kk$-subspace of $\kk[\Delta]$ spanned by standard monomials of shape $\lambda$. Elements of $\kk[\Delta]_\lambda$ are {\em homogeneous of shape} $\lambda$.
\end{definition}

With these definitions, we have
\[
\kk[\Delta] = \bigoplus_{\lambda\in \Part_n} \kk[\Delta]_\lambda,
\]
as $\kk$-vector spaces. Also, as with Definition~\ref{def:shape}, if $m \in \kk[\Delta]$ is a standard monomial, then $\shape(m)$ is a partition of $\deg(m)$, and if furthermore $\sigma\in \Aut(\Delta)$ is an automorphism, then we have
\[
\shape(m) = \shape(\sigma\cdot m).
\]

The following example illustrates the motivation for the definition of shape.

\begin{example}\label{ex:shape-motivation}
Let $\Delta:= \Delta_2$ be the $2$-simplex on the vertex set $\{0,1,2\}$ (and $n=2+1=3$). According to Definition~\ref{def:SR-ring-of-boolean}, $\kk[\Delta]$ has  generators $x_\varnothing, x_0,x_1,x_2,x_{12},x_{02},x_{01},x_{012}$, with relations such as $x_\varnothing = 1$, $x_1x_{02} = x_{012}$, $x_1x_2=x_{12}$, $x_{01}x_{02}=x_0x_{012}$, etc. Thus it is actually just the polynomial ring on $x_0,x_1,x_2$, illustrating the identification described in Setup~\ref{set:boolean-simplicial-identification}. Let us refer to the generators $x_\alpha$, $\alpha \in P(\Delta)$ (note that this excludes $x_\varnothing$) as the {\em ASL generators}, and the subset $x_0,x_1, x_2$ as the {\em polynomial generators}. Then, as a graded ASL, $\kk[\Delta]$ has a $\kk$-basis consisting of the standard monomials in the ASL generators; while as a polynomial ring, it has a $\kk$-basis consisting of {\em all} monomials in the polynomial generators. But in fact, these are the same basis. For example, the standard monomial
\[
x_1^2x_{12}^3 x_{012}
\]
supported on the chain
\[
\{1\} \subset\{1,2\}\subset \{0,1,2\}
\]
is equal to the monomial
\[
x_0 x_1^6x_2^4
\]
in the polynomial generators, by routine application of the defining relations to the latter.\footnote{P. Mantero \cite{mantero2020structure} refers to the representation of a monomial in the polynomial generators as a standard monomial in the ASL generators, as the {\em normal form} of the monomial, and generalizes this notion to monomials in an arbitrary set of linear forms, see \cite[Section~3]{mantero2020structure}.} Now, applying Definition~\ref{def:x-shape}, we get
\begin{align*}
\shape(x_1^2x_{12}^3x_{012}) &= 2\cdot(1) + 3\cdot(1,1) + (1,1,1)\\
&= (6,4,1).
\end{align*}
Note that $(6,4,1)$ is the shape (in the usual sense of the exponent vector taken in nonincreasing order) of this monomial when written in terms of the polynomial generators, i.e., as $x_0x_1^6x_2^4$.
\end{example}

We name the condition under which the product of two standard monomials is itself already a standard monomial without needing to apply any straightening relations. This definition makes sense in any ASL, although we have in mind $\kk[\Delta]$ and $\kk[\Sd\Delta]$  (which are  graded ASLs over $P(\Delta)$).

\begin{definition}[Stacking up]\label{def:stack-up}
    Following \cite{blum-smith}, if two standard monomials $m_1,m_2$ in  an ASL are supported on the same chain in the underlying poset (i.e., there exists a chain supporting both of them), we say that they {\em stack up}.
\end{definition}

\begin{observation}\label{obs:stack-up-nonzero}
    In $\kk[\Sd\Delta]$, or any discrete ASL, two standard monomials $m_1$, $m_2$ stack up if and only if their product $m_1m_2$ is nonzero.
\end{observation}

Let $\lambda = (\lambda_1,\lambda_2,\dots,\lambda_k)$ and $\mu = (\mu_1,\mu_2,\dots,\mu_\ell)$ be two partitions of the same natural number $d$. It is said that $\lambda$ {\em dominates} $\mu$, written $\lambda \trianglerighteq \mu$ or $\mu \trianglelefteq \lambda$, if for each $j=1,\dots,\min(k,\ell)$, we have
\[
\lambda_1 + \dots + \lambda_j \geq \mu_1 + \dots + \mu_j.
\]
Dominance order is a partial order on partitions.  
For strict dominance (i.e. dominance between unequal partitions), we write $\mu\triangleleft \lambda$.

\begin{observation}\label{obs:dominance-monoid-compatible}
The dominance partial order on $\Part_n$ is compatible with the monoid structure; i.e., if $\lambda,\mu,\nu\in\Part_n$ and $\mu \trianglelefteq \lambda$, then also $\mu + \nu \trianglelefteq \lambda+\nu$.
\end{observation}

\begin{observation}\label{obs:dominance-induction}
    For a fixed natural number $d$, the poset of partitions of $d$ with respect to dominance order is finite. Thus, in particular, it satisfies the ascending and descending chain conditions, so any nonempty subset has maximal elements, and also one can do induction over it.
\end{observation}

Taken together, Observations~\ref{obs:dominance-induction} and \ref{obs:dominance-monoid-compatible} amount to the statement that dominance order is an {\em admissible order} in the sense of Baclawski's theory of straightening laws \cite[p.~192]{baclawski1981rings} (indeed, it is one of the examples motivating this definition in \cite{baclawski1981rings}). With respect to this order, $\kk[\Delta]$ becomes a {\em ring with lexicographic straightening law} in the sense of Baclawski:

\begin{lemma}[Key lemma for filtering {$\kk[\Delta]$}]\label{lem:filter-main-lemma}
Let $\Delta$ be a boolean complex. Let $m_1,m_2$ be standard monomials in $\kk[\Delta]$. Represent the product $m_1m_2\in \kk[\Delta]$ on the basis of standard monomials. Then:
\begin{enumerate}
    \item the shape of each standard monomial in this representation is dominated by $\shape(m_1)+\shape(m_2)$; and\label{item:shape-dominated}
    \item this domination is strict unless $m_1,m_2$ stack up, in which case the representation consists of a single monomial whose shape is equal to $\shape(m_1)+\shape(m_2)$.
\end{enumerate}
In other words, the straightening law on $\kk[\Delta]$ is a {\em lexicographic straightening law based on $P(\Delta)$} in the sense of Baclawski \cite[Definition~4.2]{baclawski1981rings}.
\end{lemma}

It is possible that $m_1m_2 = 0$. Note that in this case, the conclusion of the lemma holds vacuously.

\begin{proof}[Proof of Lemma~\ref{lem:filter-main-lemma}]
If $m_1$ and $m_2$ stack up, then $m_1,m_2\in \kk[\Delta]$ are supported on the same chain of $P(\Delta)$. So the product $m_1m_2 \in \kk[\Delta]$ can be computed by taking lifts of $m_1$ and $m_2$ in $\overline S$, computing the product there, and then interpreting it as an element of $\kk[\Delta]$ (because it is standard). In this case, $m_1m_2$ is represented by this single monomial, and we have
\[
\shape(m_1m_2) = \shape(m_1) + \shape(m_2)
\]
because $\overline S$ is graded by shape.

If $m_1$, $m_2$ do not stack up, then $m_1,m_2\in\kk[\Delta]$ are not supported on the same chain. If we lift them to $\overline S$, the product $m_1m_2$ is not standard. Our goal is to show that replacing $m_1m_2$ with the linear combination of standard monomials that represents the same element of $\kk[\Delta]$  strictly lowers the shape with respect to dominance order. This will be done inductively, showing that each application of one of the straightening laws strictly lowers the shape.

By the general theory of graded ASLs (specifically \cite[Theorem~3.4]{eisenbud1980introduction},  \cite[Proposition~1.1]{hodge-algebras}), the representation of $m_1m_2\in \kk[\Delta]$ in terms of standard monomials can be obtained by a number of applications of the straightening laws \ref{item:no-common-upper-bound}, \ref{item:common-upper-bound} of Definition~\ref{def:SR-ring-of-boolean}. In other words, lifting $m_1,m_2$ to $\overline S$, the product $m_1m_2\in \overline S$ can be replaced with a linear combination of standard monomials belonging to the same coset of $\overline I_\Delta$ via a sequence of moves, each of which consists of replacing a product $x_\alpha x_\beta$ (with $\alpha,\beta$ incomparable in $P(\Delta)$) that appears in $m_1m_2$ with $0$ if $\alpha,\beta\in P(\Delta)$ lack a common upper bound in $P(\Delta)$, per straightening law~\ref{item:no-common-upper-bound}, or with $x_{\alpha\wedge \beta}\sum_{\gamma\in\operatorname{lub}(\alpha,\beta)} x_\gamma$  if they do have a common upper bound, per straightening law~\ref{item:common-upper-bound}. 

(In the latter formula, $x_{\alpha\wedge \beta}$ should be interpreted as $1$ if $\alpha$ and $\beta$ have no common lower bound in $P(\Delta)$. This results from relation \ref{item:empty-is-one} in Definition~\ref{def:SR-ring-of-boolean}, which is modded out in $\overline S$. This does not create a special case in the below argument because of the remark following Definition~\ref{def:x-shape-upstairs}: the shape of $x_{\alpha\wedge\beta}$ is $(1^0)=0\in \Part_n$, and everything works.) 

What we have to show is that each move of this type strictly lowers the shape in $\overline S$ (Definition~\ref{def:x-shape-upstairs}) with respect to dominance order. Because statement~\ref{item:shape-dominated} holds vacuously for any product that becomes zero, we can focus on the nontrivial straightening law \ref{item:common-upper-bound}.

Because dominance is compatible with addition in $\Part_n$ (Observation~\ref{obs:dominance-monoid-compatible}) and $\overline S$ is graded by shape (Definition~\ref{def:x-shape-upstairs}), it follows that dominance on shapes is preserved by multiplication by the monomial $m_1m_2/(x_\alpha x_\beta) \in \overline S$. Thus, it is enough to show that every monomial appearing in 
\[
x_{\alpha\wedge \beta}\sum_{\gamma\in\operatorname{lub}(\alpha,\beta)} x_\gamma
\]
has shape strictly dominated by that of $x_\alpha x_\beta$. We see this as follows. Without loss of generality, assume $\rk(\alpha)\geq \rk(\beta)$, where, as usual, ranks are calculated in $\widehat P(\Delta)$. We have
\begin{align*}
\shape(x_\alpha x_\beta) &= \left(1^{\rk(\alpha)}\right) + \left(1^{\rk(\beta)}\right)\\
&= \left(2^{\rk(\beta)},1^{\rk(\alpha)-\rk(\beta)}\right).
\end{align*}
Meanwhile, for each $\gamma\in\operatorname{lub}(\alpha,\beta)$, we have
\begin{equation}\label{eq:order-of-a-b}
\alpha\wedge\beta \prec\alpha \prec\gamma\text{ and }\alpha\wedge\beta \prec\beta \prec\gamma
\end{equation}
in $\widehat P(\Delta)$, and it follows first---because \eqref{eq:order-of-a-b} implies $\rk(\alpha\wedge\beta)< \rk(\gamma)$---that
\[
\shape(x_{\alpha\wedge\beta}x_\gamma) = \left(2^{\rk(\alpha\wedge\beta)},1^{\rk(\gamma)-\rk(\alpha\wedge\beta)}\right),
\]
and then---because \eqref{eq:order-of-a-b} implies $\rk(\alpha\wedge\beta) < \rk(\beta)$; the inequality is strict because $\alpha\wedge \beta$ and $\beta$ are comparable but distinct in $P(\Delta)$---that
\[
\shape(x_{\alpha\wedge\beta}x_\gamma) \triangleleft \shape(x_\alpha x_\beta),
\]
with strict dominance, as required.
\end{proof}

It follows immediately that $\kk[\Delta]$ is {\em filtered} over $\Part_n$ with respect to dominance order:

\begin{prop}[Filtration of {$\kk[\Delta]$} over {$\Part_n$}]\label{prop:x-filtered}
    Let $\lambda_1,\lambda_2\in \Part_n$, and let $f_1\in \kk[\Delta]_{\lambda_1}$ and $f_2\in \kk[\Delta]_{\lambda_2}$. Then
    \[
    f_1f_2\in \bigoplus_{\mu \trianglelefteq \lambda_1 + \lambda_2} \kk[\Delta]_{\mu}.
    \]
\end{prop}

\begin{proof}
    After representing $f_1,f_2$ on the basis of standard monomials, expand $f_1f_2$ as a sum of products of standard monomials, and apply Lemma~\ref{lem:filter-main-lemma} to each of these products.
\end{proof}

\begin{remark}
    Proposition~\ref{prop:x-filtered} generalizes without change to any ring with lexicographic straightening law in the sense of Baclawski \cite{baclawski1981rings}.
\end{remark}

\begin{observation}\label{obs:associated-graded}
    The fact that Proposition~\ref{prop:x-filtered} gives a filtration of $\kk[\Delta]$ suggests to define an {\em associated graded} $\kk$-algebra for $\kk[\Delta]$, by setting
    \[
    \operatorname{Gr}(\kk[\Delta]):= \bigoplus_{\lambda\in \Part_n} \left(\oplus_{\mu \trianglelefteq \lambda} \kk[\Delta]_\mu /\oplus_{\mu \triangleleft \lambda} \kk[\Delta]_\mu\right),
    \]
    with the multiplication induced by the multiplication in $\kk[\Delta]$. But in fact, the ring $\operatorname{Gr}(\kk[\Delta])$ defined this way is straightforwardly identified with $\kk[\Sd\Delta]$, so we have obtained nothing new. This follows from Lemma~\ref{lem:homomorphism-in-top-shape} below.
\end{observation}

From Proposition~\ref{prop:parameter-monom-every-term}, in view of Proposition~\ref{prop:x-filtered}, we have some information about monomials in the parameters $\theta_1,\dots,\theta_n$ for $\kk[\Delta]$:

\begin{prop}\label{prop:param-monom-x-ring}
    For any natural numbers $a_1,\dots,a_n$, the expansion of the product
    \[
    \theta_1^{a_1}\cdots \theta_n^{a_n}
    \]
    on the basis of standard monomials for $\kk[\Delta]$ contains every standard monomial of shape
    \[
    a_1(1^1) + \dots + a_n(1^n),
    \]
    and all other monomials appearing in the expansion have shapes dominated by this.\qed
\end{prop}

\subsection{The Garsia transfer}\label{sec:garsia-transfer}

It was mentioned above that the map $\Garsia:\kk[\Sd\Delta]\rightarrow \kk[\Delta]$ defined in Setup~\ref{set:garsia} plays an important role. We now pause to give it a name.
\begin{definition}[The Garsia transfer]\label{def:garsia}
The map $\Garsia:\kk[\Sd\Delta] \rightarrow \kk[\Delta]$, defined first by mapping
\[
y_\alpha \mapsto x_\alpha
\]
for $\alpha\in P(\Delta)$, then extending multiplicatively to standard monomials, and finally, extending $\kk$-linearly to the entirety of $\kk[\Sd\Delta]$, is the {\em Garsia transfer}, or just the {\em Garsia map}.
\end{definition}

\begin{example}\label{ex:parameters-under-garsia}
    For $j=1,\dots,n$, we have
    \[
    \Garsia(\gamma_j) = \theta_j.
    \]
    This will be used below. However, as mentioned in Section~\ref{sec:parameter-subring}, $\Garsia$'s restriction to $\kk[\Gamma]$ does not coincide with the ring isomorphism $\Psi:\kk[\Gamma]\rightarrow \kk[\Theta]$ defined there. For example, 
    \begin{equation}\label{eq:psi-version}
    \Psi(\gamma_1^2) = \Psi(\gamma_1)^2 = \theta_1^2 = \left(\sum_{\rk(\alpha)=1} x_\alpha\right)^2,
    \end{equation}
    while on the other hand,
    \begin{equation}\label{eq:garsia-version}
    \Garsia(\gamma_1^2) = \Garsia\left(\left(\sum_{\rk(\alpha)=1} y_\alpha\right)^2\right) = \Garsia\left(\sum_{\rk(\alpha)=1} y_\alpha^2\right) = \sum_{\rk(\alpha)=1} x_\alpha^2,
    \end{equation}
    where the middle equality in \eqref{eq:garsia-version} is because no two $\alpha$'s of the same height are supported on the same chain in $P(\Delta)$ (refer to  Setup~\ref{set:garsia}). 
    If $\Delta$ has dimension greater than zero, then there exist $\alpha,\beta\in P(\Delta)$ with $\rk(\alpha)=\rk(\beta)=1$ (i.e., $\alpha,\beta$ are vertices of $\Delta$)  but $\alpha,\beta$ have at least one common upper bound $\gamma$ in $P(\Delta)$, and then $x_\alpha x_\beta$ is nonzero (it is a sum containing $x_\gamma$). If also the characteristic of $\kk$ is different from $2$, then $2x_\alpha x_\beta$ is nonzero as well, thus \eqref{eq:psi-version} contains a cross-term missing from \eqref{eq:garsia-version}, and $\Psi(\gamma_1)^2 \neq \Garsia(\gamma_1^2)$.
\end{example}

\begin{remark}
The map $\Garsia$ (in the case that $\kk[\Delta]$ is a polynomial ring or various subrings) is called the {\em transfer} (with no modifier) in \cite{garsia, garsia-stanton, reiner, reiner2}; the name is explained by Theorem~\ref{thm:transfer-bases} below. This is also a common name for another important map, 
\[
f\mapsto \sum_{\sigma \in G} \sigma\cdot f,
\]
where $G$ is a group of automorphisms. We are following \cite{blum-smith, blum2018permutation, pevzner2024symmetric}, where it is called the {\em Garsia map} to avoid the name collision, and in honor of Garsia's introduction of it in \cite{garsia}; and we are also proposing {\em Garsia transfer} as a compromise.
\end{remark}

\begin{remark}
    All of the theory discussed in this section works without significant change at the generality of Baclawski's theory of rings with lexicographic straightening law \cite{baclawski1981rings}. The map $\Garsia$ is denoted $\phi$ in \cite{baclawski1981rings} (and $\varphi$ in \cite{baclawski1981combinatorial}). 
\end{remark}

We record some evident-but-important properties of the Garsia transfer:

\begin{observation}\label{obs:garsia-preserves-shape}
    Since the Garsia transfer maps standard monomials in $\kk[\Sd\Delta]$ to the corresponding standard monomials in $\kk[\Delta]$, by comparing Definition~\ref{def:shape} with Definitions~\ref{def:x-shape-upstairs} and \ref{def:x-shape} one sees that it preserves shape, i.e., if $m\in \kk[\Sd\Delta]$ is a standard monomial then
    \[
    \shape(\Garsia(m)) = \shape(m).
    \]
\end{observation}

\begin{observation}\label{obs:garsia-equivariant}
    If $G\subseteq \Aut(\Delta)$ is a group of automorphisms of $\Delta$, with the induced actions on $P(\Delta)$, $\kk[\Delta]$ and $\kk[\Sd\Delta]$, then the Garsia transfer $\Garsia:\kk[\Sd\Delta]\rightarrow \kk[\Delta]$ is $G$-equivariant.
\end{observation}

\begin{observation}\label{obs:garsia-graded}
    With respect to the $\NN$-gradings on $\kk[\Delta]$ and $\kk[\Sd\Delta]$ defined in Setup~\ref{set:N-grading}, the Garsia transfer $\Garsia:\kk[\Sd\Delta]\rightarrow \kk[\Delta]$ is a graded map.
\end{observation}

As noted in Setup~\ref{set:garsia}, the Garsia transfer is not a ring map. Neither is it a $\kk[\Theta]$-module map. It is, however, a {\em coarse approximation to a ring homomorphism}. By Lemma~\ref{lem:filter-main-lemma}, $\kk[\Delta]$ is a ring with lexicographic straightening law based on $P(\Delta)$ in the sense of Baclawski \cite[Definition~4.2]{baclawski1981rings}, and, per Setup~\ref{set:garsia}, $\kk[\Sd\Delta]$ is exactly $\kk[P(\Delta)]$. So, as a special case of  \cite[Lemma~4.4]{baclawski1981rings}, we have:

\begin{lemma}\label{lem:homomorphism-in-top-shape}
Let $\lambda_1,\lambda_2\in \Part_n$, and let $f_1\in \kk[\Sd\Delta]_{\lambda_1}$, $f_2\in \kk[\Sd\Delta]_{\lambda_2}$. Then
\[
\Garsia(f_1)\Garsia(f_2) - \Garsia(f_1f_2)\in \bigoplus_{\mu\triangleleft\lambda_1 +\lambda_2} \kk[\Delta]_\mu.
\]
Note that the direct sum is over $\mu$ {\em strictly} dominated by $\lambda_1 + \lambda_2$.\qed
\end{lemma}

\begin{remark}
    We have cited \cite[Lemma~4.4]{baclawski1981rings} for Lemma~\ref{lem:homomorphism-in-top-shape}, but for the reader who prefers a self-contained treatment we note that Lemma~\ref{lem:homomorphism-in-top-shape} follows essentially immediately from Lemma~\ref{lem:filter-main-lemma} by noting that the statement is bilinear in $f_1,f_2$ and breaking each into individual standard monomials, and then noting that $\Garsia$ preserves whether or not a pair of monomials stacks up.
\end{remark}

\begin{remark}
We think of Lemma~\ref{lem:homomorphism-in-top-shape} as asserting that the Garsia transfer is a ``homomorphism in the top shape". In the language of Observation~\ref{obs:associated-graded}, the Garsia transfer induces an isomorphism from $\kk[\Sd\Delta]$ to the associated graded algebra of $\kk[\Delta]$.
\end{remark}

\begin{observation}\label{obs:any-number-of-factors}
Lemma~\ref{lem:homomorphism-in-top-shape} immediately generalizes by induction to any number of factors.
\end{observation}

The following is the main original application of the Garsia transfer in \cite{garsia, baclawski1981combinatorial, baclawski1981rings, garsia-stanton}, adapted to our setting.

\begin{theorem}[Main theorem on the Garsia transfer]\label{thm:transfer-bases}
    Let $\Delta$ be a boolean complex. Let
    \[
    f_1,\dots, f_r \in \kk[\Sd\Delta]
    \]
    be homogeneous with respect to shape. Then:
    \begin{enumerate}
    \item If $f_1,\dots,f_r$ generate $\kk[\Sd\Delta]$ as a $\kk[\Gamma]$-module, then $\Garsia(f_1),\dots,\Garsia(f_r)$ generate $\kk[\Delta]$ as a $\kk[\Theta]$-module.\label{item:generate}
    \item If $f_1,\dots,f_r$ are $\kk[\Gamma]$-linearly independent, then $\Garsia(f_1),\dots,\Garsia(f_r)$ are $\kk[\Theta]$-linearly independent.\label{item:lin-indep}
    \end{enumerate}
\end{theorem}

We will cite \cite{baclawski1981rings} for assertion~\ref{item:generate}, but the analog to assertion~\ref{item:lin-indep} in \cite{baclawski1981combinatorial} (and in \cite{garsia, baclawski1981combinatorial, garsia-stanton}) is proven only under the hypothesis of  (the analog to) assertion~\ref{item:generate}. Thus, for assertion~\ref{item:lin-indep} we include a self-contained proof to show that it is independent of this hypothesis. 

\begin{proof}[Proof of Theorem~\ref{thm:transfer-bases}]
    Since by Lemma~\ref{lem:filter-main-lemma} $\kk[\Delta]$ is a ring with lexicographic straightening law based on $P(\Delta)$ in the sense of Baclawski \cite{baclawski1981rings}, and $\kk[\Sd\Delta]=\kk[P(\Delta)]$, assertion~\ref{item:generate} is a specialization of assertion~1 of \cite[Theorem~4.3]{baclawski1981rings}. For the sake of self-containedness, we briefly indicate the idea of the proof. Namely, representations of elements of $\kk[\Delta]$ as $\kk[\Theta]$-linear combinations of $\Garsia(f_1),\dots,\Garsia(f_r)$ can be built by induction on shape from representations of elements of $\kk[\Sd\Delta]$ as $\kk[\Gamma]$-linear combinations of $f_1,\dots,f_r$. Lemma~\ref{lem:homomorphism-in-top-shape} and Observation~\ref{obs:any-number-of-factors} tell us that the Garsia image of a representation in $\kk[\Sd\Delta]$ emulates a representation in $\kk[\Delta]$ in the maximal shape components, and then it can be adjusted based on the lower-shape remainders until there is no remainder. The procedure is illustrated below in Example~\ref{ex:proof-illlustration}.

    To prove assertion~\ref{item:lin-indep}, we suppose the existence of a nontrivial $\kk[\Theta]$-linear relation among $\Garsia(f_1),\dots,\Garsia(f_r)$, and use it to find a nontrivial $\kk[\Gamma]$-linear relation among $f_1,\dots,f_r$. The main idea is that, due to  Lemmas~\ref{lem:filter-main-lemma} and \ref{lem:homomorphism-in-top-shape}, we can isolate the part of a $\kk[\Theta]$-linear relation between the $\Garsia(f_j)$'s that takes place in a maximal-shape component to produce a $\kk[\Gamma]$-linear relation between the $f_j$'s.
    
    To this end, suppose we have some elements $p_1,\dots,p_r\in \kk[\Theta]$, not all zero, so that
    $
    0=\sum_{j=1}^r p_j \Garsia(f_j).
    $
    Breaking each $p_j$ into monomials in the $\theta$'s and discarding zero terms, we can express this as a nontrivial (finite) linear combination 
    \begin{equation}\label{eq:lin-rel}
    0=\sum_{j,\bfa}c_{j,\bfa}\theta^\bfa \Garsia(f_j),
    \end{equation}
    where $\theta^\bfa:= \theta_1^{a_1}\dots \theta_n^{a_n}$ for a tuple of natural numbers $\bfa:=(a_1,\dots,a_n)$, and $c_{j,\bfa}\in \kk^\times$.

    Find a shape, $\lambda$, that is maximal with respect to dominance order among all of the shapes that occur in the expansions of any of the $c_{j,\bfa}\theta^\bfa\Garsia(f_j)$'s appearing in \eqref{eq:lin-rel} into standard monomials. Because $\kk[\Delta]$ is the direct sum of its shape-homogeneous components, \eqref{eq:lin-rel} implies that all terms of shape $\lambda$ appearing in the expansions of the $c_{j,\bfa}\theta^\bfa\Garsia(f_j)$'s must cancel out.
    
    Consider any $c_{j,\bfa}\theta^\bfa\Garsia(f_j)$ whose expansion contains a standard monomial of shape $\lambda$. One of two scenarios is possible. By Proposition~\ref{prop:x-filtered} and the remark following it, either
    \begin{equation}\label{eq:shape-equality}
    \lambda = a_1(1^1) + \dots + a_n(1^n) + \shape(f_j),
    \end{equation}
    or 
    \[
    \lambda \triangleleft a_1(1^1) + \dots + a_n(1^n) + \shape(f_j)
    \]
    (strict dominance).\footnote{Actually, this second scenario is not possible, due to Proposition~\ref{prop:param-monom-x-ring}, but we formulate the proof in the present way to support generalization to settings where an analogous statement to Proposition~\ref{prop:param-monom-x-ring} does not hold, such as that of \cite{baclawski1981rings}.\label{note:baclawski-generality}} In the latter case, the maximality of $\lambda$ implies that the component of $c_{j,\bfa}\theta^\bfa\Garsia(f_j)$ in shape $a_1(1^1) + \dots + a_n(1^n) + \shape(f_j)$ is zero. Then by Lemma~\ref{lem:homomorphism-in-top-shape} and Observation~\ref{obs:any-number-of-factors}, we have $c_{j,\bfa}\gamma^\bfa f_j=0$, and this is a nontrivial linear relation of the $f_j$'s over $\kk[\Gamma]$, so we are done.
    
    Thus we can assume going forward that each $c_{j,\bfa}\theta^\bfa\Garsia(f_j)$ that contributes a standard monomial of shape $\lambda$ to the relation \eqref{eq:lin-rel} satisfies the equality \eqref{eq:shape-equality}; write $j^\star, \bfa^\star$ for any $j,\bfa$ such that this happens. Then, again by Lemma~\ref{lem:homomorphism-in-top-shape} and Observation~\ref{obs:any-number-of-factors}, the terms of shape $\lambda$ occurring in such a $c_{j^\star,\bfa^\star}\theta^{\bfa^\star}\Garsia(f_{j^\star})$ are exactly the images of the terms of $c_{j^\star,\bfa^\star}\gamma^{\bfa^\star}f_{j^\star}\in \kk[\Sd\Delta]$ under the Garsia transfer. Since we know that all terms of shape $\lambda$ in \eqref{eq:lin-rel} must cancel out, and the Garsia transfer is a linear isomorphism, we have
    \[
    0=\sum_{j^\star,\bfa^\star}c_{j^\star,\bfa^\star}\gamma^{\bfa^\star}f_{j^\star},
    \]
    and this is a nontrivial linear relation of the $f_j$'s over $\kk[\Gamma]$, completing the proof.
\end{proof}

\begin{example}\label{ex:proof-illlustration}
We illustrate the proof of Theorem~\ref{thm:transfer-bases} using our running example from Figure~\ref{fig:running-ex}. 

The proof of assertion~\ref{item:generate} inputs a shape-homogeneous basis for $\kk[\Sd\Delta]$ as $\kk[\Gamma]$-module, and provides an algorithm to express any element  $f\in \kk[\Delta]$ in terms of the Garsia-image of this basis, assuming that one has an algorithm to express $\Garsia^{-1}(f)$ on the given basis for $\kk[\Sd\Delta]$. (Such an algorithm is provided in Section~\ref{sec:construct-a-basis}.)

To illustrate, we choose $f_1 = 1$, $f_2 = y_v$, $f_3 = y_\alpha$, and $f_4 = y_v y_\alpha$, which form a basis of $\kk[\Sd\Delta]$ over $\kk[\Gamma]$ (we will prove this in Section~\ref{sec:construct-a-basis} below). In this example, we demonstrate how the algorithm works by applying it to a specific element of $\kk[\Delta]$, namely $f:=x_w^2 x_\beta$.

The first step is to write $\Garsia^{-1}(x_w^2 x_\beta) = y_w^2 y_\beta$ in terms of the basis $f_1$, $f_2$, $f_3$, $f_4$ (we show how to do this in Section~\ref{sec:construct-a-basis}). We obtain:
    \[
     y_w^2y_\beta= \gamma_1^2\gamma_2 - \gamma_1\gamma_2 y_v - \gamma_1^2 y_\alpha + \gamma_1 y_vy_\alpha.
    \]
    The proof now shows how to find the original $f\in\kk[\Delta]$ in the $\kk[\Theta]$-span of $\Garsia(f_1)=1$, $\Garsia(f_2)=x_v$, $\Garsia(f_3)=x_\alpha$, and $\Garsia(f_4)=x_vx_\alpha$.  Note that it has shape $2\cdot(1) + (1,1) = (3,1)$. The idea is that applying the Garsia transfer to each factor of each term in the above representation, we get something that matches our target $f$ in shape $(3,1)$, and all terms of the difference must have shape strictly dominated by this. We have
    \[
     x_w^2x_\beta - (\theta_1^2\theta_2 - \theta_1\theta_2x_v - \theta_1^2 x_\alpha + \theta_1x_vx_\alpha) = -x_\beta^2,
    \]
    and the point is that the sole term of the remainder has shape $(2,2)$, which is strictly dominated by $(3,1)$ as was to be expected from Lemma~\ref{lem:homomorphism-in-top-shape} and Observation~\ref{obs:any-number-of-factors}, so we have made progress. Applying the same process to $\Garsia^{-1}(-x_\beta^2)$, we obtain:
    \[
    \Garsia^{-1}(-x_\beta^2)= -y_\beta^2 =   - \gamma_2^2+\gamma_2y_\alpha,
    \]
    so the remainder
    \[
    (-x_\beta^2) - (-\theta_2^2+\theta_2x_\alpha)
    \]
    has only terms of further dominated shape, again by Lemma~\ref{lem:homomorphism-in-top-shape} and Observation~\ref{obs:any-number-of-factors}. In fact, it already equals zero, so, putting things together, we have achieved the desired representation of $x_w^2 x_\beta$ as a $\kk[\Theta]$-linear combination of $1, x_v, x_\alpha, x_vx_\alpha$, namely
    \[
    x_w^2x_\beta =(\theta_1^2\theta_2 - \theta_2^2) - \theta_1\theta_2 x_v + (-\theta_1^2 + \theta_2)x_\alpha + \theta_1 x_vx_\alpha.
    \]

   To illustrate the proof of assertion~\ref{item:lin-indep}, we present an explicit example showing how a linear dependence relation in $\kk[\Delta]$ induces a linear dependence relation in $\kk[\Sd\Delta]$, by applying the inverse of the Garsia map to a maximal-shape component of the relation in $\kk[\Delta]$ (in the sense described in the proof of the theorem). We consider the following relation between $1=\Garsia(1)$, $x_v=\Garsia(y_v)$, and $x_v^3=\Garsia(y_v^3)$ in $\kk[\Delta]$:
    \[
    0= x_v^3 - (\theta_1^2 - \theta_2)x_v + \theta_1\theta_2.
    \]
    The terms that occur in the expansions of $\theta_1\theta_2$ and $\theta_2x_v$ are all of shape $(2,1)$, while $x_v^3$ is of shape $(3)$ and $\theta_1^2 x_v$'s expansion contains terms of shapes $(2,1)$ and $(3)$. The only dominance-maximal partition among these is $\lambda = (3)$, and the $\theta$-monomial summands in the linear relation that contribute terms of this shape are $x_v^3$ and $-\theta_1^2 x_v$. Thus, we have
    \[
    0=y_v^3 - \gamma_1^2 y_v,
    \]
    a linear relation over $\kk[\Gamma]$ between the corresponding $1$, $y_v$, and $y_v^3$ in $\kk[\Sd\Delta]$ that isolates the part of the linear relation $0=x_v^3 - (\theta_1^2 - \theta_2)x_v + \theta_1\theta_2$ taking place in shape $\lambda = (3)$.
\end{example}

\begin{remark}
    Let $R\subset \kk[\Sd\Delta]$ be any subring with the following three properties:
    \begin{itemize}
        \item $R\supset \kk[\Gamma]$.
        \item $R$ is homogeneous with respect to shape, i.e., $R=\bigoplus_{\lambda\in\Part_n} R_\lambda$, where $R_\lambda = R\cap \kk[\Sd\Delta]_\lambda$. 
        \item $\Garsia(R)$ is a subring of $\kk[\Delta]$.
    \end{itemize}
    For example, $R$ could be an invariant ring  for the action of a group of automorphisms of $\Delta$ on $\kk[\Sd\Delta]$---it follows from Observations~\ref{obs:garsia-preserves-shape} and \ref{obs:garsia-equivariant} (and the fact that the $\gamma$'s are invariant) that such a ring has these properties. Then Theorem~\ref{thm:transfer-bases} generalizes to the same statements with $R$ in the place of $\kk[\Sd\Delta]$ and $\Garsia(R)$ in the place of $\kk[\Delta]$. This generalization recovers \cite[Theorem~9.2]{garsia-stanton} as the special case where $\Delta$ is a simplex and $R$ is the invariant ring of a permutation group $G$ acting on its vertices.
\end{remark}

\begin{remark}
This is a historical remark on the provenance of Theorem~\ref{thm:transfer-bases}. For convenience, in the rest of this remark we refer to a ring playing the role of $\kk[\Delta]$ in a result of this type as ``the $x$ ring", and that playing the role of $\kk[\Sd\Delta]$ as ``the $y$ ring" (after the notation we have used throughout for their generators). The upshot of such a theorem is that bases can be transferred from the $y$ ring to the $x$ ring using a map analogous to $\Garsia$.

The basic model is Garsia's \cite[Theorem~6.1]{garsia}. That theorem was a parallel result in which the $x$ ring was a {\em partition ring}, which is a certain subring of a polynomial ring depending on a finite poset $Q$, and the $y$ ring was the Stanley--Reisner ring of the poset $P(Q)$ of order ideals of $Q$, which is a distributive lattice.  

With Baclawski, Garsia then gave the case of Theorem~\ref{thm:transfer-bases} where $\Delta$ is a simplicial complex \cite[Theorem~4.3]{baclawski1981combinatorial} (modulo the point mentioned immediately after the theorem statement about the independence of assertions~\ref{item:generate} and \ref{item:lin-indep}). Actually, that result also introduced some additional flexibility in the choice of the $\theta$'s and $\gamma$'s, and the ability to enforce certain constraints on the sense in which the $f$'s are to module-generate the $x$ and $y$ rings over them. Baclawski then developed his theory of algebras with lexicographic straightening law \cite{baclawski1981rings} as a common generalization of \cite{garsia} and \cite{baclawski1981combinatorial} (and various other important examples such as bracket rings and letterplace algebras). In Baclawski's theory, the $x$ ring is any $\kk$-algebra with a linear basis of ``standard monomials" on some poset $P$, as in Definition~\ref{def:standard-monomial}, and satisfying a statement like Lemma~\ref{lem:filter-main-lemma} (with respect to some partial order on the standard monomials satisfying properties like Observations~\ref{obs:dominance-monoid-compatible} and \ref{obs:dominance-induction}); then  the $y$ ring is the Stanley--Reisner ring $\kk[P]$ of the poset. In this context, \cite[Theorem~4.3]{baclawski1981rings} is the analog of Theorem~\ref{thm:transfer-bases} in that general setting (again, modulo the point about the independence of assertions~\ref{item:generate} and \ref{item:lin-indep}).

With Stanton, Garsia then gave  a version \cite[Theorem~9.2]{garsia-stanton} where the $x$ ring is the invariant ring of a permutation group $G\subset \mathfrak{S}_n$ acting on the polynomial ring $\kk[x_1,\dots,x_n] = \kk[\Delta_d]$ (with $d=n-1$), and the $y$ ring is the $G$-invariant subring of $\kk[\Sd\Delta_d]$. The passage to invariant rings removes the setting of \cite{garsia-stanton} from the immediate purview of Baclawski's general theory---but see the previous remark. In \cite[Theorem~9.4 and 9.5]{garsia-stanton}, \cite[Theorem~9.2]{garsia-stanton} was generalized in a different direction:  $\mathfrak{S}_n$ was replaced by any Weyl group  $W$, $G\subset \mathfrak{S}_n$ was replaced by arbitrary $H\subset W$, the $x$ ring was the ring of $H$-invariants in the Laurent polynomial ring $\kk[x_1,\dots,x_n,1/x_1,\dots,1/x_n]$, viewed as the group ring of $W$'s weight lattice, and the $y$ ring was the ring of $H$-invariants in the Stanley--Reisner ring of the Coxeter complex of $W$. This application of the technique was further studied in \cite{reiner1995gobel}.

The strategy in \cite{garsia, baclawski1981combinatorial, baclawski1981rings, garsia-stanton} was to prove an analog to assertion~\ref{item:generate} of Theorem~\ref{thm:transfer-bases} by induction on shape, as illustrated in Example~\ref{ex:proof-illlustration}, and then, assuming that the $f_j$'s generate the $x$ ring as a module over the parameter subring, to prove assertion~\ref{item:lin-indep} by dimension-counting because the $x$ and $y$ rings have the same Hilbert series. In \cite{blum-smith}, it was shown that in the setting of \cite{garsia-stanton} (except with the ground field $\QQ$ replaced by an arbitrary integral domain),  assertion~\ref{item:lin-indep} holds independent of the hypothesis of assertion~\ref{item:generate} (see \cite[Theorem~2.5.68 and Remark~2.5.69]{blum-smith}). 

Beyond the expository, the present contributions are:
\begin{itemize}
    \item the (straightforward) adaptation of the theorem from simplicial to arbitrary boolean complexes $\Delta$, the key piece of which was Lemma~\ref{lem:filter-main-lemma} (verifying that $\kk[\Delta]$ is a lexicographic straightening law based on $P(\Delta)$ in the sense of Baclawski), and
    \item the verification that assertion~\ref{item:lin-indep} is independent of the hypothesis of assertion~\ref{item:generate} at this generality. Our proof follows the idea of the proof in \cite{blum-smith}, but with a modification so that it would also work with little change at the generality of Baclawski \cite{baclawski1981rings} (provided that the $\theta$'s in \cite[Theorem~4.3]{baclawski1981rings} are assumed to be shape-homogeneous)---see note~\ref{note:baclawski-generality}.
\end{itemize} 
\end{remark}

\section{Garsia's linear algebra characterization of Cohen--Macaulayness}\label{sec:garsia-LA-CM}

In this section, we present a  theorem, Theorem~\ref{thm:garsia-LA-CM} below, essentially due to Garsia \cite{garsia}, that characterizes the Cohen--Macaulayness for a pure, balanced boolean complex $\Lambda$ in terms of a certain arrangement of linear subspaces  in a single finite-dimensional vector space over $\kk$.\footnote{We resist the urge to use $\Delta$ for this complex because our principal application will take $\Lambda = \Sd\Delta$ where $\Delta$ is as throughout; in particular, $\Lambda$ will {\em not} be specialized to $\Delta$. See Setup~\ref{set:boolean-simplicial-identification} for how to think of the simplicial complex $\Sd\Delta$ as a boolean complex.\label{note:don't-think-Lambda-is-Delta}} In the case that $\Lambda$ is Cohen--Macaulay, it also allows to construct a basis for the Stanley--Reisner ring $\kk[\Lambda]$ over a certain parameter subring $\kk[\Omega]$. In Section~\ref{sec:construct-a-basis}, this theorem is used to give an algorithm to construct a $\kk[\Omega]$-basis for $\kk[\Lambda]$, which is then applied with $\Lambda = \Sd \Delta$ and $\Omega = \Gamma$.

Theorem~\ref{thm:garsia-LA-CM} is not stated explicitly in \cite{garsia}, but is implicit in \cite[Section~3]{garsia}, especially \cite[Theorem~3.3]{garsia}, in the case that $\Lambda$ is the order complex of a ranked poset. We generalize the result to an arbitrary pure, balanced Boolean complex $\Lambda$. To expedite proofs of some of the lemmas in this more general setting, we use a lemma coming from a point of view in toric topology \cite{buchstaber-panov} that the Stanley--Reisner ring is a limit (in the category of commutative, graded $\kk$-algebras) over a diagram of polynomial rings indexed by the face poset. The precise statement is Lemma~\ref{lem:proj-to-beta} below.

We establish notation used throughout this section. Let $\Lambda$ be a pure boolean complex of dimension $d$. Let $n=d + 1$ (so that the facets of $\Lambda$ have $n$ vertices). The complex $\Lambda$ is {\em balanced} if there is a labeling of its vertex set $V_\Lambda$ by $n$ labels (aka {\em colors}) such that for every face $\alpha$ of $\Lambda$, the vertices belonging to $\alpha$ all have distinct labels. (It is sufficient, and sometimes useful, to check this condition on faces $\alpha$ just over  facets. If a pure complex is balanced, each facet is incident to exactly one vertex in each of the label classes.) Going forward, we assume $\Lambda$ is balanced, and equipped with a specific labeling satisfying this condition. We take $[n]:=\{1,\dots,n\}$ as our collection of labels, and for a vertex $v\in V_\Lambda$, we write $\lb(v)$ for its label. Then to each face $\alpha \in P(\Lambda)$ is given a {\em label set} $J_\alpha \subset [n]$ of cardinality $\rk(\alpha)$, consisting of the labels of the vertices belonging to $\alpha$. In symbols,
\[
J_\alpha :=\{ \lb(v): v\in V_\Lambda\text{ s.t. }v\preceq \alpha\}.
\]

To provide a shape-homogeneous basis of $\kk[\Sd\Delta]$ as $\kk[\Gamma]$-module in Section~\ref{sec:construct-a-basis}, $\Lambda$ will be specialized to  $\Sd\Delta$ in the situation of Theorem~\ref{thm:yes-CM-in-coprime}, so we here note why the latter fulfills the hypotheses on $\Lambda$. The barycentric subdivision $\Sd\Delta$ has vertices in bijection with the faces $\alpha\in P(\Delta)$ of the original boolean complex $\Delta$; a balancing is given by assigning the label $\rk(\alpha)$ to the vertex corresponding with $\alpha$ (which can be thought of as the barycenter of the face $\alpha$ in the original complex). It is pure because Theorem~\ref{thm:yes-CM-in-coprime} assumes that $\Delta$ and thus $\Sd\Delta$ is Cohen--Macaulay over $\kk$, and a Cohen--Macaulay simplicial complex is necessarily pure \cite[Corollary~5.1.5]{bruns-herzog}.\footnote{It also follows that, more generally, a Cohen--Macaulay boolean complex is pure. Both purity and Cohen--Macaulayness are unaffected by taking the barycentric subdivision, which reduces the question to the simplicial case.}

Since $\Lambda$ is a boolean complex, the Stanley--Reisner ring $\kk[\Lambda]$ is defined as in Definition~\ref{def:SR-ring-of-boolean}, but we  use $z_\alpha$ instead of $x_\alpha$ for the generators in order to avoid confusion between $\kk[\Lambda]$ and $\kk[\Delta]$ (see Note~\ref{note:don't-think-Lambda-is-Delta}). There is an $\NN^n$-(multi)grading on $\kk[\Lambda]$ given by assigning the degree 
\[
\deg_\mathrm{md}z_\alpha :=\sum_{j\in J_\alpha}\bfe_j
\]
to the generator $z_\alpha$, where $\bfe_1,\dots,\bfe_n$ are the standard basis of $\NN^n$. (As in Section~\ref{sec:grading-by-shape}, the subscript $\mathrm{md}$ stands for ``multidegree".) This gives a well-defined grading because the defining relations \ref{item:no-common-upper-bound} and \ref{item:common-upper-bound} of Definition~\ref{def:SR-ring-of-boolean} are homogeneous with respect to it: due to the balancing, this homogeneity reduces to the identity $\sum_{j\in J}\bfe_j + \sum_{j\in J'}\bfe_j = \sum_{j\in J\cap J'} \bfe_j + \sum_{j\in J\cup J'} \bfe_j$ for subsets $J,J'\subset [n]$ of the collection of labels. Note that if a standard monomial contains at least two $z_\alpha$'s, then its degree contains some $\bfe_j$ at least twice. Thus the generators $z_\alpha$ themselves can be recognized among the standard monomials by the fact that their $\NN^n$-degrees have the form $\sum_{j\in J} \bfe_j$ for $J\subset [n]$ a set.

When we apply this with $\Lambda = \Sd\Delta$, this will specialize to the grading defined in Section~\ref{sec:grading-by-shape} (and thus can also be viewed as a grading over $\Part_n$, although we will not use this here, as for the present purpose we will not have need to compare $\kk[\Lambda]$ to a ring filtered over $\Part_n$).

As in Setup~\ref{set:asl}, a $\kk$-basis for $\kk[\Lambda]$ is given by standard monomials, i.e., monomials in the $z_\alpha$'s that are supported on chains in the face poset $P(\Lambda)$. Any chain in $P(\Lambda)$ is upper-bounded by a facet $\epsilon$ of $\Lambda$.

\begin{definition}
    We will say that a standard monomial supported on a chain upper-bounded by a given facet $\epsilon$ {\em sits under} that facet, and the facet {\em sits over} the monomial. (Note that every standard monomial sits under some facet.) More generally, if a monomial is supported on a chain upper bounded by any face $\beta$, whether a facet or not, we say that this monomial sits under $\beta$, and $\beta$ sits over it.
\end{definition} 

We now pull in an idea from toric topology: the Stanley--Reisner ring $\kk[\Lambda]$ is the limit of a diagram of polynomial rings indexed by $P(\Lambda)$ \cite[Lemma~3.5.11]{buchstaber-panov}. In particular, we have the following lemma. Any face $\beta\in P(\Lambda)$ of $\Lambda$, being a simplex, can itself be viewed as a boolean (and even a simplicial) complex, with Stanley--Reisner ring $\kk[\beta]$ isomorphic to the polynomial ring on the vertex set of $\beta$, and there is a canonical graded algebra map from $\kk[\Lambda]$ to $\kk[\beta]$, which can be described explicitly in terms of the basis of standard monomials. We write
\[
\kk[\beta]:=\kk\left[\left\{z_v^\beta\right\}_{v\in V_\Lambda\text{ and }v\preceq \beta}\right],
\]
where the $z_v^\beta$'s are the indeterminates of the polynomial ring $\kk[\beta]$. This ring is naturally $\NN^n$-graded by assigning $\deg_\mathrm{md} z_v^\beta:= \bfe_{\lb(v)}$. Then:

\begin{lemma}\label{lem:proj-to-beta}
    For any $\beta\in P(\Lambda)$, there is a canonical $\NN^n$-graded $\kk$-algebra map $s_\beta:\kk[\Lambda]\rightarrow\kk[\beta]$ defined by sending
\[
z_\alpha \mapsto \begin{cases}
    \prod_{v\preceq \alpha}z_v^\beta&,\; \alpha\preceq \beta\\
    0&,\; \alpha\npreceq \beta.
\end{cases}
\]
The map $s_\beta$ restricts to a linear isomorphism on the span of the standard monomials sitting under $\beta$, sending such standard monomials to monomials of $\kk[\beta]$; and all other standard monomials lie in the kernel.\footnote{Note that the linear span of the standard monomials sitting under $\epsilon$ does not form a subalgebra, so the restriction to this span is not an algebra map.}
\end{lemma}

\begin{proof}
    In \cite[Proposition~3.5.5]{buchstaber-panov}, the quotient 
    ring $\kk[\Lambda] / (\{z_\alpha\}_{\alpha \npreceq \beta})\kk[\Lambda]$ is identified with the polynomial ring $\kk[\beta]$. The ideal in the denominator contains all and only those standard monomials not sitting under $\beta$; meanwhile, the standard monomials sitting under $\beta$ are in natural bijection with the monomials of the polynomial ring $\kk[\beta]$. Thus, the canonical quotient map is identified with the map $s_\beta$ described in the lemma. It is $\NN^n$-graded because for any $\alpha\preceq \beta$, we have
    \[
    \deg_\mathrm{md}\left(\prod_{v\preceq \alpha}z_v^\beta\right) = \sum_{v\preceq \alpha}\bfe_{\lb(v)} = \deg_\mathrm{md} z_\alpha,
    \]
    so that the image of $z_\alpha\in \kk[\Lambda]$ has the same degree as $z_\alpha$ has, and for any $\alpha \npreceq \beta$, the image of $z_\alpha$ is zero.
\end{proof}

With this language, (pure) balanced complexes have the following interesting property, which generalizes the fact that in a multivariate polynomial ring, with the standard multigrading, all monomials have distinct degrees:

\begin{lemma}\label{lem:sits-under-lemma}
    If $f\in \kk[\Lambda]$ is homogeneous with respect to the $\NN^n$-grading described above, then no two distinct standard monomials of $f$ can sit under the same facet. In other words, a standard monomial of $\kk[\Lambda]$ is uniquely identified by the data of its $\NN^n$-degree and any facet it sits under.
\end{lemma} 

\begin{proof}
    Let $\epsilon$ be a facet of $\Lambda$. Then $s_\epsilon(f)\in \kk[\epsilon]$ is an $\NN^n$-homogeneous element in a polynomial ring with the standard multigrading, and therefore a monomial. Since by Lemma~\ref{lem:proj-to-beta} the restriction of $s_\epsilon$ to the span of the standard monomials sitting under $\epsilon$ is a linear isomorphism sending standard monomials to monomials, and $s_\epsilon$ is zero on standard monomials not sitting under $\epsilon$, it follows that in the expression for $f$ in terms of standard monomials, only one of them sits under $\epsilon$.
\end{proof}

Write  $\omega_1,\dots,\omega_n\in \kk[\Lambda]$ for the sums over label classes of the generators $z_v$ corresponding with vertices, i.e.,
\[
\omega_j := \sum_{\substack{v\in V_\Lambda \\ \lb(v)=j}} z_v
\]
for $j=1,\dots,n$. Write $\Omega:=\omega_1,\dots,\omega_n$, so that $\kk[\Omega]$ is the $\kk$-subalgebra of $\kk[\Lambda]$ generated by $\omega_1,\dots,\omega_n$. The $\omega_1,\dots,\omega_n$ form a homogeneous system of parameters for $\kk[\Lambda]$ (e.g., \cite[2.5.91]{blum-smith}, which is a common generalization of \cite[Proposition~III.4.3]{stanley1996combinatorics} and \cite[Theorem~6.3]{hodge-algebras}). Note that they are even homogeneous with respect to the $\NN^n$-grading defined above, with $\deg_\mathrm{md}(\omega_j) = \bfe_j$. When we apply this with $\Lambda = \Sd\Delta$, the $\omega_j$'s will specialize to the $\gamma_j$'s.\footnote{Because of this and because the labels of a balancing are often  called colors, in \cite{adams-reiner} the authors refer to $\gamma_1,\dots,\gamma_n\in\kk[\Sd\Delta]$ as the {\em colorful parameters}.}

We give some lemmas that establish a picture of how the ring $\kk[\Lambda]$ works as a $\kk[\Omega]$-module.

\begin{lemma}\label{lem:torsion-free}
    The ring $\kk[\Lambda]$ is torsion-free as a $\kk[\Omega]$-module.
\end{lemma}

\begin{proof}
In the case that $\Lambda$ is the order complex of a ranked poset, this was proven in \cite[Theorem~2.1]{garsia}. For the present generalization, we give a different proof based on the map $s_\epsilon$ described in Lemma~\ref{lem:proj-to-beta}.

Let $0\neq f\in \kk[\Lambda]$ and $j\in [n]$. We aim to show that $\omega_j f$ is nonzero. Expand $f$ as a $\kk$-linear combination of standard monomials in $\kk[\Lambda]$, and pick any monomial that appears in this linear combination with nonzero coefficient. Let $\epsilon$ be a facet of $\Lambda$ under which that monomial sits. Then $s_\epsilon(f)$ is nonzero, since by Lemma~\ref{lem:proj-to-beta} each monomial appearing in the expansion is sent to a monomial or zero, and the one sitting under $\epsilon$ does not lie in the kernel of $s_\epsilon$. Meanwhile, because $\epsilon$ has exactly one vertex of every label, it follows that $s_\epsilon(\omega_j)=z_v^\epsilon$ for the unique $v\preceq \epsilon$ with $\lb(v)=j$, and in particular, it is nonzero. Then
\[
s_\epsilon(\omega_j f) = s_\epsilon(\omega_j)s_\epsilon(f)\neq 0
\]
because $\kk[\epsilon]$ is a domain. It follows that $\omega_j f$ is nonzero.
\end{proof}

The following lemma generalizes Proposition~\ref{prop:parameter-monom-every-term}. (We think of it as folklore, but include a full proof for completeness.)

\begin{lemma}\label{lem:everything-of-degree}
    For any natural numbers $a_1,\dots,a_n$, the expansion of the product 
    \[
    \omega_1^{a_1}\dots \omega_n^{a_n}
    \]
    on the basis of standard monomials for $\kk[\Lambda]$ consists precisely of the sum of all standard monomials with $\NN^n$-degree
    \[
    a_1\bfe_1 + \dots + a_n \bfe_n.
    \]
\end{lemma}

\begin{proof}
    In any case, the expansion is a linear combination of standard monomials of the given $\NN^n$-degree, and what is to be argued is that every one of them appears with coefficient $1$. 
    
    Fix a facet $\epsilon$, and apply the map $s_\epsilon$ described in Lemma~\ref{lem:proj-to-beta}. Because of the balancing, $\epsilon$ contains exactly one vertex from each label class; let $v_j$ be its vertex with label $j$ for each $j=1,\dots,n$. Then because $\omega_j$ is the sum of $z_v$ over all vertices $v$ satisfying $\lb(v)=j$, but only one of these vertices is in $\epsilon$ (namely, $v_j$), we have $s_\epsilon(\omega_j)=z_{v_j}^\epsilon$. This holds for each $j$; thus, 
    \[
    s_\epsilon(\omega_1^{a_1}\dots \omega_n^{a_n}) = (z_{v_1}^\epsilon)^{a_1}\dots (z_{v_n}^\epsilon)^{a_n}.
    \]
    There is a unique standard monomial of $\NN^n$-degree $a_1\bfe_1 + \dots + a_n \bfe_n$ sitting under $\epsilon$, by Lemma~\ref{lem:sits-under-lemma}. It follows from the explicit description in Lemma~\ref{lem:proj-to-beta} of the effect of $s_\epsilon$ on the basis of standard monomials that $\omega_1^{a_1}\dots \omega_n^{a_n}$ contains this monomial with coefficient $1$. Since this logic applies for every facet $\epsilon$, and every standard monomial of $\NN^n$-degree $a_1\bfe_1 + \dots + a_n \bfe_n$ sits under some facet, all the standard monomials with that degree must appear in $\omega_1^{a_1}\dots \omega_n^{a_n}$ with coefficient $1$, so the lemma is proven. 
\end{proof}

\begin{lemma}\label{lem:cells-generate}
    The ring $\kk[\Lambda]$ is generated as a $\kk[\Omega]$-module by the elements $z_\alpha$, $\alpha \in \widehat P(\Lambda)$.
\end{lemma}

\begin{proof}
    This is also in \cite[Theorem~2.1]{garsia} in the case that $\Lambda$ is the chain complex of a ranked poset; the proof given here is essentially the same idea, transposed to the current setting. For any $\alpha\preceq \beta$ with $\alpha,\beta \in \widehat P(\Lambda)$, we claim that
    \begin{equation}\label{eq:alg-to-module}
    z_\alpha z_\beta = \left(\prod_{j\in J_\alpha}\omega_j\right)z_\beta.
    \end{equation}
    (The product on the right side of \eqref{eq:alg-to-module} is empty if $\alpha = \varnothing$.) Indeed, let $\epsilon$ be any facet that $\beta$ belongs to. Then $\epsilon$ has one vertex of every label, so the lower interval $[\varnothing,\epsilon]$ is poset-isomorphic to the boolean lattice of subsets of the label set; therefore, it contains a unique face with any given label set. In particular, since $\alpha \preceq \beta$, $\alpha$ is the only element of $[\varnothing,\epsilon]$ with label set $J_\alpha$. It follows that $z_\alpha$ is the only standard monomial of $\NN^n$-degree $\sum_{j\in J_\alpha} \bfe_j$ sitting under $\epsilon$. Meanwhile, by Lemma~\ref{lem:everything-of-degree}, $\prod_{j\in J_\alpha} \omega_j$ is the sum of every standard monomial of this degree, so in particular, it is the sum of $z_\alpha$ and some other monomials of this degree that, by Lemma~\ref{lem:sits-under-lemma}, do not sit under $\epsilon$. Since this logic applies for every $\epsilon$ to which $\beta$ belongs, it follows that
    \[
    \prod_{j\in J_\alpha}\omega_j = z_\alpha + \sum m,
    \]
    where the $m$'s in the sum are each standard monomials that do not sit under any facet to which $\beta$ belongs. It follows that $m z_\beta= 0$ for each $m$, and \eqref{eq:alg-to-module} holds after multiplying through by $z_\beta$.

    It follows from \eqref{eq:alg-to-module} by induction on the number of $z_\alpha$'s in a standard monomial, that any standard monomial belongs to the $\kk[\Omega]$-module generated by the $z_\alpha$'s. Since $\kk[\Lambda]$ is $\kk$-spanned by standard monomials, this shows it is contained in the $\kk[\Omega]$-span of the $z_\alpha$'s. We need to include $z_\varnothing$ to reach the monomial $1$.
\end{proof}

\begin{remark}
    It follows from Lemma~\ref{lem:cells-generate} (by way of Lemmas~\ref{lem:hironaka-decomposition} and \ref{lem:minimal-module-generating-set}) that, in the Cohen--Macaulay case, there exists a $\kk[\Omega]$-module basis for $\kk[\Lambda]$ consisting of generators $z_\alpha$ corresponding to cells $\alpha\in \widehat P(\Lambda)$ in the CW complex $\Lambda$. Such a basis is referred to in \cite{blum-smith} as a {\em cell basis}. By Lemma~\ref{lem:minimal-module-generating-set}, cell bases are precisely those subsets of $\{z_\alpha:\alpha \in \widehat P(\Lambda)\}$ whose images in the finite-dimensional $\kk$-algebra $\kk[\Lambda]/\Omega \kk[\Lambda]$ form $\kk$-vector space bases for it. It follows immediately that the collection of cell bases of a balanced Cohen-Macaulay boolean complex $\Lambda$ forms a linear matroid. 
\end{remark}

With this preparation, we can now start to lay out the core of the beautiful combinatorial picture that Garsia uncovered in \cite[Section~3]{garsia}, generalized to pure, balanced boolean complexes. For any subset $S\subset [n]$ of the collection of labels, one can form from $\Lambda$ the {\em label-selected subcomplex} $\Lambda_S$ consisting of those faces of $\Lambda$ whose label set is contained in $S$. Then $\Lambda_S$ is itself a pure, balanced boolean complex (with $S$ as its collection of labels)---see Figure~\ref{fig:subspaces-of-facet-space} in Section~\ref{sec:construct-a-basis} below for an example. Then $\Lambda$ is filtered by the subcomplexes $\Lambda_S$ over the poset of subsets $S\subset [n]$ (ordered by inclusion). 

Suppose, going forward, that $\Lambda$ has facets $\epsilon_1,\dots,\epsilon_m$. Note that $J_{\epsilon_1}= \dots = J_{\epsilon_m} = [n]$, so that $\deg_\mathrm{md} z_{\epsilon_1} = \dots = \deg_\mathrm{md} z_{\epsilon_m} = \bfe_1+ \dots + \bfe_n$, and in fact there are no other standard monomials with this degree in $\kk[\Lambda]$, so that 
\[
\kk[\Lambda]_{\bfe_1+\dots+\bfe_n}=\kk z_{\epsilon_1} \oplus \dots \oplus \kk z_{\epsilon_m},
\]
i.e., $z_{\epsilon_1},\dots,z_{\epsilon_m}$ is a {\em basis} for the homogeneous component of $\kk[\Lambda]$ of degree $\bfe_1+\dots+\bfe_n$.

\begin{definition}\label{def:facet-vector}
    For any $\alpha \in P(\Lambda)$, the {\em facet vector of $\alpha$ in $\Lambda$}, written $\bfv_\alpha^\Lambda$, is the $0,1$-vector with entries indexed by the facets $\epsilon_i$, with a $1$ in the $\epsilon_i$-entry if $\alpha$ belongs to $\epsilon_i$, and a $0$ there otherwise. For a set $\alpha_1,\dots,\alpha_\ell$ of elements of $P(\Lambda)$, their {\em incidence matrix in $\Lambda$} is the $0,1$-matrix with rows indexed by the $\alpha_i$'s and columns indexed by the $\epsilon_j$'s, with the $\alpha_i,\epsilon_j$-entry being $1$ if $\alpha_i$ belongs to $\epsilon_j$ and $0$ otherwise. (So the incidence matrix has the facet vectors of the $\alpha_i$'s as rows.) The ``in $\Lambda$" in these definitions, allowing the complex to vary, is because we will use them with the label-selected subcomplexes $\Lambda_S$, but this can be dropped when the complex is clear from context. 
\end{definition}

\begin{lemma}\label{lem:what-is-the-facet-vector}
    For any $\alpha \in P(\Lambda)$, the facet vector of $\alpha$ consists of the coordinates of the element 
    \[
    \left(\prod_{j\in [n]\setminus J_\alpha} \omega_j\right)z_\alpha \in \kk[\Lambda]_{\bfe_1+\dots+\bfe_n}
    \]
    with respect to the basis $z_{\epsilon_1},\dots,z_{\epsilon_m}$.
\end{lemma}

\begin{proof}
    By Lemma~\ref{lem:everything-of-degree}, the product $\prod_{j\in [n]\setminus J_\alpha} z_{\omega_j}$ consists of the sum of all the standard monomials of $\NN^n$-degree $\sum_{j\in [n]\setminus J_\alpha} \bfe_j$, which in turn are exactly those $z_\beta$'s, $\beta\in P(\Lambda)$, such that $J_\beta = [n]\setminus J_\alpha$.  For each $z_\beta$, the product $z_\beta z_\alpha$ is the sum of the $z_{\epsilon_j}$'s for the facets $\epsilon_j$ to which both $\alpha$ and $\beta$ belong. Because the sets of facets sitting over each of these $z_\beta$'s form a partition of $\{\epsilon_1,\dots,\epsilon_m\}$ by Lemma~\ref{lem:sits-under-lemma}, we conclude that the product in the lemma consists of the sum of exactly those $z_{\epsilon_j}$'s for $\epsilon_j$ to which $\alpha$ belongs. The statement of the lemma is saying exactly this.
\end{proof}

From Lemma~\ref{lem:what-is-the-facet-vector} it follows that, given a proposed basis $z_{\alpha_1},\dots,z_{\alpha_\ell}$ for $\kk[\Lambda]$ over $\kk[\Omega]$, a necessary condition for it to be indeed a basis is for its incidence matrix be square and nonsingular over $\kk$, as follows. The $\ell$ products appearing in Lemma~\ref{lem:what-is-the-facet-vector} if one takes $\alpha = \alpha_1,\dots,\alpha_\ell$, have $\kk$-span equal  precisely to the component of the $\kk[\Omega]$-module generated by $z_{\alpha_1},\dots,z_{\alpha_\ell}$ in degree $\bfe_1+\dots+\bfe_n$. These $\ell$ products need to be $\kk$-linearly independent for $z_{\alpha_1},\dots,z_{\alpha_\ell}$ to be $\kk[\Omega]$-linearly independent, and this component needs to be equal to $\kk[\Lambda]_{\bfe_1+\dots+\bfe_n}$ for them to generate $\kk[\Lambda]$ as $\kk[\Omega]$-module. We now give a criterion due to Garsia that is both necessary and sufficient.

\begin{prop}[essentially Theorem~3.1 in \cite{garsia}]\label{prop:garsia-3.1}
    Given a proposed basis $B:=z_{\beta_1},\dots,z_{\beta_\ell}$ for $\kk[\Lambda]$ as $\kk[\Omega]$-module, and a subset $S\subset [n]$, let $B_S$ be the subset of $B$ consisting of those $z_{\beta_i}$'s whose label sets $J_{\beta_i}$ are contained in $S$. Then $B$ is indeed a basis for $\kk[\Lambda]$ as $\kk[\Omega]$-module if and only if, for every subset $S\subset [n]$, the incidence matrix of $B_S$ in the label-selected subcomplex $\Lambda_S$ is square and nonsingular over $\kk$.
\end{prop}

\begin{proof}
    Because $z_{\beta_i}$ has degree $\sum_{j\in J_{\beta_i}} \bfe_j$, the submodule $\kk[\Omega] z_{\beta_i}\subset \kk[\Lambda]$ is zero in degree $\sum_{j\in S} \bfe_j$ unless $J_{\beta_i}\subset S$. Therefore, the only elements of $B$ that can contribute to generating this component of $\kk[\Lambda]$ over $\kk[\Omega]$ are the ones in $B_S$. We also have $\kk[\Lambda]_\mathbf{d}=\kk[\Lambda_S]_\mathbf{d}$ for any degree $\mathbf{d}\in \sum_{j\in S}\NN \bfe_j$. It follows that, for all $S\subset [n]$, the same argument as in the paragraph before the lemma, applied to the homogeneous component $\kk[\Lambda]_{\sum_{j\in S} \bfe_j}$ and the subcomplex $\Lambda_S$, shows that the condition that the incidence matrix of $B_S$ in $\Lambda_S$ be square and nonsingular, is necessary for $B$ to be a $\kk[\Omega]$-basis.

    Sufficiency is as follows. Assuming that the incidence matrix of $B_S$ is square and nonsingular in $\Lambda_S$ for each $S$, we show that $B$ is $\kk[\Omega]$-linearly independent and generates $\kk[\Lambda]$ over $\kk[\Omega]$.

    For $\kk[\Omega]$-module generation, consider any $z_\alpha$, $\alpha\in P(\Lambda)$, and set $S:=J_\alpha$. Then the square nonsingularity of the incidence matrix of $B_S$, together with Lemma~\ref{lem:what-is-the-facet-vector} with $\Lambda_S$ in the place of $\Lambda$, imply that $\kk[\Lambda]_{\sum_{j\in S}\bfe_j}$ lies in the $\kk[\Omega]$-span of $B_S\subset B$. In particular, $z_\alpha$ lies in this span. Since this holds for all $z_\alpha$, Lemma~\ref{lem:cells-generate} then implies that $B$ generates $\kk[\Lambda]$ as $\kk[\Omega]$-module.

    For linear independence, suppose there is a nontrivial relation $\sum_{i=1}^\ell f_i(\omega_1,\dots,\omega_n) z_{\beta_i} = 0$ for some choice of $n$-variate polynomials $f_1,\dots,f_\ell$ over $\kk$. We may take this relation to be $\NN^n$-graded without loss of generality, say of degree $a_1\bfe_1 + \dots + a_n\bfe_n$. Since $\deg_\mathrm{md} \omega_j = \bfe_j$ for each $j=1,\dots,n$, and the $\bfe_j$ are linearly independent in $\NN^n$, the space of homogeneous elements of $\kk[\Omega]$ of a given degree is $1$-dimensional. Thus each $f_i$ can be taken to be a monomial in the $\omega_j$'s. Because $\deg_\mathrm{md} z_{\beta_i}=\sum_{j\in J_{\beta_i}}\bfe_j$, it contributes at most one to each of the coefficients $a_1,\dots,a_n$. Thus, if there is any $a_j\geq 2$, it must be that every $f_i$ is divisible by $\omega_j$. Then we can cancel $\omega_j$ from the relation by Lemma~\ref{lem:torsion-free}. Proceeding inductively, we can arrive at a nontrivial relation in which $a_j\leq 1$ for all $j=1,\dots,n$. But taking $S$ to be the set of $j$'s for which $a_j=1$ in this relation, we obtain a nontrivial relation between the elements of $B_S$ occurring in the component $\kk[\Lambda]_{\sum_{j\in S} \bfe_j} = \kk[\Lambda_S]_{\sum_{j\in S}\bfe_j}$. This is ruled out by the square nonsingularity of the incidence matrix of $B_S$ in $\Lambda_S$, completing the proof.
\end{proof}

The next step is to inject all of the subspaces $\kk[\Lambda_S]_{\sum_{j\in S}\bfe_j}= \bigoplus_{\alpha:J_\alpha = S}\kk z_\alpha$ into the single subspace $\kk[\Lambda]_{\bfe_1+\dots+\bfe_n}=\kk z_{\epsilon_1}\oplus \dots \oplus \kk z_{\epsilon_m}$. Following Garsia's notation, for $S\subset [n]$ define
\begin{equation}\label{eq:def-of-M_S}
M_S := \left(\prod_{j\in [n]\setminus S} \omega_j\right)\left(\bigoplus_{\alpha:J_\alpha = S}\kk z_\alpha \right) \subset \kk z_{\epsilon_1}\oplus \dots \oplus \kk z_{\epsilon_m}.
\end{equation}
(Note that $M_{[n]} = \kk z_{\epsilon_1}\oplus \dots \oplus \kk z_{\epsilon_m}$ itself.) Lemma~\ref{lem:torsion-free} (stating that $\kk[\Lambda]$ is torsion-free over $\kk[\Omega]$) means that $M_S$ is isomorphic to $\bigoplus_{\alpha:J_\alpha = S}\kk z_\alpha$ as a $\kk$-vector space. We can also deduce that $T\subset S$ implies $M_T\subset M_S$, because in this case we have
\[
\left(\prod_{j\in S\setminus T} \omega_j\right)\left(\bigoplus_{\alpha:J_\alpha = T}\kk z_\alpha\right)\subset \bigoplus_{\alpha:J_\alpha = S} \kk z_\alpha,
\]
and therefore
\[
M_T  = \left(\prod_{j\in [n]\setminus S} \omega_j\right)\left(\prod_{j\in S\setminus T} \omega_j\right)\left(\bigoplus_{\alpha:J_\alpha = T}\kk z_\alpha \right) \subset \left(\prod_{j\in [n]\setminus S} \omega_j\right)\left(\bigoplus_{\alpha:J_\alpha = S}\kk z_\alpha \right)=M_S.
\]
We thus have a filtration of the finite-dimensional vector space $M_{[n]} = \kk z_{\epsilon_1}\oplus \dots \oplus \kk z_{\epsilon_m}$ by subspaces $M_S$, indexed in an inclusion-respecting way by the subsets $S\subset [n]$. An illustration is found in Figure~\ref{fig:subspace-filtration}.

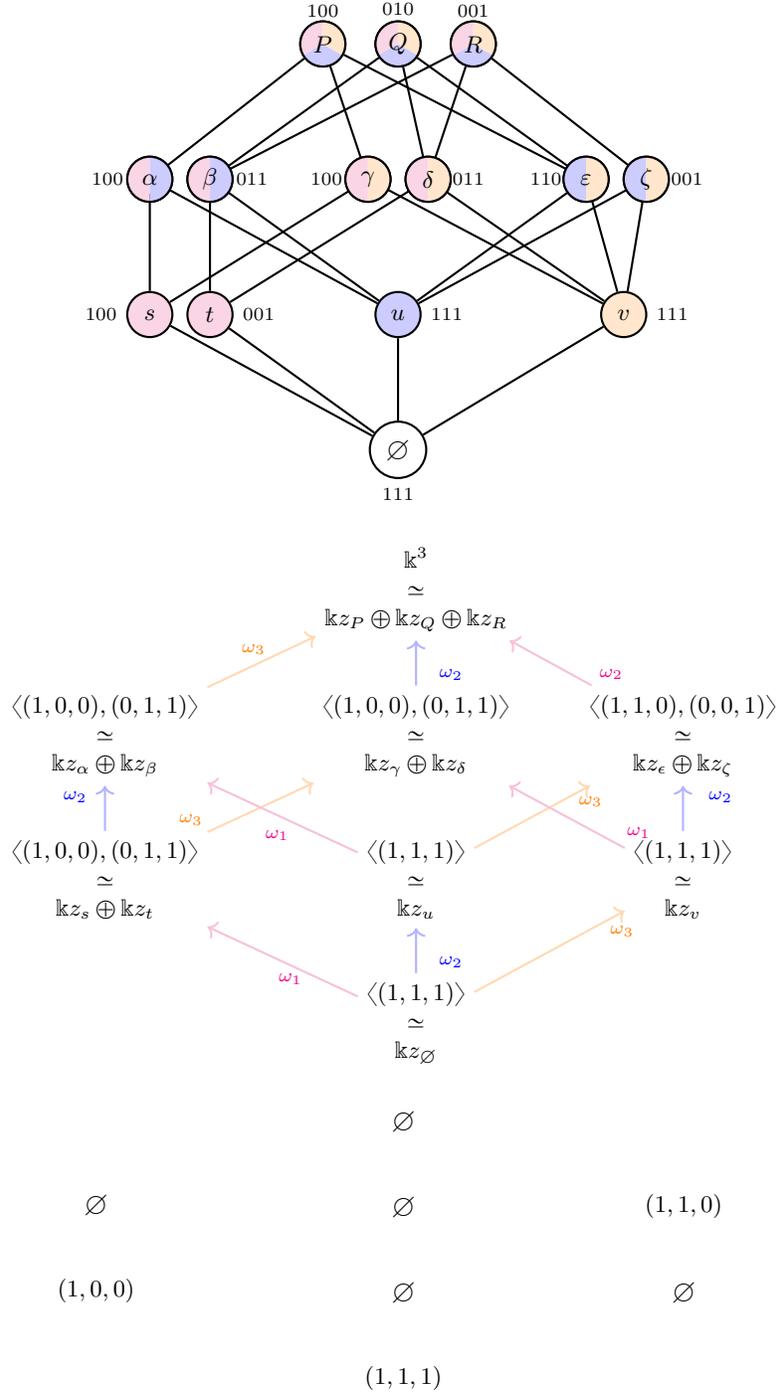
\begin{figure}
\newcommand{\halfcolorcircleOpacity}[6]{%
  \begin{scope}
    \path (#1) coordinate (center);
    \fill[#2!#5] (center) -- ++(90:0.3) arc(90:270:0.3) -- cycle;
    \fill[#3!#6] (center) -- ++(-90:0.3) arc(-90:90:0.3) -- cycle;
    \draw[thick, blue!20] (center) circle [radius=0.3];
    \node at (center) {#4};
  \end{scope}
}

\begin{center}
\begin{tikzpicture}[
  every node/.style={font=\small},
  magentanode/.style={circle, draw=black, fill=magenta!20, thick, minimum size=6mm},
  bluenode/.style={circle, draw=black, fill=blue!20, thick, minimum size=6mm},
  orangenode/.style={circle, draw=black, fill=orange!20, thick, minimum size=6mm},
  whitecircle/.style={circle, draw=black, fill=white, thick, minimum size=6mm},
  tricolornode/.style={circle, draw=black, thick, minimum size=6mm, fill=none}
]

\def\levelsep{1.8}

\node[whitecircle, label=below:{\scriptsize 111}] (empty) at (0,0) {$\emptyset$};

\node[magentanode, label=left:{\scriptsize 100}] (s) at (-3.3,\levelsep) {$s$};
\node[magentanode, label=right:{\scriptsize 001}] (t) at (-2.5,\levelsep) {$t$};
\node[bluenode, label=right:{\scriptsize 111}]     (u) at (0,\levelsep) {$u$};
\node[orangenode, label=right:{\scriptsize 111}]   (v) at (3,\levelsep) {$v$};

\foreach \x in {s,t,u,v} {
  \draw[thick] (empty) -- (\x);
}

\node[label=left:{\scriptsize 100}] at (-3.3, 2*\levelsep) (alpha) {$\alpha$};
\node[label=right:{\scriptsize 011}] at (-2.5, 2*\levelsep) (beta) {$\beta$};
\node[label=left:{\scriptsize 100}] at (-0.4, 2*\levelsep) (gamma) {$\gamma$};
\node[label=right:{\scriptsize 011}] at (0.4, 2*\levelsep) (delta) {$\delta$};
\node[label=left:{\scriptsize 110}] at (2.5, 2*\levelsep) (epsilon) {$\varepsilon$};
\node[label=right:{\scriptsize 001}] at (3.3, 2*\levelsep) (zeta) {$\zeta$};

\draw[thick ] (s) -- (alpha);
\draw[thick ] (u) -- (alpha);

\draw[thick ] (t) -- (beta);
\draw[thick ] (u) -- (beta);

\draw[thick ] (s) -- (gamma);
\draw[thick ] (v) -- (gamma);

\draw[thick ] (t) -- (delta);
\draw[thick ] (v) -- (delta);

\draw[thick ] (u) -- (epsilon);
\draw[thick ] (v) -- (epsilon);

\draw[thick ] (u) -- (zeta);
\draw[thick ] (v) -- (zeta);

\node[tricolornode] (P) at (-1,3*\levelsep) {};
\node[tricolornode] (Q) at (0,3*\levelsep) {};
\node[tricolornode] (R) at (1,3*\levelsep) {};


\node[label=above:{\scriptsize 100}] at (P) {\textbf{\small \(P\)}};
\node[label=above:{\scriptsize 010}] at (Q) {\textbf{\small \(Q\)}};
\node[label=above:{\scriptsize 001}] at (R) {\textbf{\small \(R\)}};

\draw[thick] (P) -- (alpha);
\draw[thick] (P) -- (epsilon);
\draw[thick] (P) -- (gamma);

\draw[thick] (Q) -- (beta);
\draw[thick] (Q) -- (epsilon);
\draw[thick] (Q) -- (delta);

\draw[thick] (R) -- (beta);
\draw[thick] (R) -- (zeta);
\draw[thick] (R) -- (delta);

\begin{scope}
  \foreach \pos/\leftcolor/\rightcolor/\name in {
    {alpha}/magenta/blue/{$\alpha$},
    {beta}/magenta/blue/{$\beta$},
    {gamma}/magenta/orange/{$\gamma$},
    {delta}/magenta/orange/{$\delta$},
    {epsilon}/blue/orange/{$\varepsilon$},
    {zeta}/blue/orange/{$\zeta$}
  } {
    \path (\pos) coordinate (center);
    \draw[thick, black] (center) circle (0.3);
    \fill[\leftcolor!20] (center) ++(90:0.3) arc(90:270:0.3) -- cycle;
    \fill[\rightcolor!20] (center) ++(270:0.3) arc(270:450:0.3) -- cycle;
    \draw[thick, black] (center) circle (0.3);
    \node at (center) {\name};
  }
\end{scope}

\foreach \name/\x in {P/-1, Q/0, R/1} {
  \coordinate (coord) at (\x,3*\levelsep);
  \begin{scope}
    \fill[magenta!20] (coord) -- ++(90:0.3) arc(90:210:0.3) -- cycle;
    \fill[blue!20]    (coord) -- ++(210:0.3) arc(210:330:0.3) -- cycle;
    \fill[orange!20]  (coord) -- ++(330:0.3) arc(330:450:0.3) -- cycle;
    \node at (coord) {\textbf{\small \(\name\)}};
    \draw[thick, black] (coord) circle [radius=0.3];
  \end{scope}
  \coordinate (\name) at (coord);
}

\end{tikzpicture}
\end{center}

\begin{center}
\begin{tikzpicture}[
  node distance=0.6cm and 2cm,
  every node/.style={align=center, font=\small},
  arrow/.style={->, thick},
  magentaarrow/.style={arrow, draw=magenta!30},
  bluearrow/.style={arrow, draw=blue!30},
  orangearrow/.style={arrow, draw=orange!30}
]

\node (L1) at (0,0) {
  $\langle (1,1,1) \rangle$\\
  $\simeq$\\
  $\kk z_{\emptyset}$
};

\node (L2a) [above left=of L1] {
  $\langle (1,0,0), (0,1,1) \rangle$\\
  $\simeq$\\
  $\kk z_s \oplus \kk z_t$
};
\node (L2b) [above=of L1] {
  $\langle (1,1,1) \rangle$\\
  $\simeq$\\
  $\kk z_u$
};
\node (L2c) [above right=of L1] {
  $\langle (1,1,1) \rangle$\\
  $\simeq$\\
  $\kk z_v$
};

\node (L3a) [above=of L2a] {
  $\langle (1,0,0), (0,1,1) \rangle$\\
  $\simeq$\\
  $\kk z_\alpha \oplus \kk z_\beta$
};
\node (L3b) [above=of L2b] {
  $\langle (1,0,0), (0,1,1) \rangle$\\
  $\simeq$\\
  $\kk z_\gamma \oplus \kk z_\delta$
};
\node (L3c) [above=of L2c] {
  $\langle (1,1,0), (0,0,1) \rangle$\\
  $\simeq$\\
  $\kk z_\epsilon \oplus \kk z_\zeta$
};

\node (L4) [above=of L3b] {
  $\kk^3$\\
  $\simeq$\\
  $\kk z_P \oplus \kk z_Q \oplus \kk z_R$
};

\draw[magentaarrow] (L1) -- node[left=3pt, near start] {\scriptsize\textcolor{magenta}{$\omega_1$}} (L2a);
\draw[bluearrow]    (L1) -- node[right=5pt, near start] {\scriptsize\textcolor{blue}{$\omega_2$}} (L2b);
\draw[orangearrow]  (L1) -- node[right=5pt, near end] {\scriptsize\textcolor{orange}{$\omega_3$}} (L2c);

\draw[bluearrow]    (L2a) -- node[left=3pt, near end] {\scriptsize\textcolor{blue}{$\omega_2$}} (L3a);
\draw[orangearrow]  (L2a) -- node[left=8pt, near start] {\scriptsize\textcolor{orange}{$\omega_3$}} (L3b);

\draw[magentaarrow] (L2b) -- node[left=8pt, near start] {\scriptsize\textcolor{magenta}{$\omega_1$}} (L3a);
\draw[orangearrow]  (L2b) -- node[right=3pt, near end] {\scriptsize\textcolor{orange}{$\omega_3$}} (L3c);

\draw[magentaarrow] (L2c) -- node[right=8pt, near start] {\scriptsize\textcolor{magenta}{$\omega_1$}} (L3b);
\draw[bluearrow]  (L2c) -- node[right=6pt, near end] {\scriptsize\textcolor{blue}{$\omega_2$}} (L3c);

\draw[orangearrow]  (L3a) -- node[left=5pt, near end] {\scriptsize\textcolor{orange}{$\omega_3$}} (L4);
\draw[bluearrow]    (L3b) -- node[right=5pt, near start] {\scriptsize\textcolor{blue}{$\omega_2$}} (L4);
\draw[magentaarrow] (L3c) -- node[right=7pt, near start] {\scriptsize\textcolor{magenta}{$\omega_2$}} (L4);

\end{tikzpicture}
\end{center}

\begin{center}
\begin{tikzpicture}[
  node distance=0.6cm and 2.7cm,
  every node/.style={align=center, font=\small, draw=none}
]

\node (L1) at (0,0) {
  $(1,1,1)$
};

\node (L2a) [above left=of L1] {
  $(1,0,0)$
};
\node (L2b) [above=of L1] {
  $\emptyset$
};
\node (L2c) [above right=of L1] {
  $\emptyset$
};

\node (L3a) [above=of L2a] {
  $\emptyset$
};
\node (L3b) [above=of L2b] {
  $\emptyset$
};
\node (L3c) [above=of L2c] {
  $(1,1,0)$
};

\node (L4) [above=of L3b] {
  $\emptyset$
};

\end{tikzpicture}
\end{center}
\caption{Top---the face poset of a Cohen--Macaulay balanced boolean complex $\Lambda$, with colors indicating the labels (magenta=$1$, violet=$2$, and orange=$3$), so $\omega_1=z_s+z_t$, $\omega_2 = z_u$, and $\omega_3=z_v$. By coincidence, $n=3=m$ in this example. Each face is also labeled by its facet vector.  Middle---the components of the ring $\kk[\Lambda]$ with $\NN^n$-degree of the form $\sum_{j\in S}\bfe_j$ for each $S\subset [n]$, along with the corresponding subspaces $\iota^{-1}(M_S)$, where $\iota$ is the  identification of the space $\kk^m$ of row vectors with the facet space $\kk z_P \oplus \kk z_Q \oplus \kk z_R$ that sends the standard basis to $z_P,z_Q,z_R$. The arrows (color-coded by label) show how multiplication by the $\omega_j$'s sends the components into each other. Bottom---for each $S\subset [n]$, a basis for a complement to $\iota^{-1}(\sum_{T\subsetneq S} M_T)$ in  $\iota^{-1}(M_S)$. The fact that they amalgamate to a basis for $\kk^m$ illustrates \eqref{eq:garsia-LA-CM}.}
\label{fig:subspace-filtration}
\end{figure}

Implicit in the approach to testing Cohen--Macaulayness laid out in \cite[Section~3]{garsia} is that Cohen--Macaulayness is equivalent to a purely linear-algebraic condition on the way that the subspaces of this filtration interact with each other. The following theorem is the objective of this section. To state it, we note that because by Lemma~\ref{lem:torsion-free}, multiplication by $\prod_{j\in [n]\setminus S} \omega_j$ is a bijection from $\bigoplus_{\alpha:J_\alpha = S} \kk z_\alpha = \kk[\Lambda]_{\sum_{j\in S} \bfe_j}$ to $M_S$, there is a well-defined inverse map, which we denote by 
\begin{equation}\label{eq:def-of-inverse-omega}
\left(\prod_{j\in [n]\setminus S} \omega_j\right)^{-1}: M_S \rightarrow \kk[\Lambda]_{\sum_{j\in S} \bfe_j}.
\end{equation}

\begin{theorem}[Implicit in Theorems~3.1--3.3 in \cite{garsia}]\label{thm:garsia-LA-CM}
    For each $S\subset [n]$, choose a vector space complement $L_S$ to $\sum_{T\subsetneq S} M_T$ in $M_S$. Then $\Lambda$ is Cohen--Macaulay over the field $\kk$ if and only if
    \begin{equation}\label{eq:garsia-LA-CM}
    \kk z_{\epsilon_1}\oplus \dots \oplus \kk z_{\epsilon_m} = \bigoplus_{S\subset [n]} L_S.
    \end{equation}
    Furthermore, in all cases, if for each $S\subset [n]$ we take a $\kk$-basis $B(L_S)$ for $L_S$, then a minimal $\NN^n$-homogeneous $\kk[\Omega]$-module generating set for $\kk[\Lambda]$ is obtained from the union $B$ of the sets
    \begin{equation}\label{eq:min-gen-set}
    \left(\prod_{j\in [n]\setminus S} \omega_j\right)^{-1} B(L_S) \subset \kk[\Lambda]_{\sum_{j\in S}\bfe_j} 
    \end{equation}
    over $S\subset[n]$, and in the Cohen--Macaulay case, this minimal generating set is a basis.
\end{theorem}

We give the proof after discussing Figure~\ref{fig:subspace-filtration}, which illustrates the  picture of $\kk[\Lambda]$ underlying the theorem. The top image is the Hasse diagram of the face poset of a balanced Cohen--Macaulay complex $\Lambda$, with each face's label set indicated by colors. Next to each face, its facet vector appears. (The geometric realization of $\Lambda$ is homeomorphic to a disk. We have omitted an image of the complex itself; it can be found in \cite[Figure~2.16]{blum-smith}. This complex results from taking the quotient of the barycentric subdivision of the boundary of a tetrahedron by a dihedral group of order 8.)

The middle image shows the components \[
\kk[\Lambda]_{\sum_{j\in S}\bfe_j} = \bigoplus_{\alpha: J_\alpha = S} z_\alpha
\]
and their corresponding images $M_S$ in the facet space $\kk[\Lambda]_{\bfe_1+\dots + \bfe_n} = \kk z_{\epsilon_1} \oplus \dots \oplus \kk z_{\epsilon_m}$ (after identifying the latter with the space $\kk^m$ of row vectors via the map $\iota$ that sends the standard basis for $\kk^m$ to $z_{\epsilon_1},\dots,z_{\epsilon_m}$). The arrows indicate the way that multiplication by the $\omega_j$'s injects these components into each other. As an illustration, consider the component $\kk z_s \oplus \kk z_t$, with label set $\{1\}$. One must multiply by $\omega_2\omega_3 = z_uz_v$ to get $z_s$ and $z_t$ into the facet component $\kk z_P \oplus \kk z_Q \oplus \kk z_R$. We have
\[
\omega_2\omega_3 z_s = z_P\text{ and }\omega_2\omega_3 z_t = z_Q + z_R,
\]
which  corresponds (via Lemma~\ref{lem:what-is-the-facet-vector}) to the fact that $s$'s facet vector is $(1,0,0)$ and $t$'s facet vector is $(0,1,1)$. Thus, 
\[
M_{\{1\}} = \kk z_P + \kk (z_Q+z_R),
\]
which is identified via $\iota^{-1}$ with the subspace $\langle (1,0,0), (0,1,1)\rangle$ in the space $\kk^m$ of row vectors.

The bottom image in Figure~\ref{fig:subspace-filtration} shows the bases $B(L_S)$ discussed in Theorem~\ref{thm:garsia-LA-CM} for the complements $L_S$ to $\sum_{T\subsetneq S} M_T$ in $M_S$, written as row vectors via $\iota^{-1}$. To illustrate, consider the $S=\{1\}$ component again. We computed above that $M_{\{1\}} = \kk z_P + \kk (z_Q+z_R) \cong \langle (1,0,0), (0,1,1)\rangle$. Meanwhile, $\sum_{T\subsetneq \{1\}} M_T = M_\varnothing = \kk(\omega_1\omega_2\omega_3\cdot 1) = \kk(z_P + z_Q + z_R) \cong \langle (1,1,1)\rangle$. A complement to the latter in the former is spanned by $z_P = \iota(1,0,0)$, and this is illustrated in the image at the location corresponding to the label set $\{1\}$. (One could also have taken for a complement the span of $z_Q+z_R$, or of any linear combination of $z_P,z_Q+z_R$ other than the sum.) Meanwhile, at the location corresponding to the label set $\{2\}$ one finds $\varnothing$ because $M_{\{2\}} = \kk \omega_1\omega_3z_u = \kk(z_P+z_Q+z_R)$ is already exhausted by $\sum_{T\subsetneq \{2\}} M_T = M_\varnothing$. Because the complex is Cohen--Macaulay, the key equality \eqref{eq:garsia-LA-CM} of Theorem~\ref{thm:garsia-LA-CM} is illustrated by the fact that amalgamating the bases $B(L_S)$ that appear in the bottom figure yields a basis for $\kk^m$.

\begin{proof}[Proof of Theorem~\ref{thm:garsia-LA-CM}]
    We argue the second statement (about the union $B$ of the sets in \eqref{eq:min-gen-set}) first, and then use this to show that the condition \eqref{eq:garsia-LA-CM} is equivalent to Cohen--Macaulayness.

    For a fixed $S\subset [n]$, consider the homogeneous component of degree $\sum_{j\in S}\bfe_j$ in the quotient ring $\kk[\Lambda]/\Omega\kk[\Lambda]$. This is the vector space quotient of the component $\kk[\Lambda]_{\sum_{j\in S}\bfe_j}$ by its intersection with the ideal generated by the $\omega_j$'s. The latter ideal intersection can be computed by summing the images, in $\kk[\Lambda]_{\sum_{j\in S}\bfe_j}$, of every strictly ``lower degree" component $\kk[\Lambda]_{\sum_{j\in T}\bfe_j}$, for $T\subsetneq S$, under multiplication by the $\omega_j$'s that put it in $\kk[\Lambda]_{\sum_{j\in S}\bfe_j}$. Thus, we have
    \begin{equation}\label{eq:ideal-in-homog-component}
    \Omega \kk[\Lambda] \cap \kk[\Lambda]_{\sum_{j\in S} \bfe_j} = \sum_{T\subsetneq S} \left(\prod_{j\in S\setminus T} \omega_j\right) \kk[\Lambda]_{\sum_{j\in T}\bfe_j}.
    \end{equation}
    Therefore, a vector space basis for the $\sum_{j\in S}\bfe_j$-component in the quotient $\kk[\Lambda]/\Omega \kk[\Lambda]$ can be computed by taking the image of a basis of a complement to \eqref{eq:ideal-in-homog-component} in $\kk[\Lambda]_{\sum_{j\in S}\bfe_j}$.
    
    We now argue that for each $S$, \eqref{eq:min-gen-set} is exactly such a basis of a complement to \eqref{eq:ideal-in-homog-component} in $\kk[\Lambda]_{\sum_{j\in S}\bfe_j}$. Indeed, because multiplication by $\prod_{j\in [n]\setminus S}\omega_j$ is an injection (Lemma~\ref{lem:torsion-free}), it sends $\kk[\Lambda]_{\sum_{j\in S}\bfe_j}$ isomorphically to $M_S$, and \eqref{eq:ideal-in-homog-component} isomorphically to 
    \[
    \sum_{T\subsetneq S} \left(\prod_{j\in [n]\setminus S}\omega_j\right)\left(\prod_{j\in S\setminus T} \omega_j\right) \kk[\Lambda]_{\sum_{j\in T}\bfe_j} = \sum_{T\subsetneq S} M_T.
    \]
    Therefore, the inverse map $(\prod_{j\in [n]\setminus S}\omega_j)^{-1}$ sends $L_S$ to a complement of \eqref{eq:ideal-in-homog-component} in $\kk[\Lambda]_{\sum_{j\in S}\bfe_j}$, so it sends $B(L_S)$ to a basis for such a complement.

    By Lemma~\ref{lem:cells-generate}, the ring $\kk[\Lambda]/\Omega \kk[\Lambda]$ is $\kk$-spanned by the elements $z_\alpha$, $\alpha\in P(\Lambda)$, which have $\NN^n$-degrees of the form $\sum_{j\in S}\bfe_j$ for $S\subset [n]$. It follows that this quotient ring is zero in all components of degree not of this form, i.e.,
    \[
    \kk[\Lambda]/\Omega\kk[\Lambda] = \bigoplus_{S\subseteq [n]} (\kk[\Lambda]/\Omega\kk[\Lambda])_{\sum_{j\in S}\bfe_j}.
    \] 
    From this and the previous paragraph, it is immediate that the image in $\kk[\Lambda]/\Omega\kk[\Lambda]$ of the union $B$ of the sets in \eqref{eq:min-gen-set} forms a $\kk$-basis for this quotient, and it follows from Lemma~\ref{lem:minimal-module-generating-set} that this union $B$ is a minimal generating set for $\kk[\Lambda]$ as a $\kk[\Omega]$-module, as claimed. If $\Lambda$ is Cohen--Macaulay, then $B$ is itself a $\kk[\Omega]$-module basis, by Lemma~\ref{lem:hironaka-decomposition}. This completes the proof of the statement about \eqref{eq:min-gen-set} in the theorem.

    We now argue that Cohen--Macaulayness of $\Lambda$ is equivalent to \eqref{eq:garsia-LA-CM}. First, again by Lemma~\ref{lem:hironaka-decomposition} and Lemma~\ref{lem:minimal-module-generating-set}, Cohen--Macaulayness of $\Lambda$ is equivalent to the linear independence over $\kk[\Omega]$ of the union $B$ of the sets in \eqref{eq:min-gen-set}. Now, such linear independence immediately implies \eqref{eq:garsia-LA-CM}, because the $L_S$'s are precisely the spans of the images of the sets in \eqref{eq:min-gen-set} under multiplication by certain elements of $\kk[\Omega]$; thus if the sum in \eqref{eq:garsia-LA-CM} is not direct, then there is a linear relation over $\kk[\Omega]$ between some of the elements of $B$.
    
    Conversely, if there is a nontrivial linear relation over $\kk[\Omega]$ anywhere in $\kk[\Lambda]$ between the elements of $B$, which we can take to be $\NN^n$-homogeneous, then it can be witnessed by a linear relation occurring in the $\kk[\Lambda]_{\bfe_1+\dots+\bfe_n}$ component, as follows. First, it implies a nontrivial linear relation in some component of the form $\kk[\Lambda]_{\sum_{j\in S}\bfe_j}$, by the same logic as in the final paragraph of the proof of Proposition~\ref{prop:garsia-3.1}. Then, this relation can be multiplied through by $\prod_{j\in [n]\setminus S} \omega_j$ to put it in $\kk z_{\epsilon_1} + \dots + \kk z_{\epsilon_m}$, where it remains nontrivial by a final call to Lemma~\ref{lem:torsion-free}. This completes the proof.
\end{proof}

We pull out as a corollary one of the intermediate steps in this proof:

\begin{cor}\label{cor:B-is-a-basis-downstairs}
Whether $\Lambda$ is Cohen--Macaulay or not, using the notation of the statement of Theorem \ref{thm:garsia-LA-CM}, the image of $B$ in the quotient $ \kk[\Lambda]/\Omega\kk[\Lambda]$ is a $\kk$-vector space basis for this quotient. \qed
\end{cor}

For later use, we also draw out an implication:

\begin{observation}\label{obs:M_S-doesn't-meet-the-other-L_T's}
    From the definition of $L_S$ in the statement of Theorem~\ref{thm:garsia-LA-CM} and induction on the cardinality of $S\subset[n]$, it is immediate that 
    \[
    M_S = \sum_{T\subseteq S} L_T.
    \]
    Therefore, by \eqref{eq:garsia-LA-CM}, if  $\Lambda$ is Cohen--Macaulay then we must have
    \[
    M_S = \bigoplus_{T\subseteq S} L_T
    \]
    and 
    \[
    M_S \cap \left(\bigoplus_{T\nsubseteq S} L_T\right)=\{0\}
    \]
    for any $S\subset [n]$.
\end{observation}

\begin{remark}
    While Theorem~\ref{thm:garsia-LA-CM} is essentially proven in \cite{garsia} (in the situation that $\Lambda$ is the order complex of a ranked poset), the presentation there leaves it in the background, while emphasizing another criterion of Cohen--Macaulayness that we have chosen to leave in the background. One can define the {\em fine $f$-vector} of $\Lambda$---see \cite{stanley1979balanced}---a function $2^{[n]}\rightarrow \NN$ given by
    \[
    f_\Lambda(S) := \#\{\alpha \in \widehat P(\Lambda): J_\alpha = S\}
    \]
    for $S\subseteq [n]$, and then the {\em fine $h$-vector}, related to it by an inclusion-exclusion formula:
    \[
    h_\Lambda(S) := \sum_{T\subseteq S} (-1)^{\# S - \# T}f_\Lambda(T).
    \]
    (These are referred to in \cite{garsia} as $\alpha(S)$ and $\beta(S)$ respectively.) The fine $h$-vector predicts, for each $S$, the number of elements of $\NN^n$-degree $\sum_{j\in S}\bfe_j$ in an $\NN^n$-homogeneous $\kk[\Omega]$-basis of $\kk[\Lambda]$, if it exists. Theorem~3.2  of \cite{garsia} compares a proposed basis against these predicted numbers: if it has the predicted number of elements in each $\NN^n$-degree, and its incidence matrix in $\Lambda$ is nonsingular over $\kk$, then it is a basis (see \cite[Theorem~3.2]{garsia} and \cite[Theorem~5.2]{garsia-stanton}, and the arguments generalize to the present setting). Meanwhile, Theorem~3.3 of \cite{garsia} tests Cohen--Macaulayness by doing row-reduction on the facet vectors of all the faces of $\Lambda$ until a candidate basis is found, and then seeing if it has the predicted numbers of elements of each $\NN^n$-degree. (The order in which the row reduction is carried out is important; see Section~\ref{sec:construct-a-basis} below.)
\end{remark}

\section{The counterexample}\label{sec:counterex}

In this section we prove Theorem~\ref{thm:no-equivariant-iso}. Recall that it concerns the $d$-simplex $\Delta :=\Delta_d$, for a natural number $d\geq 2$, and the group $G := \Aut(\Delta)$ of automorphisms of $\Delta$ (as a simplicial complex). Note that $G \cong \mathfrak{S}_n$, the full symmetric group on $n$ letters, where $n:=d+1$, because any permutation of the $n$ vertices of $\Delta$ extends uniquely to an automorphism of $\Delta$. We have the barycentric subdivision $\Sd \Delta$, and we claim that if $\operatorname{char} \kk = 2$, there is no $G$-equivariant isomorphism as modules over the parameter ring $\kk[\Theta]$. We remind the reader that the $\kk[\Theta]$-module structure on $\kk[\Sd\Delta]$ is defined, per Setup~\ref{set:theta-module}, by identifying $\kk[\Theta]\subset\kk[\Delta]$ with $\kk[\Gamma]\subset \kk[\Sd\Delta]$ along the $\kk$-algebra isomorphism 
\begin{align*}
\Psi:\kk[\Gamma]&\rightarrow \kk[\Theta]\\
p(\gamma_1,\dots,\gamma_n)&\mapsto p(\theta_1,\dots,\theta_n),
\end{align*}
where $p$ is an arbitrary $n$-variate polynomial over $\kk$. We use this freely in what follows.

The main idea of the proof is to hypothesize an equivariant isomorphism, which must then also induce an isomorphism on the $\mathfrak{A}_n$-invariant subrings (as modules over the parameter ring $\kk[\Theta]$), and to show that this leads to a contradiction by directly examining module bases for the $\mathfrak{A}_n$-invariants in the two rings. The details are as follows. We first articulate some key lemmas; all are straightforward, well-known, or both.

Without loss of generality, let $V:=\{0,\dots,d\}$ be the vertex set of $\Delta =\Delta_d$. Then $\kk[\Delta]$ is the polynomial ring $\kk[x_0,\dots,x_d]$, and for $j = 1,\dots, n$, the parameter $\theta_j$ is the elementary symmetric polynomial of degree $j$ in the indeterminates $x_0,\dots,x_d$. 

\begin{lemma}\label{lem:Sn-invariants-Delta}
    The parameter subring $\kk[\Theta]$ coincides with the invariant ring $\kk[\Delta]^{\mathfrak{S}_n}$.
\end{lemma}

\begin{proof}
    This is the fundamental theorem on symmetric polynomials (FTSP).
\end{proof}

Meanwhile, the generators $y_\alpha$ of $\kk[\Sd\Delta]$ are indexed by nonempty subsets $\alpha \subset V$, and the parameters $\gamma_j$ are sums of these generators across $j$-subsets:
\[
\gamma_j = \sum_{\alpha\in \binom{V}{j}} y_\alpha.
\]
\begin{lemma}\label{lem:Sn-invariants-SdDelta}
Again, the parameter ring $\kk[\Gamma]$ coincides with the invariant ring $\kk[\Sd\Delta]^{\mathfrak{S}_n}$.
\end{lemma} 

\begin{proof}
It is clear that $\kk[\Gamma]\subseteq \kk[\Sd\Delta]^{\mathfrak{S}_n}$. For the reverse, any element of $\kk[\Sd\Delta]^{\mathfrak{S}_n}$ is a linear combination of $\mathfrak{S}_n$-orbit sums of standard monomials $y_{\alpha_1}^{c_1}y_{\alpha_2}^{c_2}\dots y_{\alpha_r}^{c_r}$ with $\alpha_1\subsetneq \alpha_2\subsetneq \dots\subsetneq \alpha_r$ a chain in the poset of subsets of $\{0,\dots,d\}$. Because $\mathfrak{S}_n$ acts transitively on the chains of this poset, such an orbit sum consists of all monomials of the given shape $c_1(1^{|\alpha_1|})+ \dots + c_r(1^{|\alpha_r|})$. Thus it lies in $\kk[\Gamma]$ by Proposition~\ref{prop:parameter-monom-every-term}.
\end{proof} 

\begin{remark}
This result is already implicit in \cite{garsia-stanton}. When $\Char \kk = 0$, it is a special case of \cite[Theorem~7.4]{garsia-stanton}. The proof sketched here is written out carefully in  \cite[Proposition~2.5.72]{blum-smith}. It is identical in spirit to the classical Gauss proof of the FTSP (found in \cite[Paragraphs 3--5]{gauss}), and the computations implied by the proof are shorter, with a single calculation replacing an induction.  Indeed, the classical FTSP can be proven by starting with the result for $\kk[\Sd\Delta]$, and applying induction on the shape of monomials, precisely as in the proof of  Theorem~\ref{thm:transfer-bases}---this is carried out explicitly in \cite[Theorem~2.5.74]{blum-smith}, but it can be viewed from a certain point of view as nothing other than what the Gauss proof was already doing (see \cite[Remark~2.5.75]{blum-smith} and   \cite{blum2017fundamental}).
\end{remark}

We also need some information valid in characteristic $2$ about the $\mathfrak{A}_n$-invariants in $\kk[\Delta]$ and $\kk[\Sd\Delta]$.

\begin{lemma}\label{lem:basis-Delta-An}
    The ring $\kk[\Delta]^{\mathfrak{A}_{n}}$ is a free $\kk[\Theta]$-module of rank two, with basis $1, D$, where
\[
D := \sum_{g\in \mathfrak{A}_n} gm,
\]
the $\mathfrak{A}_n$-orbit sum of the monomial $m\in \kk[\Delta]$ defined by
\[
m := x_1x_2^2\dots x_d^d.
\]
\end{lemma}

\begin{proof}
    This is well-known, but it is written down carefully for example in \cite[Lemma~5.4.1]{biesel}.
\end{proof}

\begin{lemma}\label{lem:basis-SdDelta-An}
Similarly, $\kk[\Sd\Delta]^{\mathfrak{A}_n}$ is a free $\kk[\Theta]$-module of rank two with basis $1,\widehat D$, where $\widehat D$ is the $\mathfrak{A}_n$-orbit sum of the $\Garsia$-preimage $\widehat m$ of $m$, namely
\[
\widehat m := y_{\{d\}} y_{\{d-1,d\}} \dots y_{[d]}.
\]
\end{lemma}

\begin{proof}[Proof sketch]
This can be seen for example by shelling the quotient by $\mathfrak{A}_n$ of the Coxeter complex of $G$, as in \cite[Theorem~4.3.5]{reiner} (with $W=G$ and $E'=\mathfrak{A}_n$), and applying \cite[Theorem~6.2]{garsia-stanton} (which is stated over a field of characteristic zero, but that hypothesis is not required in the proof of this claim).\end{proof}

Finally, we will use the following elementary calculation, for which we replace the ground field $\kk$ with $\ZZ$; the definitions of $\theta_1,\dots,\theta_d$ are modified accordingly.

\begin{lemma}\label{lem:cross-term}
For \( d \geq 2 \), in the decomposition of \( \theta_1 \cdots \theta_d \) into sums of monomials over $\ZZ$, the monomial 
\[ x_0x_1x_2\prod_{i=2}^{d-1}(x_0\cdots x_i) 
\]
appears with coefficient \( 3 \). (Note: in the $d=2$ case, the product is empty.)
\end{lemma}

\begin{proof}
    We proceed by induction on \( d \). It is convenient for the sake of this induction to work in the ring $\Lambda$ of symmetric functions (i.e., the $\ZZ$-algebra of bounded-degree power series in countably many indeterminates $x_0,x_1,\dots$ that are invariant under all permutations of the indeterminates), so that we do not have to concern ourselves with the number of variables, only the number of factors. It is well-known that $\Lambda$ is a polynomial ring generated by the elementary symmetric functions $e_1 = x_0+x_1+\dots$, $e_2=x_0x_1 + x_0x_2 + x_1x_2 + \dots$, etc.\footnote{This statement is really just the FTSP. For background on $\Lambda$, see \cite{macdonald} or \cite{stanley}.} Via the $\ZZ$-algebra homomorphism that sends $\Lambda=\ZZ[e_1,e_2,\dots]$ to $\ZZ[\Delta]^{\mathfrak{S}_n}$ by mapping $e_i\mapsto \theta_i$ for $i=1,\dots,n$ and $e_i\mapsto 0$ if $i>n$, proving the claim for $e_1\dots e_d$ in $\Lambda$ will yield the same result for $\theta_1\dots\theta_d$ in $\ZZ[\Delta]$.

    The base case \( d = 2 \) can be seen by direct computation:
    \begin{align*}
    e_1e_2
    &=(x_0+x_1+\dots)(x_0x_1 + x_0x_2 + \dots )\\
    &= (x_0^2x_1 + \dots) + 3(x_0x_1x_2 + \dots).
    \end{align*}
    
    Now, suppose the result is true for some integer \( d \). We prove that it remains true for \( d+1 \): in the product $\prod_{i=1}^{d+1}e_i = e_{d+1}\prod_{i=1}^de_i$, the monomial \( x_0x_1x_2\prod_{i=2}^{d}(x_0\cdots x_i) \) can only be obtained from a product of the monomial \( x_0\cdots x_{d} \) of \( e_{d+1} \), which occurs with coefficient 1, with the monomial \( x_0x_1x_2\prod_{i=2}^{d-1}(x_0\cdots x_i) \) of \( \prod_{i=1}^{d} e_i \), which occurs with coefficient $3$ by induction. So it too occurs with coefficient 3.
\end{proof}

Now we can prove the main result of the section.

\begin{proof}[Proof of Theorem~\ref{thm:no-equivariant-iso}]

By way of contradiction, suppose that $\kk$ has characteristic two, and let
\[
\Phi : \kk[\Sd\Delta]\rightarrow \kk[\Delta]
\]
be a $G=\mathfrak{S}_n$-equivariant isomorphism of $\kk[\Theta]$-modules. Equivariance implies that $\Phi$ induces an isomorphism
\[
\phi :  \kk[\Sd\Delta]^{\mathfrak{A}_n} \rightarrow \kk[\Delta]^{\mathfrak{A}_n}
\]
of $\kk[\Theta]$-modules. Furthermore, because $\mathfrak{A}_n\subset G$ is normal, the actions of $G$ restrict to actions on these subrings, which factor through $G/\mathfrak{A}_n\cong C_2$. The $G$-equivariance of $\Phi$ implies that the restricted map $\phi$ is $C_2$-equivariant. We will show that for $n \geq 3$, no $C_2$-equivariant $\kk[\Theta]$-module isomorphism $\kk[\Sd\Delta]^{\mathfrak{A}_n}\rightarrow\kk[\Delta]^{\mathfrak{A}_n} $ exists; this will be the desired contradiction. Let $\tau$ be the nontrivial element in $G/\mathfrak{A}_n \cong C_2$.

In view of Lemmas~\ref{lem:basis-Delta-An} and \ref{lem:basis-SdDelta-An}, we write the images of $1,\widehat D\in \kk[\Sd\Delta]^{\mathfrak{A}_n}$ under $\phi$ on the $\kk[\Theta]$-basis $1, D\in \kk[\Delta]^{\mathfrak{A}_n}$:
\begin{align*}
\phi(1) &= s + t D\\
\phi(\widehat D) &= u + vD,
\end{align*}
where $s,t,u,v\in \kk[\Theta]$. By the fact that $\phi$ is equivariant, we can immediately conclude $t=0$ (because the action of $\tau$ is trivial on $1\in \kk[\Delta]^{\mathfrak{A}_n}$, but is not trivial on $D$). Then, because $\kk[\Gamma]$ is the $\kk[\Theta]$-span of $1\in \kk[\Sd\Delta]^{\mathfrak{A}_n}$, $\phi(\kk[\Gamma])$ has the form $\kk[\Theta]s$, the ideal generated by $s$ in $\kk[\Theta]$. Because $\phi$ must restrict to a $\kk[\Theta]$-module isomorphism from  $\kk[\Sd\Delta]^{\mathfrak{S}_{n}} = \kk[\Gamma]$ to $\kk[\Delta]^{\mathfrak{S}_{n}} = \kk[\Theta]$, and in particular this restriction is surjective, we conclude $s$ must generate the unit ideal in $\kk[\Theta]$, and thus $s$ is an element of $\kk^\times$. (The equalities in the last sentence are in view of Lemmas~\ref{lem:Sn-invariants-Delta} and \ref{lem:Sn-invariants-SdDelta}.)

Again by the fact that $\phi$ is a $G/\mathfrak{A}_n \cong C_2$-equivariant $\kk[\Theta]$-module map, we have
\[
\phi(\tau \cdot \widehat D) = \tau\cdot\phi(\widehat D)=u + v(\tau\cdot D)\in\kk[\Delta].
\]
Then 
\begin{equation}\label{eq:first-phi}
\phi(\widehat D + \tau \cdot \widehat D) = (u+v D) + (u + v(\tau \cdot  D)) = v(D + \tau \cdot D),
\end{equation}
recalling that the characteristic of $\kk$ is $2$.

Because $\widehat D$ is the $\mathfrak{A}_n$-orbit sum of the monomial $\widehat m$ defined above, and $\tau$ is the nontrivial coset of $\mathfrak{A}_n$ in $G$, we see that $\widehat D + \tau \widehat D$ is the $G = \mathfrak{S}_{n}$-orbit sum of the same monomial. By Proposition~\ref{prop:parameter-monom-every-term} or by direct computation, this is equal to $\gamma_1\gamma_2\dots\gamma_d$. Thus, 
\[
\Psi(\widehat D + \tau \cdot \widehat D) = \theta_1\theta_2\cdots \theta_d.
\]
Therefore,
\begin{equation}\label{eq:second-phi}
\phi(\widehat D + \tau\cdot \widehat D) = \Psi(\widehat D+\tau \cdot \widehat D)\phi(1) = \theta_1\theta_2\cdots \theta_ds.
\end{equation}
Combining \eqref{eq:first-phi} and \eqref{eq:second-phi}, we find that
\begin{equation}\label{eq:gamma-is-a-factor}
\theta_1\theta_2\cdots \theta_ds = v(D + \tau \cdot D).
\end{equation}

We will now derive the promised contradiction. Since $s\in \kk^\times$, equation \eqref{eq:gamma-is-a-factor} asserts precisely that $D + \tau \cdot D$ is a factor of $\theta_1\theta_2\dots \theta_{d}$ in the polynomial ring $\kk[\Theta]$. They are of the same degree, so this means they differ by a scalar factor. In fact, the terms of $D+\tau\cdot D$ are a proper subset of the terms of $\theta_1\theta_2\dots\theta_d$: expanding everything into monomials, $D + \tau \cdot D$ consists precisely of the terms of the product $\theta_1\theta_2\dots \theta_d$ that stack up, i.e., those of shape $(d,\dots,2,1)$, by Proposition~\ref{prop:param-monom-x-ring}. So to contradict \eqref{eq:gamma-is-a-factor}, one only has to check that $\theta_1\theta_2\dots\theta_d$ has at least one cross-term (i.e., a term of shape strictly dominated by $(d,\dots,2,1)$) that is nonzero in $\kk$, i.e., has an odd coefficient. One such cross-term is furnished by Lemma~\ref{lem:cross-term}. This completes the proof.
\end{proof}

\section{The positive result}\label{sec:positive}

In this section, we prove Theorem~\ref{thm:yes-CM-in-coprime}, stating that, in spite of the negative result in Section~\ref{sec:counterex}, there is guaranteed to exist a $G$-equivariant $\kk[\Theta]$-module isomorphism $\Phi:\kk[\Sd\Delta]\rightarrow \kk[\Delta]$ in the best-case scenario where $\Delta$ is Cohen--Macaulay and $\Char \kk$ does not divide the order of $G$, and furthermore, it can be constructed explicitly. As mentioned in the introduction, the Cohen--Macaulay assumption already renders it automatic that $\kk[\Sd\Delta]$ and $\kk[\Delta]$ are isomorphic as $\kk[\Theta]$-modules, being free of the same rank; the work is to show that an isomorphism can be taken to be $\Aut(\Delta)$-equivariant.

The existence statement is proven two ways: it follows from the explicit construction, which is based on the ideas developed in Section~\ref{sec:garsia}, but we also include a nonconstructive proof that hews closely to ideas in \cite{adams-reiner} and was developed in conversation with V. Reiner. We give the nonconstructive proof in Section~\ref{sec:nonconstructive}, and the constructive proof in Section~\ref{sec:construction-mod-basis}, modulo one step. That step is to find a shape-homogeneous basis for $\kk[\Sd\Delta]$ as $\kk[\Gamma]$-module. This is carried out in Section~\ref{sec:construct-a-basis}, itself in two ways. The first is an essentially routine method using Gr\"obner bases, while the second is a linear-algebraic method due to Garsia \cite{garsia}, based on the ideas in Section~\ref{sec:garsia-LA-CM}.

\subsection{Nonconstructive existence proof}\label{sec:nonconstructive}

In this section, we prove the existence part of Theorem~\ref{thm:yes-CM-in-coprime} in a nonconstructive way based on ideas in \cite{adams-reiner}. This proof was developed in conversation with V. Reiner. Throughout this section, $G$ is a finite group, $\kk$ is a field of characteristic not dividing the order of $G$, $R$ is an $\NN$-graded, connected ($R_0=\kk$), finitely generated $\kk$-algebra, $G$ acts on $R$ by graded $\kk$-algebra automorphisms, and $\Theta=\theta_1,\dots,\theta_n$ is a homogeneous system of parameters for $R$ consisting of $G$-invariant elements. Homogeneity implies the quotient $R/\Theta R$ is $\NN$-graded. Also, the assumption that $R$ is connected implies that all the $\theta_j$'s have positive degree, so the positively-graded ideal in the polynomial subring $\kk[\Theta]$ is exactly $\Theta \kk[\Theta]$.

The {\em Grothendieck ring} $\mathbf{R}_\kk(G)$ of $G$ over $\kk$ is the quotient of the free $\ZZ$-module generated by the isomorphism classes $[V]$ of objects $V$ in the category $\Rep_\kk(G)$ of finite-dimensional representations  of $G$ over $\kk$, by the submodule generated by relators
\[
[V']-[V] +[V'']
\]
for each short exact sequence
\[
0 \rightarrow V' \rightarrow V \rightarrow V'' \rightarrow 0
\]
in $\Rep_\kk(G)$, and equipped with a multiplication induced from the tensor product:
\[
[V][W] = [V\otimes_\kk W];
\]
see \cite[Section~16B]{curtis-reiner} for a careful development. As an abelian group, $\mathbf{R}_\kk(G)$ is the free $\ZZ$-module generated by the isomorphism classes of the irreducible $\kk G$-modules  \cite[Proposition~16.6]{curtis-reiner}; this follows from the Jordan-H\"older theorem.

Furthermore, because the characteristic of $\kk$ does not divide the order of $G$, Maschke's theorem holds, so every short exact sequence in $\Rep_\kk(G)$ splits. Therefore, the isomorphism class of $V$ in the exact sequence $0\rightarrow V'\rightarrow V\rightarrow V''\rightarrow 0$ is determined by the isomorphism classes of $V'$ and $V''$. By induction on the length of a composition series in $\Rep_\kk(G)$, we have:

\begin{observation}\label{obs:equality-in-Groth-ring-implies-isomorphism}
    If the characteristic of $\kk$ does not divide the order of $G$, representations are in the same class in $\mathbf{R}_\kk(G)$ (if and) only if they are isomorphic in $\Rep_\kk(G)$.
\end{observation}

One has (e.g., \cite[Section~1]{mitchell1985finite}, \cite[Section~1.1]{broer2011extending}, \cite[Section~2]{adams-reiner}) a refinement of the Hilbert series of $R$ (or, more generally, of any $\NN$-graded representation of $G$ over $\kk$) called the {\em equivariant Hilbert series}, taking values in a power series ring over the Grothendieck ring of $G$: 
\[
\Hilb^\mathrm{eq}(R,t):= \sum_{d\in \NN} [R_d]t^d \in \mathbf{R}_\kk(G)[[t]],
\]
where $R_d$ is the $d$th homogeneous component of $R$, viewed as a representation of $G$ over $\kk$. One can check that if $S$ is a second $\NN$-graded $\kk$-algebra with a $G$-action (or more generally an $\NN$-graded $G$-representation over $\kk$), then
\begin{equation}\label{eq:hilbert-series-multiply}
\Hilb^\mathrm{eq}(R\otimes_\kk S,t) = \Hilb^\mathrm{eq}(R,t)\Hilb^\mathrm{eq}(S,t).
\end{equation}
The calculation is essentially identical to the one that proves the analogous identity for ordinary Hilbert series.

The following lemma was drawn to our attention by V. Reiner, who characterized it as probably folklore. It is closely related to  \cite[Proposition~2.1(ii)]{broer2011extending}. The special case where $R$ is a polynomial ring and $\kk[\Theta]$ is the invariant ring of a reflection group is \cite[Theorem~7.4.3]{smith-poly-invars}. The action of $G$ on $R$ naturally descends to the quotient $R/\Theta R$ because the $\theta_j$'s are $G$-invariant. The tensor product $(R/\Theta R)\otimes_\kk \kk[\Theta]$ has the structure of an $\NN$-graded $\kk$-vector space because both tensor factors are $\NN$-graded. Furthermore, it is a $G$-representation and $\kk[\Theta]$-module, with the $G$-action coming from the first tensor factor, and the $\kk[\Theta]$-action from the second. 

\begin{lemma}\label{lem:vic-lemma}
In coprime characteristic, there is a $G$-equivariant surjection of $\NN$-graded $\kk[\Theta]$-modules
\[
\Phi:(R/\Theta R)\otimes_\kk \kk[\Theta] \rightarrow R.
\]
If, furthermore, $R$ is Cohen--Macaulay, then  $\Phi$ is an isomorphism.
\end{lemma}

\begin{proof}
Let $U:=R/\Theta R$. Because the characteristic of $\kk$ is prime to the order of $G$, the group ring $\kk G$ is semisimple. Because $\theta_1,\dots,\theta_n\in R$ are $G$-invariant, the ideal $\Theta R\subset R$ is $G$-stable. So, by the semisimplicity of $\kk G$, $\Theta R\subset R$ has a $G$-stable complement $U'\subset R$. Since the ideal $\Theta R$ is graded,  $U'$ can be taken to be graded. The restriction of the canonical surjection $\pi:R\rightarrow R/\Theta R=U$ to $U'$ is then an isomorphism of $\NN$-graded $G$-representations (because $U'$ is complementary to $\pi$'s kernel); let $\phi:U\rightarrow U'$ be the inverse isomorphism. Consider the $\kk$-linear map
\[
\Phi: U \otimes_\kk \kk[\Theta] \rightarrow R
\]
induced by the $\kk$-bilinear map
\[
U \times \kk[\Theta]  \ni (u,f) \mapsto \phi(u)f \in R.
\]
We claim that $\Phi$ is the promised surjective, $G$-equivariant morphism of $\NN$-graded $\kk[\Theta]$-modules. Indeed, $G$-equivariance is immediate because if $\sigma\in G$, $u\in U$, $f\in \kk[\Theta]$, then 
\[
\sigma\cdot (u\otimes f)= (\sigma\cdot u)\otimes f \mapsto \phi(\sigma \cdot u)f = (\sigma\cdot\phi(u))f = \sigma\cdot(\phi(u)f),
\]
with the first equality by definition of the $G$-action on $U\otimes_\kk \kk[\Theta]$, the second because $\phi$ is $G$-equivariant, and the third because $\sigma$ acts by algebra automorphisms and $f$ is $G$-invariant. Similarly, $\Phi$ is a $\kk[\Theta]$-module homomorphism because if $f'\in \kk[\Theta]$, then
\[
f'\cdot (u\otimes f) = u \otimes (f'f) \mapsto \phi(u)(f'f) = f'(\phi(u)f).
\]
The $\NN$-gradedness of $\Phi$ is similarly automatic from the  definition of the $\NN$-grading on $U\otimes_\kk\kk[\Theta]$. Meanwhile, surjectivity of $\Phi$ follows from the graded Nakayama lemma.

Now, suppose that $R$ is Cohen--Macaulay, and let $\Phi:U\otimes_\kk \kk[\Theta]\rightarrow R$ be a $G$-equivariant surjection of $\NN$-graded $\kk[\Theta]$-modules. Then $\Phi$ is actually an isomorphism by Vasconcelos' theorem \cite[Proposition~1.2]{vasconcelos1969finitely}, by the same argument as in the implication \ref{item:free-over-any-parameter-ring}$\Rightarrow$\ref{item:any-basis-lifts} in Lemma~\ref{lem:hironaka-decomposition}.
\end{proof}

With this preparation in place, we can give a proof of the existence part of Theorem~\ref{thm:yes-CM-in-coprime}. The idea is this. Under the coprime and Cohen--Macaulay hypotheses, Observation~\ref{obs:equality-in-Groth-ring-implies-isomorphism} and Lemma~\ref{lem:vic-lemma} imply that the $\NN$-graded $\kk G$-module structure of $\kk[\Delta]$ (without considering the $\kk[\Theta]$-module structure!) determines the $\NN$-graded $\kk G$-module structure of $\kk[\Delta]/\Theta \kk[\Delta]$, but meanwhile, the $\NN$-graded $\kk G$-module structure of $\kk[\Delta]/\Theta \kk[\Delta]$ determines even the $\NN$-graded $\kk G[\Theta]$-module structure of $\kk[\Delta]$. A similar statement applies to $\kk[\Sd\Delta]$ and $\kk[\Sd\Delta]/\Gamma \kk[\Sd\Delta]$. Since $\kk[\Delta]$ and $\kk[\Sd\Delta]$ are isomorphic as $\NN$-graded $\kk G$-modules (with the isomorphism given by the Garsia transfer), it follows that they must even be isomorphic as $\kk G[\Theta]$-modules. Here are the details. 

\begin{proof}[Proof of existence in Theorem~\ref{thm:yes-CM-in-coprime}]
    Because $\Delta$ is Cohen--Macaulay over $\kk$, both rings $\kk[\Delta]$ and $\kk[\Sd\Delta]$ are Cohen--Macaulay rings. So, taking $U:=\kk[\Delta]/\Theta\kk[\Delta]$ and $U^{\Sd}:=\kk[\Sd\Delta]/\Gamma\kk[\Sd\Delta]$, Lemma~\ref{lem:vic-lemma} combines with \eqref{eq:hilbert-series-multiply} to tell us that 
    \begin{equation}\label{eq:hilb-prod-for-delta}
    \Hilb^\mathrm{eq}(\kk[\Delta],t) = \Hilb^\mathrm{eq}(U,t)\Hilb^\mathrm{eq}(\kk[\Theta],t)
    \end{equation}
    and
    \begin{equation}\label{eq:hilb-prod-for-sd}
    \Hilb^\mathrm{eq}(\kk[\Sd\Delta],t) = \Hilb^\mathrm{eq}(U^{\Sd},t)\Hilb^\mathrm{eq}(\kk[\Gamma],t).
    \end{equation}
    Meanwhile, \begin{equation}\label{eq:gamma-theta-same-hilb}
    \Hilb^\mathrm{eq}(\kk[\Gamma],t) = \Hilb^\mathrm{eq}(\kk[\Theta],t)
    \end{equation}
    because $\kk[\Theta]$ and $\kk[\Gamma]$ are isomorphic as graded $\kk$-algebras and both carry trivial $G$-action. Also, \begin{equation}\label{eq:delta-sd-same-hilb}
    \Hilb^\mathrm{eq}(\kk[\Delta],t) = \Hilb^\mathrm{eq}(\kk[\Sd\Delta],t)
    \end{equation}
    because $\Garsia:\kk[\Sd\Delta]\rightarrow \kk[\Delta]$ is a graded, $G$-equivariant linear isomorphism and so induces isomorphisms as $G$-representations between $\kk[\Sd\Delta]_d$ and $\kk[\Delta]_d$ for all $d\in \NN$. Finally, because $\kk[\Theta]$ is connected (i.e., the zero-degree component is $\kk$), by \cite[Proposition~2.1(iv)]{broer2011extending} we know that $\Hilb^\mathrm{eq}(\kk[\Theta],t)$ is a unit in the ring $\mathbf{R}_\kk(G)[[t]]$. 
    From this together with \eqref{eq:hilb-prod-for-delta}, \eqref{eq:hilb-prod-for-sd}, \eqref{eq:gamma-theta-same-hilb}, and \eqref{eq:delta-sd-same-hilb}, we deduce that
    \[
    \Hilb^\mathrm{eq}(U,t) = \Hilb^\mathrm{eq}(U^{\Sd},t).
    \]
    In other words, for each $d\in \NN$, $[U_d] = [U^{\Sd}_d]$ in $\mathbf{R}_\kk(G)$. Because $\kk$'s characteristic does not divide the order of $G$, it follows that $U_d \cong U^{\Sd}_d$ as $G$-representations for each $d$, and so $U\cong U^{\Sd}$ as graded representations of $G$. Therefore, again by Lemma~\ref{lem:vic-lemma} (and in view of the fact that $\kk[\Theta]$ and $\kk[\Gamma]$ are isomorphic as $\kk[\Theta]$-modules), we have
    \[
    \kk[\Sd\Delta] \cong U^{\Sd}\otimes_\kk \kk[\Gamma] \cong U \otimes_\kk \kk[\Theta] \cong \kk[\Delta]
    \]
    as graded $G$-representations and $\kk[\Theta]$-modules. This completes the proof.
\end{proof}

\begin{remark}
    In this proof, equations~\eqref{eq:hilb-prod-for-delta} and \eqref{eq:hilb-prod-for-sd} are deduced from Lemma~\ref{lem:vic-lemma}, but they could alternatively have been deduced from \cite[Theorem~2.1(iv)]{broer2011extending}, which does not require the coprime characteristic hypothesis. Thus the equality $\Hilb^\mathrm{eq}(U,t) =\Hilb^\mathrm{eq}(U^{\Sd}, t)$ does not require this hypothesis. Indeed, this is used in \cite[Corollary~6.7]{adams-reiner}. The important uses of coprimality in the proof were the inference from $\Hilb^\mathrm{eq}(U,t) =\Hilb^\mathrm{eq}(U^{\Sd}, t)$ that $U$ and $U^{\Sd}$ are actually isomorphic as $\NN$-graded $G$-representations, and the {\em second} application of Lemma~\ref{lem:vic-lemma}, lifting the latter isomorphism up to an $\NN$-graded $\kk G[\Theta]$-isomorphism of $\kk[\Delta]$ and $\kk[\Sd\Delta]$.
\end{remark}

\subsection{Explicit construction, modulo construction of a basis}\label{sec:construction-mod-basis}

In this section, under the Cohen--Macaulay and coprime hypotheses, we construct an explicit $G$-equivariant $\kk[\Theta]$-module isomorphism between $\kk[\Sd\Delta]$ and $\kk[\Delta]$, given as input a shape-homogeneous $\kk[\Theta]$-module basis for $\kk[\Sd\Delta]$. Constructions of such a basis are given in Section~\ref{sec:construct-a-basis}. 

The results of this section make heavy use of the theory of shape-grading, shape-filtering, and the Garsia transfer developed in Section~\ref{sec:garsia}. To articulate them, we make two additional (hopefully natural) definitions, and prove a lemma about one of them:

\begin{definition}\label{def:shape-filtered-map}
    A $\kk$-linear map $\varphi:\kk[\Sd\Delta]\rightarrow \kk[\Delta]$ is {\em shape-filtered} if for $f\in \kk[\Sd\Delta]_\lambda$ homogeneous of shape $\lambda$, one has 
    \[
    \varphi(f) \in \bigoplus_{\mu\trianglelefteq\lambda} \kk[\Delta]_\mu.
    \]
\end{definition}

\begin{remark}
Definition~\ref{def:shape-filtered-map} could equally well have said that $\varphi$ maps $\bigoplus_{\mu\trianglelefteq\lambda}\kk[\Sd\Delta]_\mu$ to $\bigoplus_{\mu\trianglelefteq\lambda}\kk[\Delta]_\mu$.
\end{remark}

\begin{remark}
    Because the dominance relation is only between partitions of the same natural number, a shape-filtered map is automatically $\NN$-graded. That said, the theory developed here would work equally well if the dominance partial order were replaced with any order on partitions that refines it and such that lower intervals are finite (for example, the degree-lexicographic order; this is the way the theory is developed in \cite{blum-smith}). The corresponding definition of a shape-filtered map would be more relaxed.
\end{remark}

\begin{lemma}\label{lem:inverse-of-filtered-is-filtered}
    The inverse of a $\kk$-linear shape-filtered isomorphism $\varphi:\kk[\Sd\Delta]\rightarrow \kk[\Delta]$ is also shape-filtered.
\end{lemma}

\begin{proof}
    This is a routine counting argument. For any $\lambda\in\Part_n$, the finite-dimensional $\kk$-vector spaces $\kk[\Delta]_\lambda$ and $\kk[\Sd\Delta]_\lambda$ are ($\kk$-linearly) isomorphic via $\Garsia$. Thus, $\bigoplus_{\mu \trianglelefteq\lambda} \kk[\Sd\Delta]_\mu$ and $\bigoplus_{\mu \trianglelefteq\lambda} \kk[\Delta]_\mu$ have the same (finite) $\kk$-dimension. The restriction of $\varphi$ to $\bigoplus_{\mu \trianglelefteq\lambda} \kk[\Sd\Delta]_\mu$ is injective because $\varphi$ is a $\kk$-isomorphism, and it maps into $\bigoplus_{\mu \trianglelefteq\lambda} \kk[\Delta]_\mu$ because $\varphi$ is shape-filtered. Thus it induces a $\kk$-linear isomorphism of $\bigoplus_{\mu \trianglelefteq\lambda} \kk[\Sd\Delta]_\mu$ with $\bigoplus_{\mu \trianglelefteq\lambda} \kk[\Delta]_\mu$. Therefore, its inverse maps $\bigoplus_{\mu \trianglelefteq\lambda} \kk[\Delta]_\mu$ into (in fact, bijectively onto) $\bigoplus_{\mu \trianglelefteq\lambda} \kk[\Sd\Delta]_\mu$.
\end{proof}

\begin{definition}\label{def:equivariant-in-top-shape}
    Let $G\subset\Aut(\Delta)$ be a group of automorphisms. A shape-filtered  $\kk$-linear map $\varphi:\kk[\Sd\Delta]\rightarrow \kk[\Delta]$ is {\em $G$-equivariant in the top shape} if for any $f\in \kk[\Sd\Delta]_\lambda$ homogeneous of shape $\lambda$, and  any $\sigma \in G$, one has
    \[
    \sigma \cdot \varphi(f) - \varphi(\sigma\cdot f) \in \bigoplus_{\mu \triangleleft\lambda} \kk[\Delta]_\mu.
    \]
    (Note the strict dominance in the direct sum.)  
\end{definition}

\begin{convention}
    To emphasize that a map is equivariant, and not only equivariant in the top shape, we will refer to it below as {\em fully} or {\em completely} equivariant.
\end{convention}

\begin{remark} Definition~\ref{def:equivariant-in-top-shape} could equally well have said that $\sigma\cdot \varphi(f)$ and $\varphi(\sigma\cdot f)$ have equal projections to $\kk[\Delta]_\lambda$.
\end{remark}

\begin{example}
    An example illustrating both definitions is the Garsia transfer $\Garsia$, as it is even shape-graded (i.e., $\Garsia(\kk[\Sd\Delta]_\lambda) \subset \kk[\Delta]_\lambda$), and fully $G$-equivariant. (Indeed, a shape-graded map that is equivariant in the top shape is automatically fully equivariant.) A more substantive example (i.e., shape-filtered but not shape graded, and equivariant in the top shape but not fully equivariant) is given by the map $\Phi$ defined below in equation \eqref{eq:non-equivar-iso}; that it is an example is proven in Proposition~\ref{prop:basis-to-iso}.
\end{example}

With these definitions, the steps of the construction of an explicit $\kk G[\Theta]$-module isomorphism are:
\begin{enumerate}
    \item Show that, from a shape-homogeneous $\kk[\Theta]$-module basis for $\kk[\Sd\Delta]$ one can construct a (not necessarily equivariant) $\kk[\Theta]$-module isomorphism $\Phi:\kk[\Sd\Delta]\rightarrow \kk[\Delta]$ that is shape-filtered and $G$-equivariant in the top shape.\label{step:basis-to-iso}
    \item Show that, if the characteristic of $\kk$ does not divide the order of $G$, then $\Phi$ as in Step~\ref{step:basis-to-iso} can be deformed into an equivariant isomorphism $\tilde \Phi$ by averaging over $G$.\label{step:iso-to-equivariant}
\end{enumerate}
Step~\ref{step:basis-to-iso} involves the Cohen--Macaulay assumption (in order for the $\kk[\Theta]$-module basis to exist) but not the coprime characteristic assumption. On the other hand, Step~\ref{step:iso-to-equivariant} requires the coprime characteristic assumption but does not involve the Cohen--Macaulay assumption. 

To carry out Step~\ref{step:basis-to-iso}, suppose $b_1,\dots,b_r\in \kk[\Sd\Delta]$ constitute a shape-homogeneous $\kk[\Theta]$-basis of $\kk[\Sd\Delta]$. (Thus, we are requiring that $\Delta$ be Cohen--Macaulay; but we do not yet assume that $\kk$ has characteristic coprime to $|G|$.) Then, by Theorem~\ref{thm:transfer-bases}, $\Garsia(b_1),\dots,\Garsia(b_r)\in \kk[\Delta]$ form a (shape-homogeneous) $\kk[\Theta]$-basis of $\kk[\Delta]$. Immediately we can write down a (non-equivariant) isomorphism: define 
\[
\Phi: \kk[\Sd\Delta] \rightarrow \kk[\Delta]
\]
by mapping 
\begin{equation}\label{eq:non-equivar-iso}
b_j \mapsto \Garsia(b_j)
\end{equation}
for $j=1,\dots,r$ and $\kk[\Theta]$-linearly extending. We now prove that the $\Phi$ so constructed is shape-filtered, and equivariant in the top shape. This will be deduced from the following preparatory lemma.

\begin{lemma}\label{lem:match-garsia-in-top-shape}
    The $\Phi$ constructed above in \eqref{eq:non-equivar-iso} is shape-filtered, and additionally, it agrees with the Garsia transfer in the top shape, i.e., for $f\in \kk[\Sd\Delta]_\lambda$ homogeneous of shape $\lambda$, we have
    \[
    \Phi(f) - \Garsia(f) \in \bigoplus_{\mu \triangleleft\lambda} \kk[\Delta]_\mu.
    \]
\end{lemma}

\begin{proof}
    Because $b_1,\dots,b_r$ form a shape-homogeneous $\kk[\Theta]\cong \kk[\Gamma]$-module basis for the $\Part_n$-graded ring $\kk[\Sd\Delta]$, there exist $n$-variate polynomials $p_1,\dots,p_r$ such that
    \[
    f = \sum_{j=1}^r p_j(\gamma_1,\dots,\gamma_n)b_j,
    \]
    where each polynomial expression $p_j(\gamma_1,\dots,\gamma_n)$, viewed as an element of $\kk[\Sd\Delta]$, is shape-homogeneous such that
    \[
    \shape(p_j(\gamma_1,\dots,\gamma_n)) + \shape(b_j) = \lambda.
    \]
    That $\Phi$ is shape-filtered now follows immediately from Proposition~\ref{prop:x-filtered} by induction. That it agrees with the Garsia transfer in the top shape follows by comparing
    \[
    \Garsia(f) = \sum_{j=1}^r \Garsia(p_j(\gamma_1,\dots,\gamma_n)b_j)
    \]
    with
    \begin{align*}
    \Phi(f) &= \sum_1^r p_j(\theta_1,\dots,\theta_n) \Garsia(b_j)\\
    &= \sum_1^r p_j(\Garsia(\gamma_1),\dots,\Garsia(\gamma_n))\Garsia(b_j)
    \end{align*}
    using Lemma~\ref{lem:homomorphism-in-top-shape} and Observation~\ref{obs:any-number-of-factors}.
\end{proof}

We can now complete Step~\ref{step:basis-to-iso}.

\begin{prop}\label{prop:basis-to-iso}
    Suppose the boolean complex $\Delta$ is Cohen--Macaulay over $\kk$, and let $b_1,\dots,b_r$ be a shape-homogeneous basis for $\kk[\Sd\Delta]$ as $\kk[\Theta]$-module. Let $G\subset\Aut(\Delta)$ be a group of automorphisms. Then the isomorphism $\Phi$ constructed above in \eqref{eq:non-equivar-iso} is shape-filtered and  $G$-equivariant in the top shape.
\end{prop}

\begin{proof}
    That $\Phi$ is shape-filtered was already proven in Lemma~\ref{lem:match-garsia-in-top-shape}. To prove equivariance in the top shape, we use a ``three-epsilon" argument. We have
    \[
    \sigma\cdot \Phi(f) - \Phi(\sigma\cdot f) = \left(\sigma \cdot \Phi(f) - \sigma \cdot \Garsia(f)\right) + \left(\sigma \cdot \Garsia(f) - \Garsia(\sigma \cdot f) \right) + \left(\Garsia(\sigma \cdot f) - \Phi(\sigma \cdot f)\right).
    \]
    The middle summand on the right side of the equality is zero because the Garsia transfer is (fully) $G$-equivariant (Observation~\ref{obs:garsia-equivariant}). Because $\sigma$ is shape-preserving by \eqref{eq:autos-preserve-shape}, we have from Lemma~\ref{lem:match-garsia-in-top-shape} that the first and last summands on the right are both contained in $\bigoplus_{\mu \triangleleft\lambda} \kk[\Delta]_\mu$. The latter is an abelian group, so we can conclude.
\end{proof}

We turn to Step~\ref{step:iso-to-equivariant}. We suspend the Cohen--Macaulay hypothesis on $\Delta$, but must now instate the coprime characteristic hypothesis. The takeaway is that any $\kk[\Theta]$-module isomorphism between $\kk[\Sd\Delta]$ and $\kk[\Delta]$ (whether or not they are free over $\kk[\Theta]$) that is shape-filtered and equivariant in the top shape remains an isomorphism (and becomes fully equivariant) after averaging over $G$.

\begin{prop}\label{prop:iso-to-equivariant}
    Suppose $\Delta$ is a boolean complex, not necessarily Cohen--Macaulay, but such that there exists a $\kk[\Theta]$-module isomorphism
    \[
    \Xi:\kk[\Sd\Delta] \rightarrow \kk[\Delta]
    \]
    that is shape-filtered and $G$-equivariant in the top shape. Let $G\subset \Aut(\Delta)$ be a group of automorphisms, and assume that the order of $G$ is not divisible by the characteristic of $\kk$.  Then the map
    \[
    \Xi^\star:\kk[\Sd\Delta] \rightarrow \kk[\Delta]
    \]
    defined on $f\in \kk[\Sd\Delta]$ by
    \[
    \Xi^\star(f) := \frac{1}{|G|}\sum_{\sigma\in G} \sigma \cdot [\Xi(\sigma^{-1}\cdot f)]
    \]
    is a $G$-equivariant $\kk[\Theta]$-module isomorphism.
\end{prop}

\begin{proof}
    By construction, $\Xi^\star$ is $G$-equivariant and a $\kk[\Theta]$-module map. The point is to show that it is an isomorphism. Note that because $\Xi$ is shape-filtered and the action of $G$ is shape-preserving on both $\kk[\Sd\Delta]$ and $\kk[\Delta]$, each summand $\sigma \cdot [\Xi(\sigma^{-1}\cdot -)]$ of $\Xi^\star$ is shape-filtered. Using that $\bigoplus_{\mu \trianglelefteq\lambda} \kk[\Delta]_\mu$ is a $\kk$-vector space, it follows that $\Xi^\star$ is shape-filtered.

    Let $f\in \kk[\Sd\Delta]_\lambda$ be homogeneous of shape $\lambda$, and let $\sigma\in G$ be arbitrary. Applying the facts that $\sigma,\sigma^{-1}$ preserve shape and $\Xi$ is equivariant in the top shape, we find that 
    \[
    \sigma\cdot [\Xi(\sigma^{-1}\cdot f)] - \Xi(f) = \sigma\cdot [\Xi(\sigma^{-1}\cdot f)] - \Xi(\sigma\cdot\sigma^{-1}\cdot f)\in \bigoplus_{\mu \triangleleft\lambda} \kk[\Delta]_\mu.
    \]
    Averaging the left-side expression over $G$ and using that $\bigoplus_{\mu \triangleleft\lambda} \kk[\Delta]_\mu$ is a $\kk$-vector space, we find that 
    \begin{equation}\label{eq:average-agrees-in-top-shape}
    \Xi^\star(f) - \Xi(f)\in \bigoplus_{\mu \triangleleft\lambda} \kk[\Delta]_\mu.
    \end{equation}
    Immediately, the same statement holds if $f$ is not shape-homogeneous but merely contained in $\bigoplus_{\mu \trianglelefteq\lambda} \kk[\Sd\Delta]_\mu,$ by splitting $f$ into shape-homogeneous components and applying \eqref{eq:average-agrees-in-top-shape} to each component.

    As remarked after Definition~\ref{def:shape-filtered-map}, a shape-filtered map is $\NN$-graded. Therefore, since $\kk[\Sd\Delta]$ and $\kk[\Delta]$ have the same $\NN$-graded Hilbert series, injectivity and surjectivity of $\Xi^\star$ imply each other. Thus it suffices to prove either one. We prove surjectivity.

    For a contradiction, suppose $\Xi^\star$ is not surjective, and let $\lambda\in \Part_n$ be dominance-minimal such that $\kk[\Delta]_\lambda$ is not contained in the image of $\Xi^\star$. Find $f'\in \kk[\Delta]_\lambda$ which does not lie in this image. By Lemma~\ref{lem:inverse-of-filtered-is-filtered}, we know
    \[
    f:=\Xi^{-1}(f')\in \bigoplus_{\mu \trianglelefteq\lambda} \kk[\Sd\Delta]_\mu.
    \]
    By \eqref{eq:average-agrees-in-top-shape} and the sentence after it, we have 
    \[
    \Xi^\star(f) - f' \in \bigoplus_{\mu \triangleleft\lambda} \kk[\Delta]_\mu.
    \]
    But meanwhile, $\Xi^\star$ is surjective onto $\bigoplus_{\mu \triangleleft\lambda} \kk[\Delta]_\mu,$ by the minimality of $\lambda$. This is a contradiction because $\Xi^\star(f)$ is contained in the image of $\Xi^\star$, but $f'$ was supposed not to be. This completes the proof.
\end{proof}

\begin{proof}[Proof of Theorem~\ref{thm:yes-CM-in-coprime} modulo construction of a basis]
    We combine Steps~\ref{step:basis-to-iso} and \ref{step:iso-to-equivariant}. Proposition \ref{prop:basis-to-iso} shows that given a shape-homogeneous $\kk[\Theta]$-module basis for $\kk[\Sd\Delta]$, the $\kk[\Theta]$-module map given in \eqref{eq:non-equivar-iso} is a shape-filtered  isomorphism $G$-equivariant in the top shape, and then Proposition~\ref{prop:iso-to-equivariant} averages this isomorphism across the group to obtain a fully $G$-equivariant $\kk[\Theta]$-module isomorphism.
\end{proof}

\subsection{Construction of a basis}\label{sec:construct-a-basis}

In this section, we complete the proof of Theorem~\ref{thm:yes-CM-in-coprime} by showing how to compute a shape-homogeneous basis for $\kk[\Sd\Delta]$ as $\kk[\Theta]\cong \kk[\Gamma]$-module (under the hypothesis that the former is Cohen--Macaulay). We give two independent methods for doing this. 

The first is a routine application of Gr\"obner bases and we describe it only in brief outline. We include it because it makes Theorem~\ref{thm:yes-CM-in-coprime} fully constructive via well-known tools. 

The second, Algorithm~\ref{alg:basis} below, is the primary goal of this section. It is a purely linear-algebraic method that avoids any Gr\"obner basis calculations,  essentially due to Adriano Garsia in \cite{garsia}, the same paper that introduced the Garsia transfer. We extend the proof to the setting in which $\Sd\Delta$ is replaced with an arbitrary pure, balanced boolean complex, clarifying an ambiguity in \cite{garsia} in the process.

Garsia's algorithm also implicitly contains an algorithm to represent a given element of $\kk[\Sd\Delta]$ as a $\kk[\Theta]\cong \kk[\Gamma]$-linear combination of the elements of the bases they provide. This is drawn out at the end of the section, fulfilling the promise made in Example~\ref{ex:proof-illlustration}.

\paragraph{Gr\"obner basis method.}\label{sec:grobner-basis}
View $\kk$ as the $\kk[\Gamma]$-module $\kk[\Gamma]/\Gamma\kk[\Gamma]$. Because $\Gamma$ is a homogeneous system of parameters for $\kk[\Sd\Delta]$, the graded quotient ring $\kk[\Sd\Delta]/\Gamma\kk[\Sd\Delta] = \kk[\Sd\Delta]\otimes_{\kk[\Gamma]} \kk$ has finite $\kk$-dimension, and any homogeneous $\kk$-spanning set for it lifts to a homogeneous $\kk[\Gamma]$-module generating set for $\kk[\Sd\Delta]$, by the graded Nakayama lemma. As we are assuming $\kk[\Sd\Delta]$ is Cohen--Macaulay, it is a free $\kk[\Theta]\cong \kk[\Gamma]$-module; then $\kk[\Sd\Delta]\otimes_{\kk[\Gamma]} \kk$ is a $\kk$-vector space of the same rank. Thus, any homogeneous $\kk$-basis for $\kk[\Sd\Delta]/(\Gamma)$ lifts to a $\kk[\Theta]\cong \kk[\Gamma]$-basis of $\kk[\Sd\Delta]$. Our work is thus reduced to providing a shape-homogeneous $\kk$-basis for $\kk[\Sd\Delta]/(\Gamma)$. (Note that because the $\gamma_i$ are shape-homogeneous, this latter quotient inherits a grading by shape.) The following is a standard procedure for computing a $\kk$-basis for a finitely generated $\kk$-algebra for which we have an explicit presentation.

The ring $\kk[\Sd\Delta]/(\Gamma)$ is presented as follows: we have generators $y_\alpha$, $\alpha\in P(\Delta)$, and relations of two types:
\begin{enumerate}
    \item $y_\alpha y_\beta$ for any pair $\alpha,\beta$ of incomparable elements of $P(\Delta)$, \label{rel:incomparable} and
    \item $\gamma_j := \sum_{\rk(\alpha)=j} y_\alpha$ for $j=1,\dots, n$ (where, as usual, ranks are calculated in $\widehat{P}(\Delta)$).\label{rel:gammas}
\end{enumerate}
These relations generate an ideal $I^\Delta_n$ in the parent polynomial ring $\kk[\{y_\alpha\}_{\alpha \in P(\Delta)}]$ of $\kk[\Sd\Delta]$ (the notation roughly follows \cite{adams2023further}). One now chooses any monomial order on this polynomial ring, and computes a Gr\"obner basis of $I^\Delta_n$, which  determines the initial ideal $\operatorname{in}(I^\Delta_n)$.  Then, by  Gr\"obner basis theory,  the complement of $\operatorname{in}(I^\Delta_n)$ in the set of monomials on the $y_\alpha$'s yields a $\kk$-basis for $\kk[\Sd\Delta]/(\Gamma)$. Because this basis consists of monomials in the $y_\alpha$'s, it is automatically shape-homogeneous. \qed

\paragraph{Garsia's linear-algebraic method.}\label{sec:garsia-lin-alg}

The procedure to be described here avoids Gr\"obner bases, using only linear algebra, and is reasonable to compute by hand in small examples. It is essentially found in \cite[Theorem~3.3]{garsia}, where it is presented as a {\em test} of Cohen--Macaulayness, although it actually computes a basis in the Cohen--Macaulay case. Based on the work in Section~\ref{sec:garsia-LA-CM}, we give it in a bit more generality than the original setting of \cite{garsia} (and than our application to $\kk[\Sd\Delta]$ requires). The complex $\Sd \Delta$ is not only a boolean complex but a simplicial complex, and in fact the order complex of a ranked poset---the latter is the original context of \cite{garsia}. Since $\Delta$ is Cohen--Macaulay (in the situation of Theorem~\ref{thm:yes-CM-in-coprime}), it is pure, and it follows that $\Sd \Delta$ is also pure, see Section~\ref{sec:garsia-LA-CM}. The procedure can be run on any pure, balanced boolean complex (whether simplicial or not, let alone the order complex of a poset), assesses Cohen--Macaulayness, and delivers a basis for the Stanley--Reisner ring in the Cohen--Macaulay case.\footnote{In this generality, it was described, with correctness only conjectured, in \cite[Section~2.8]{blum-smith}, whose author did not at the time realize that it had in essence already been described in \cite{garsia}.}

By the remark following Lemma~\ref{lem:cells-generate}, if $\kk[\Lambda]$ has a basis over $\kk[\Omega]$ then it can be taken to consist of $z_\alpha$'s, and the goal is to determine algorithmically whether such a set of $z_\alpha$'s exists, and find it when it does. Here is the procedure. Recall Definition~\ref{def:facet-vector}, the {\em facet vector} $\bfv_\alpha^\Lambda$ of a face $\alpha$ in the complex $\Lambda$.

\begin{algorithm}\label{alg:basis}
    Input: a pure, balanced boolean complex $\Lambda$ with facets $\epsilon_1,\dots,\epsilon_m$.
    \begin{enumerate}
        \item Initialize $B=\varnothing$; this is a container for the elements of the candidate basis.\label{step:initialize-B}
        \item Initialize $V=\{0\}\subset \kk^m$; this keeps track of the contribution of the $\kk[\Omega]$-span of $B$ to the $\kk\epsilon_1\oplus\dots\oplus \kk\epsilon_m$-component of $\kk[\Lambda]$, as in Theorem~\ref{thm:garsia-LA-CM}. At every stage of the algorithm, the subspace $V$ will have the facet vectors $\{\bfv_\beta^\Lambda: z_\beta\in B\}$ as a basis.\label{step:initialize-V}
        \item Order the faces $\alpha$ of $\widehat P(\Lambda)$, including the empty face $\varnothing$, in the following way. Partition them into blocks $\{\alpha:J_\alpha = S\}$ according to their label set $S\subset[n]$; totally order the collection of blocks in any way that refines the containment order on the corresponding label sets $S$ (so the block containing the facets comes last); and then within each block $\{\alpha:J_\alpha = S\}$, impose any total order whatsoever on the faces in that block.\label{step:total-order}
        \item  Inductively process the faces of $\Lambda$, as follows:
        Consider the face $\alpha\in \widehat P(\Lambda)$ minimal with respect to the order defined in Step~\ref{step:total-order} among those that have not yet been processed, and compute its facet vector $\bfv_\alpha^\Lambda\in \kk^m$. If there is no $\alpha\in \widehat P(\Lambda)$ that has not yet been processed, then go to Step~\ref{step:output}.\label{step:process-next-alpha}
        \item Check membership of $\bfv_\alpha^\Lambda$ in $V$:
            \begin{enumerate}
                \item If $\bfv_\alpha^\Lambda\notin V$, then set $B:=B\cup \{z_\alpha\}$ and $V:= V \oplus \kk \bfv_\alpha^\Lambda$, and go back to Step~\ref{step:process-next-alpha}.\label{step:if-v-notin-V}
                \item If $\bfv_\alpha^\Lambda \in V$, then compute its representation on the basis $\{\bfv_\beta^\Lambda:z_\beta\in B\}$ for $V$.\label{step:if-v-in-V}
                \begin{enumerate}
                    \item If there is any $\bfv_\beta^\Lambda$ with nonzero coefficient in the representation of $\bfv_\alpha^\Lambda$ on the basis $\{\bfv_\beta^\Lambda:z_\beta\in B\}$ for $V$ such that $J_\beta \not\subseteq J_\alpha$, then terminate the algorithm and output ``$\Lambda$ is not Cohen--Macaulay".\label{step:if-color-condition-fails}
                    \item If every $\bfv_\beta^\Lambda$ appearing with nonzero coefficient in the representation of $\bfv_\alpha^\Lambda$ on the basis $\{\bfv_\beta^\Lambda:z_\beta\in B\}$ satisfies $J_\beta\subseteq J_\alpha$, then discard $z_\alpha$ and go back to Step~\ref{step:process-next-alpha}.\label{step:if-color-condition-succeeds}
                \end{enumerate}
            \end{enumerate}
        \item \label{step:output} Terminate the algorithm and output ``$\Lambda$ is Cohen--Macaulay" along with the basis $B$ for $\kk[\Lambda]$ over $\kk[\Omega]$.
    \end{enumerate}
\end{algorithm}

\begin{remark}
    If the algorithm reaches Step~\ref{step:output}, then $\{\bfv_\beta^\Lambda : z_\beta \in B\}$ must be a basis for $\kk^m$ when it does. This will be proven over the course of the proof of correctness. Although the algorithm is formulated in a way that processes every $\alpha\in P(\Lambda)$, one can get away with going straight to Step~\ref{step:output} once $V=\kk^m$ and the only unprocessed faces of $P(\Lambda)$ are facets, because when $V=\kk^m$ the condition in Step~\ref{step:if-v-in-V} holds automatically, while for a facet $\epsilon_i$, the condition in Step~\ref{step:if-color-condition-succeeds}  holds automatically. So any facets remaining at that stage will be discarded.
\end{remark}

Before giving the proof of correctness, we first give a pair of complete examples, illustrating the Cohen--Macaulay and non-Cohen--Macaulay cases. 

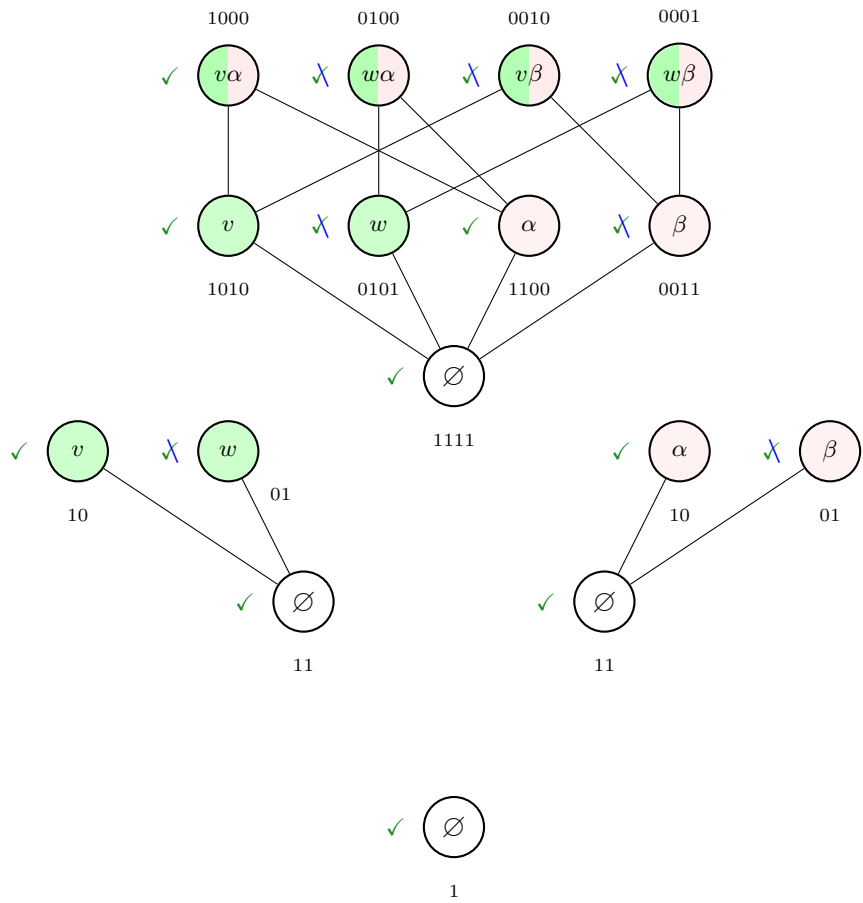
\begin{figure}
\begin{center}

\begin{tikzpicture}[
  scale=1, every node/.style={font=\small},
  greennode/.style={circle, draw=black, fill=green!20, thick, minimum size=8mm},
  pinknode/.style={circle, draw=black, fill=pink!20, thick, minimum size=8mm},
  emptynode/.style={circle, draw=black, fill=white, thick, minimum size=8mm},
  labelnode/.style={font=\scriptsize, align=center},
  halfcircle/.style={
    circle,
    draw=black,
    thick,
    minimum size=8mm,
    path picture={
      \path[clip] (path picture bounding box.center) circle(4mm);
      \fill[green!30] ([xshift=-4mm,yshift=-4mm]path picture bounding box.center)
          rectangle ([xshift=0mm,yshift=4mm]path picture bounding box.center);
      \fill[pink!30] ([xshift=0mm,yshift=-4mm]path picture bounding box.center)
          rectangle ([xshift=4mm,yshift=4mm]path picture bounding box.center);
    }
  }
]

\node[emptynode] (emptyBig) at (0,0) {$\emptyset$};
\node[left=3pt of emptyBig] {\greencheck};
\node[labelnode, below=1.5ex of emptyBig] {1111};

\node[greennode] (vBig) at (-3,2) {$v$};
\node[left=3pt of vBig] {\greencheck};
\node[labelnode, below=1.5ex of vBig] {1010};

\node[greennode] (wBig) at (-1,2) {$w$};
\node[left=3pt of wBig] {\notcheckmark};
\node[labelnode, below=1.5ex of wBig] {0101};

\node[pinknode] (alphaBig) at (1,2) {$\alpha$};
\node[left=3pt of alphaBig] {\greencheck};
\node[labelnode, below=1.5ex of alphaBig] {1100};

\node[pinknode] (betaBig) at (3,2) {$\beta$};
\node[left=3pt of betaBig] {\notcheckmark};
\node[labelnode, below=1.5ex of betaBig] {0011};

\node[halfcircle] (vaBig) at (-3,4) {$v\alpha$};
\node[left=3pt of vaBig] {\greencheck};
\node[labelnode, above=1ex of vaBig] {1000};

\node[halfcircle] (waBig) at (-1,4) {$w\alpha$};
\node[left=3pt of waBig] {\notcheckmark};
\node[labelnode, above=1ex of waBig] {0100};

\node[halfcircle] (vbBig) at (1,4) {$v\beta$};
\node[left=3pt of vbBig] {\notcheckmark};
\node[labelnode, above=1ex of vbBig] {0010};

\node[halfcircle] (wbBig) at (3,4) {$w\beta$};
\node[left=3pt of wbBig] {\notcheckmark};
\node[labelnode, above=1ex of wbBig] {0001};

\draw (emptyBig) -- (vBig);
\draw (emptyBig) -- (wBig);
\draw (emptyBig) -- (alphaBig);
\draw (emptyBig) -- (betaBig);
\draw (vBig) -- (vaBig);
\draw (vBig) -- (vbBig);
\draw (wBig) -- (waBig);
\draw (wBig) -- (wbBig);
\draw (alphaBig) -- (vaBig);
\draw (alphaBig) -- (waBig);
\draw (betaBig) -- (vbBig);
\draw (betaBig) -- (wbBig);

\begin{scope}[shift={(0,-3)}]

\begin{scope}[shift={(-2.0,0)}] 
\node[emptynode] (emptyVW) at (0,0) {$\emptyset$};
\node[left=3pt of emptyVW] {\greencheck};
\node[labelnode, below=1.5ex of emptyVW] {11};

\node[greennode] (vSmall) at (-3,2) {$v$};
\node[left=3pt of vSmall] {\greencheck};
\node[labelnode, below=1.5ex of vSmall] {10};

\node[greennode] (wSmall) at (-1,2) {$w$};
\node[left=3pt of wSmall] {\notcheckmark};
\node[labelnode, below right=2pt and 4pt of wSmall] {01}; 
\draw (emptyVW) -- (vSmall);
\draw (emptyVW) -- (wSmall);
\end{scope}

\begin{scope}[shift={(2.0,0)}] 
\node[emptynode] (emptyAB) at (0,0) {$\emptyset$};
\node[left=3pt of emptyAB] {\greencheck};
\node[labelnode, below=1.5ex of emptyAB] {11};

\node[pinknode] (alphaSmall) at (1,2) {$\alpha$};
\node[left=3pt of alphaSmall] {\greencheck};
\node[labelnode, below=1.5ex of alphaSmall] {10};

\node[pinknode] (betaSmall) at (3,2) {$\beta$};
\node[left=3pt of betaSmall] {\notcheckmark};
\node[labelnode, below=1.5ex of betaSmall] {01};

\draw (emptyAB) -- (alphaSmall);
\draw (emptyAB) -- (betaSmall);
\end{scope}
\end{scope}

\node[emptynode] (emptySolo) at (0,-6) {$\emptyset$};
\node[left=3pt of emptySolo] {\greencheck};
\node[labelnode, below=1.5ex of emptySolo] {1};

\end{tikzpicture}

\end{center}
\caption{The face poset of the barycentric subdivision $\Sd\Delta$ of the Boolean complex $\Delta$ depicted in Figure~\ref{fig:running-ex}, showing the balancing and the (face posets of the) label-selected subcomplexes. The balancing is indicated by colors: green indicates label $1$, i.e., vertices of $\Sd\Delta$ coming from vertices in $\Delta$, while pink indicates label $2$, i.e., vertices of $\Sd\Delta$ coming from edges in $\Delta$. The check marks indicate the operation of Algorithm~\ref{alg:basis}. Green check marks indicate faces $\alpha$ that reach Step~\ref{step:if-v-notin-V}, so that $z_\alpha$ gets included in the proposed basis $B$. The green check marks with blue slashes through them indicate faces $\alpha$ that reach Step~\ref{step:if-color-condition-succeeds}, so that the corresponding $z_\alpha$'s get discarded. Note that, per Proposition~\ref{prop:garsia-3.1} and the proof of correctness of Algorithm~\ref{alg:basis}, the basis obtained by the algorithm also restricts to bases for each of the label-selected subcomplexes.}
\label{fig:subspaces-of-facet-space}
\end{figure}

\begin{example}\label{ex:running-ex-basis}
    We take $\Lambda=\Sd\Delta$, with $\Delta$ as in Figure~\ref{fig:running-ex}. The face poset of $\Sd\Delta$ appears in Figure~\ref{fig:subspaces-of-facet-space}, with the nodes labeled according to their facet vectors. The nodes are also color-coded according to their label sets, with green being label 1 (the vertices of $\Sd\Delta$ representing the barycenters of faces of $\Delta$ that are rank $1$ in the face poset of $\Delta$ itself, i.e., vertices of $\Delta$), and pink being label 2 (the vertices of $\Sd\Delta$ coming from edges in $\Delta$). The presentation of $\kk[\Sd\Delta]$ we worked with in Section~\ref{sec:garsia} has $y_v,y_w,y_\alpha,y_\beta$ as generators, one for each barycenter of a face of $\Delta$; but here we work with the presentation as boolean complex. It has corresponding generators $z_v=y_v,\dots$ etc., but also additional generators $z_{v\alpha}=y_vy_\alpha,\dots$ corresponding to the non-vertex faces of $\Sd\Delta$.

    After initializing $B$ and $V$ (Steps~\ref{step:initialize-B} and \ref{step:initialize-V}), we order (Step~\ref{step:total-order}) the label sets in any inclusion-respecting way---we choose $\varnothing < \{1\} < \{2\} < \{1,2\}$---and then the faces within each label set in any way at all. We choose $v<w$,  $\alpha<\beta$, and the order on the facets in which they appear in Figure~\ref{fig:subspaces-of-facet-space}, so the total order is \[
    \varnothing<v<w<\alpha<\beta<v\alpha<w\alpha<v\beta<w\beta.
    \]
    We move on to Step~\ref{step:process-next-alpha}. The first face to process is $\varnothing$. We have $\bfv_\varnothing^{\Sd\Delta}=(1,1,1,1)$. This is not in $V$, so according to Step~\ref{step:if-v-notin-V}, we reset $B:=\{z_\varnothing\}=\{1\}$ and $V:=\kk(1,1,1,1)$, and go back to Step~\ref{step:process-next-alpha}. The next unprocessed facet is $v$, with facet vector $(1,0,1,0)$, and again it is not in $V$, so we reset $B:=\{1,z_v\}$ and $V:=\kk (1,1,1,1) \oplus \kk(1,0,1,0)$ and go back to Step~\ref{step:process-next-alpha}.

    The next unprocessed facet is $w$. This time, we have $\bfv_w^{\Sd\Delta}\in V$, because
    \[
    \bfv_w^{\Sd\Delta} = (0,1,0,1) = \bfv_\varnothing^{\Sd\Delta} - \bfv_v^{\Sd\Delta}
    \]
    as in Step~\ref{step:if-v-in-V}. (Note that this representation works over any field, so the present computation is unaffected by the choice of $\kk$.) We check the condition that distinguishes Step~\ref{step:if-color-condition-fails} from Step~\ref{step:if-color-condition-succeeds}: $J_\varnothing = \varnothing \subset \{1\} = J_w$, and $J_v = \{1\}\subset J_w$ as well, so Step~\ref{step:if-color-condition-succeeds} gets implemented: we discard $z_w$ and go back to Step~\ref{step:process-next-alpha}.

    Continuing in the same way for the block with label set $\{2\}$, we end up adding $z_\alpha$ to the basis, and discarding $z_\beta$, at which point we have $B=\{1,z_v,z_\alpha\}$ and $V=\kk(1,1,1,1) \oplus \kk(1,0,1,0) \oplus \kk(1,1,0,0)$. Again, the arithmetic works in any field.

    We finally reach the facets. The facet vector of $v\alpha$ is $(1,0,0,0)$, which does not lie in $V$, thus $B$ becomes $\{1,z_v,z_\alpha,z_{v\alpha}\}$, and $V$, being the span of four linearly independent vectors, becomes $\kk^4$. As in the remark following the algorithm description, the remaining facets will be discarded. We output $\{1,z_v,z_\alpha,z_{v\alpha}\}$ as a $\kk[\Gamma]$-basis for $\kk[\Sd\Delta]$.

    Together with the proof of correctness below, this example fulfills the promise made in Example~\ref{ex:proof-illlustration} to warrant the claim that $1=z_\varnothing,y_v=z_v, y_\alpha=z_\alpha, y_vy_\alpha=z_{v\alpha}$ constitutes a basis for $\kk[\Sd\Delta]$ over $\kk[\Gamma]$ in this case.
\end{example}

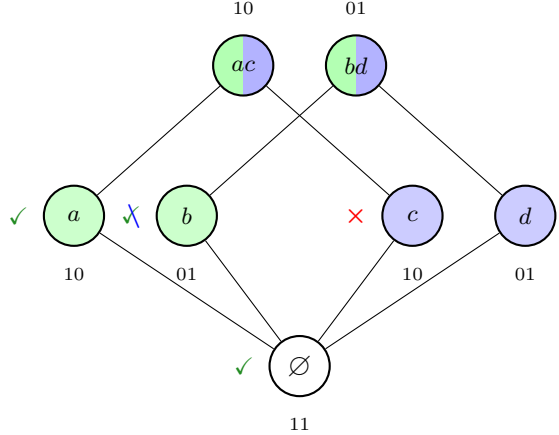
\begin{figure}
\centering

\begin{tikzpicture}[
  scale=1, every node/.style={font=\small},
  greennode/.style={circle, draw=black, fill=green!20, thick, minimum size=8mm},
  bluenode/.style={circle, draw=black, fill=blue!20, thick, minimum size=8mm},
  halfcircle/.style={
    circle,
    draw=black,
    thick,
    minimum size=8mm,
    path picture={
      \path[clip] (path picture bounding box.center) circle(4mm);
      \fill[green!30] ([xshift=-4mm,yshift=-4mm]path picture bounding box.center)
          rectangle ([xshift=0mm,yshift=4mm]path picture bounding box.center);
      \fill[blue!30] ([xshift=0mm,yshift=-4mm]path picture bounding box.center)
          rectangle ([xshift=4mm,yshift=4mm]path picture bounding box.center);
    }
  },
  emptynode/.style={circle, draw=black, fill=white, thick, minimum size=8mm},
  labelnode/.style={font=\scriptsize, align=center}
]

\node[emptynode] (emptyset) at (0,-3) {$\emptyset$};
\node[left=2pt of emptyset] {\greencheck};
\node[labelnode, below=1ex of emptyset] {11};

\node[greennode] (a2) at (-3,-1) {$a$};
\node[greennode] (b2) at (-1.5,-1) {$b$};
\node[bluenode] (c2) at (1.5,-1) {$c$};
\node[bluenode] (d2) at (3,-1) {$d$};

\node[halfcircle] (ac) at (-0.75,1) {$ac$};
\node[halfcircle] (bd) at (0.75,1) {$bd$};

\draw (emptyset) -- (a2);
\draw (emptyset) -- (b2);
\draw (emptyset) -- (c2);
\draw (emptyset) -- (d2);

\draw (a2) -- (ac);
\draw (c2) -- (ac);
\draw (b2) -- (bd);
\draw (d2) -- (bd);

\node[left=2pt of a2] {\greencheck};
\node[labelnode, below=1ex of a2] {10};

\node[left=2pt of c2] {\redcross};
\node[labelnode, below=1ex of c2] {10};

\node[labelnode, above=1ex of ac] {10};

\node[left=2pt of b2] {\notcheckmark};
\node[labelnode, below=1ex of b2] {01};

\node[labelnode, below=1ex of d2] {01};

\node[labelnode, above=1ex of bd] {01};

\end{tikzpicture}

\caption{The face poset of the non-Cohen--Macaulay pure, balanced complex $\Lambda$ consisting of vertices $a,b,c,d$ and disjoint line segments $ac$ and $bd$, and the operation of Algorithm~\ref{alg:basis} on it. The balancing is indicated by colors, with label $1$ vertices in the green circles and label $2$ vertices in the indigo circles. The operation of the algorithm is indicated by check marks and $\times$ signs. The green check marks indicate faces $\alpha$ that reach Step~\ref{step:if-v-notin-V}, so that $z_\alpha$ gets included in the proposed basis $B$. The green check mark with blue slash through it indicates a face $\alpha$ that reaches Step~\ref{step:if-color-condition-succeeds}. The red $\times$ sign indicates a face that reaches Step~\ref{step:if-color-condition-fails}, terminating the algorithm and outputting the failure of Cohen--Macaulayness.}
\label{fig:non-CM-example}
\end{figure}

\begin{example}
    A pair of disjoint edges is an example of a small pure, balanced, but non-Cohen--Macaulay  boolean complex. (This is not Cohen--Macaulay over any field; a more elaborate example would be required to see the difference between fields of different characteristics.) So take $\Lambda$ to be the boolean complex with vertices $a,b,c,d$ and edges $ac,bd$. A balancing is given by assigning label $1$ to vertices $a$ and $b$, and $2$ to vertices $c$ and $d$. The face poset is illustrated in Figure~\ref{fig:non-CM-example}, along with the results of Algorithm~\ref{alg:basis}.
    
    We go somewhat more breezily through the process than in Example~\ref{ex:running-ex-basis}. We initialize $B$ empty, $V=\{0\}$, and choose a total order on the faces of $\Lambda$  compatible with the order $\varnothing<\{1\}<\{2\}<\{1,2\}$ on the label sets. Such an order is given by
    \[
    \varnothing<a<b<c<d<ac<bd.
    \]
There are $m=2$ facets, so the ambient vector space of the facet vectors is $\kk^2$. We begin to process the faces. As usual, $z_\varnothing = 1$ goes in $B$, and the span of its facet vector $\bfv_\varnothing^\Lambda=(1,1)$ is added to $V$. Next, the span of the facet vector $\bfv_a^\Lambda = (1,0)$ is added to $V$ and $z_a$ is added to $B$, at which point we have $V=\kk^2$, the entire ambient space. Now $\bfv_b^\Lambda=(0,1)$ lies in $V$, and its representation $\bfv_b^\Lambda = \bfv_\varnothing^\Lambda - \bfv_a^\Lambda$ only involves the facet vectors of faces $\emptyset, a$ with label sets $J_\varnothing, J_a$ contained in $b$'s label set $J_b=\{1\}$, so $z_b$ is discarded and we move on to face $c$. 

Here, the algorithm reaches Step~\ref{step:if-color-condition-fails}. The facet vector $\bfv_c^\Lambda=(1,0)$ belongs to $V$; indeed, it is the same as $\bfv_a^\Lambda$. So we have the representation $\bfv_c^\Lambda = \bfv_a^\Lambda$ on the basis $\{\bfv_\varnothing^\Lambda,\bfv_a^\Lambda\}$ for $V$. But $a$'s label set $J_a=\{1\}$ is not contained in $c$'s label set $J_c = \{2\}$. Thus the algorithm terminates and outputs ``$\Lambda$ is not Cohen--Macaulay", as it should.
\end{example}

We now prove correctness for the algorithm.

\begin{proof}[Proof of correctness of Algorithm~\ref{alg:basis}]
    We need to show that if the algorithm ever reaches Step~\ref{step:if-color-condition-fails}, then $\Lambda$ is not Cohen--Macaulay, while if the algorithm terminates without reaching Step~\ref{step:if-color-condition-fails}, then $\Lambda$ is Cohen--Macaulay and at the end, $B$ is a $\kk[\Omega]$-module basis of $\kk[\Lambda]$.
    
    Before looking at the two cases, we enumerate some facts that apply to both:
\begin{enumerate}
\item By Lemma~\ref{lem:what-is-the-facet-vector}, the ambient space $\kk^m$ of the facet vectors can be identified with the component
\[
\kk[\Lambda]_{\bfe_1+\dots+\bfe_n} = \kk z_{\epsilon_1} \oplus \dots \oplus \kk z_{\epsilon_m}
\]
of $\kk[\Lambda]$ of $\NN^n$-degree $\bfe_1+\dots+\bfe_n$, via the map, call it $\iota$, that sends the standard basis for $\kk^m$ to the basis $z_{\epsilon_1},\dots,z_{\epsilon_m}$ for this space. Again by Lemma~\ref{lem:what-is-the-facet-vector}, this identification $\iota$ maps
\[
\bfv_\alpha^\Lambda \mapsto \left(\prod_{j\in [n]\setminus J_\alpha} \omega_j\right)z_\alpha
\]
for any $\alpha\in \widehat P(\Lambda)$.\label{fact:identify-top-component-with-kn}  Utilizing Garsia's notation \eqref{eq:def-of-M_S}, for any $S\subset [n]$ we have
\[
M_S = \iota\left( \bigoplus_{\alpha : J_\alpha = S} \kk \bfv_\alpha^\Lambda\right).
\]

\item Because the order in which faces are processed in the algorithm, fixed in Step~\ref{step:total-order}, is consistent with containment order on the label sets of the faces, it follows that when a given face $\alpha$ is being processed, every face with label set contained strictly in $S:=J_\alpha$ has already been processed. In particular, for any $T\subsetneq S$ and any $\gamma\in \widehat P(\Lambda)$ with $J_\gamma = T$, it must be that $\bfv_\gamma^\Lambda \in V$ (either because the span of $\bfv_\gamma^\Lambda$ was added to $V$ when $z_\gamma$ was added to $B$, per Step~\ref{step:if-v-notin-V}, or because $\bfv_\gamma^\Lambda$ was already in $V$ when we started to process $\gamma$, per Step~\ref{step:if-v-in-V}; in either case this all happened prior to processing $\alpha$). With $\iota$ the identification in fact~\ref{fact:identify-top-component-with-kn} and recalling Garsia's notation \eqref{eq:def-of-M_S} from Section~\ref{sec:garsia-LA-CM}, it follows that
\begin{equation}\label{eq:lower-Ts}
\sum_{T\subsetneq S} M_T \subseteq \iota(V)
\end{equation}
by allowing $\gamma$ to range over all faces preceding $\alpha$ in the order from Step~\ref{step:total-order}, which in particular includes all faces such that $J_\gamma =T\subsetneq S$. This holds whether or not $\bfv_\alpha^\Lambda \in V$ (i.e., whether we go to Step~\ref{step:if-v-notin-V} or Step~\ref{step:if-v-in-V}).\label{fact:already-processed}

\item For any $\alpha$, we have $\iota(\bfv_\alpha^\Lambda)\in M_S$, with $S:=J_\alpha$, by fact~\ref{fact:identify-top-component-with-kn}. We argue that at every stage of the algorithm, for an $\alpha$ with $z_\alpha\in B$, $\iota(\bfv_\alpha^\Lambda)$ cannot also be in $\sum_{T\subsetneq S} M_T$.. 

If $z_\alpha \in B$, then the algorithm reached Step~\ref{step:if-v-notin-V} when $\alpha$ was processed (since this is the only step where $z_\alpha$ might get added to $B$). Therefore, the condition in  Step~\ref{step:if-v-notin-V}, i.e., that $\bfv_\alpha^\Lambda \notin V$, was satisfied when $\alpha$ was processed. 

It follows that $\iota(\bfv_\alpha^\Lambda)\notin \sum_{T\subsetneq S}M_T$. This is seen by applying $\iota$ to $\bfv_\alpha^\Lambda \notin V$ and combining with \eqref{eq:lower-Ts}. It follows that $\iota(\bfv_\alpha^\Lambda)$ is extendable to a basis for a vector space complement to $\sum_{T\subsetneq S} M_T$ in $M_S$. In the language of Theorem~\ref{thm:garsia-LA-CM}, we can choose $B(L_S)$ to have $\iota(\bfv_\alpha^\Lambda)$ as a member. \label{fact:beta-is-in-L_S}

\item At the completion of the processing of any face $\alpha$, unless Step~\ref{step:if-color-condition-fails} was reached and the algorithm was terminated, there is a unique expression of $\bfv_\alpha^\Lambda$ as a $\kk$-linear combination of the basis $\{\bfv_\beta^\Lambda: z_\beta\in B\}$ for $V$, satisfying the condition in Step~\ref{step:if-color-condition-succeeds}. This is trivial if $\alpha$ already satisfied this condition before being processed (so that the processing of $\alpha$ ended up in Step~\ref{step:if-color-condition-succeeds}), but it is also true if $\alpha$ satisfied the condition in Step~\ref{step:if-v-notin-V}, because in this case processing $\alpha$ involved adding $z_\alpha$ to $B$, so that $\{\bfv_\beta^\Lambda: z_\beta\in B\}$ now contains $\bfv_\alpha^\Lambda$; the desired expression for $\bfv_\alpha^\Lambda$ as a linear combination of elements of $\{\bfv_\beta^\Lambda: z_\beta\in B\}$ then has the form $\bfv_\alpha^\Lambda = \bfv_\alpha^\Lambda$.\label{fact:color-condition-satisfied}

\end{enumerate}
    
With this preparation, we first consider the case where at some point over the course of the algorithm, Step~\ref{step:if-color-condition-fails} is reached. In this situation we have a face $\alpha$ such that all prior faces (in the order defined in Step~\ref{step:total-order}) have been processed, and we have a representation
\[
\bfv_\alpha^\Lambda = \sum_{\beta:z_\beta\in B} c_\beta \bfv_\beta^\Lambda
\]
with each $c_\beta\in \kk$, and at least one $\beta$ for which $J_\beta \nsubseteq J_\alpha$ satisfying $c_\beta \neq 0$. Sorting the terms on the right according to whether $J_\beta$ is contained in $J_\alpha$, this can be written
\begin{equation}\label{eq:M_S-meets-the-L_T's}
\bfv_\alpha^\Lambda - \sum_{\substack{\beta:z_\beta\in B \\ J_\beta\subseteq J_\alpha}}c_\beta \bfv_\beta^\Lambda = \sum_{\substack{\beta:z_\beta\in B \\ J_\beta\nsubseteq J_\alpha}}c_\beta \bfv_\beta^\Lambda,
\end{equation}
where the linear combination on the right is nontrivial, and the $\bfv_\beta^\Lambda$'s that appear in it are linearly independent; thus both sides are nonzero.

Apply $\iota$ to both sides of \eqref{eq:M_S-meets-the-L_T's}. The left side then lies in $M_S$, while by fact~\ref{fact:beta-is-in-L_S}, the $L_T$'s of Theorem~\ref{thm:garsia-LA-CM} can be chosen so that the right side lies in $\bigoplus_{T\nsubseteq S} L_T$. Thus,
\[
M_S \cap \left(\bigoplus_{T\nsubseteq S} L_T\right) \neq \{0\}.
\]
By Observation~\ref{obs:M_S-doesn't-meet-the-other-L_T's}, this means that $\Lambda$ cannot be Cohen--Macaulay.

Now suppose that instead, the algorithm reaches Step~\ref{step:output}, i.e., processes every face of $\widehat P(\Lambda)$ without ever reaching Step~\ref{step:if-color-condition-fails}. Then it follows from fact~\ref{fact:color-condition-satisfied} that every $\bfv_\alpha^\Lambda$ is uniquely expressible in the form
\begin{equation}\label{eq:express-any-alpha}
\bfv_\alpha^\Lambda = \sum_{\substack{\beta:z_\beta\in B \\ J_\beta\subseteq J_\alpha}} c_\beta \bfv_\beta^\Lambda,
\end{equation}
with each $c_\beta\in \kk$. 

Fix any label set $S\subset [n]$, and consider the $\alpha$'s in $\widehat P(\Lambda)$ satisfying $J_\alpha = S$, which are precisely the facets of the label-selected subcomplex $\Lambda_S$. Then the condition $J_\beta\subseteq J_\alpha$ in the sum on the right side of \eqref{eq:express-any-alpha} is precisely the condition that $\beta$ belong to the same label-selected subcomplex $\Lambda_S$, so the corresponding $z_\beta$'s are precisely those in the label-selected part $B_S$ of the proposed basis. (See Section~\ref{sec:garsia-LA-CM}, especially Proposition~\ref{prop:garsia-3.1}, for the notation.) Let
\[
\iota_S: \kk^{\{\alpha:J_\alpha = S\}} \rightarrow \bigoplus_{\alpha:J_\alpha = S}\kk z_\alpha = \kk[\Lambda_S]_{\sum_{j\in S}\bfe_j}
\]
be the analogue of the natural identification $\iota$ from fact~\ref{fact:identify-top-component-with-kn} for the label-selected complex $\Lambda_S$, and note that by Lemma~\ref{lem:what-is-the-facet-vector} it maps
\[
\bfv_\delta^{\Lambda_S}\mapsto \left(\prod_{j\in S\setminus J_\delta}\omega_j\right)z_\delta
\]
for any $\delta \in \widehat P(\Lambda_S)$. Then
\[
\bfv_\delta^{\Lambda_S} = \iota_S^{-1}\circ \left(\prod_{j\in [n]\setminus S}\omega_j\right)^{-1}\circ \iota (\bfv_\delta^\Lambda)
\]
for $\delta \in \widehat P(\Lambda_S)$, where the middle map in the composition on the right side is the one defined in \eqref{eq:def-of-inverse-omega}; note that it is well-defined here because $\iota(\bfv_\delta^\Lambda)$ lies in $M_S$, since $\delta$ belongs to the $S$-label selected part of $\Lambda$. In particular, applying the map
\begin{equation}\label{eq:move-to-subcomplex-map}
\iota_S^{-1}\circ \left(\prod_{j\in [n]\setminus S}\omega_j\right)^{-1}\circ \iota
\end{equation}
to \eqref{eq:express-any-alpha}, we get an expression
\begin{equation}\label{eq:express-alpha-rank-selected}
\bfv_\alpha^{\Lambda_S} = \sum_{\beta:z_\beta\in B_S} c_\beta \bfv_\beta^{\Lambda_S},
\end{equation}
of each $\bfv_\alpha^{\Lambda_S}$ (for $\alpha$ with $J_\alpha = S$) as a linear combination of the facet vectors in $\Lambda_S$ of the label-selected proposed basis $B_S$, and the injectivity of the map \eqref{eq:move-to-subcomplex-map} means that this linear combination is unique. Because $\{\bfv_\alpha^{\Lambda_S}: J_\alpha = S\}$ is the standard basis for the space $\kk^{\#\{\alpha:J_\alpha = S\}}$ of facet vectors of the label-selected subcomplex $\Lambda_S$, the existence and uniqueness of the linear combination \eqref{eq:express-alpha-rank-selected} imply that $\{\bfv_\beta^{\Lambda_S}: z_\beta \in B_S\}$ is also a basis. In other words,  the incidence matrix of $B_S$ in $\Lambda_S$ is square and nonsingular. All of this holds for every $S\subset [n]$, so $B$ is a basis for $\kk[\Lambda]$ over $\kk[\Omega]$ by Proposition~\ref{prop:garsia-3.1} (and thus also, $\Lambda$ is Cohen--Macaulay).
\end{proof}

\begin{remark}
    We take the opportunity to clear up an ambiguity in \cite{garsia}. It is important to the proof of correctness of Algorithm~\ref{alg:basis} that the order fixed in Step~\ref{step:total-order} be compatible with the containment order on the label sets; this was used to establish fact~\ref{fact:already-processed}, which gave us the important equation \eqref{eq:lower-Ts}. In \cite[p.~242]{garsia}, a specific order is fixed, which is described as lexicographic order on the label set blocks, and lexicographic order on chains within each block.\footnote{Recall that in the setting of \cite{garsia}, $\Lambda$ is the order complex of a ranked poset, so the faces $\alpha\in \widehat P(\Lambda)$ are chains in this poset, and labels are ranks; Garsia fixes once and for all a total order on the underlying poset that refines the poset order, and the lexicographic order on the chains with a given label/rank set is with respect to this total order.} There is a natural interpretation for ``lexicographic order on subsets" that would fail to respect containment order; on subsets of $[2]$, it would look like $\varnothing < \{1\}<\{1,2\}<\{2\}$. (Indeed, this seems to be the interpretation suggested by the discussion on \cite[pp.~238--9]{garsia}, as the word $12$ precedes the word $2$ lexicographically.) However, Garsia {\em must} have intended the reader to interpret ``lexicographic order on subsets" to mean an order that refines containment order, for example the ``length-lexicographic" order that on subsets of $[3]$ looks like
    \[
    \varnothing<\{1\}<\{2\}<\{3\}<\{1,2\}<\{1,3\}<\{2,3\}<\{1,2,3\}.
    \]
    If the order were not compatible with containment order, it would not be possible to infer the third displayed equation at the top of p.~244 in the proof of \cite[Theorem~3.3]{garsia}  from the second displayed equation.
\end{remark}

\begin{remark}
    Garsia's algorithm reduces the computation of a module basis for the $n$-dimensional ring $\kk[\Lambda]$ over the parameter subring $\kk[\Omega]$ to linear algebra in a single, finite-dimensional vector space $\kk^m \cong \kk \epsilon_1\oplus\dots\oplus \kk \epsilon_m$; underlying the algorithm's functioning is the filtration of this vector space by the subspaces $M_S$. A distinct but reminiscent approach to a related problem was taken in the thesis of Nicolas Borie, under the direction of Nicolas Thi\'ery (\cite[Chapitre~3]{borie}, see also \cite{borie-thiery}). The goal there was to compute a basis for a ring of permutation invariants over the subring of symmetric polynomials, just as in \cite{garsia-stanton}. The method used an evaluation homomorphism to reduce the problem to computations inside a finite-dimensional vector space carrying a filtration. The finite-dimensional vector space in question was the image of this  homomorphism, not a component of the ring of interest itself, and the filtration was over $\NN$ (coming from the degree grading of the invariant ring), rather than over the subsets of $[n]$. But the similarities are suggestive.
\end{remark}

\paragraph{Computing a representation on the basis.}\label{sec:computing-representations}

The proof of correctness  of the method for computing a $\kk[\Omega]$-basis for $\kk[\Lambda]$ described in Algorithm~\ref{alg:basis} (including all the involved lemmas) also implicitly contains a procedure that, given an arbitrary element of $\kk[\Lambda]$, computes a representation of it as a $\kk[\Omega]$-linear combination of the elements of a basis $B$ output by the algorithm. In outline, this procedure is as follows. It is sufficient to express standard monomials in $\kk[\Lambda]$. Induction on the first displayed equation in the proof of Lemma~\ref{lem:cells-generate} allows one to represent any standard monomial as a monomial in $\kk[\Omega]$ times a single $z_\alpha$, reducing the problem to expressing the $z_\alpha$'s in terms of $B$. Then, the $\kk[\Omega]$-module generation part of the proof of Proposition~\ref{prop:garsia-3.1} shows how to express any $z_\alpha$ in terms of $B$: it amounts to inverting the incidence matrix of $B_{J_\alpha}$ in $\Lambda_{J_\alpha}$. Rather than state a theorem, we illustrate by showing how to compute the representation of $y_w^2y_\beta\in \kk[\Sd\Delta]$ used in Example~\ref{ex:proof-illlustration} on the basis computed in Example~\ref{ex:running-ex-basis}. 

Referring to the face poset for $\Sd\Delta$ depicted in Figure~\ref{fig:subspaces-of-facet-space}, we have the expression
\[
y_w^2y_\beta = z_wz_{w\beta}
\]
for $y_w^2y_\beta$ as a standard monomial in the ASL generators $z_\delta$, $\delta\in \widehat P(\Sd\Delta)$ for $\kk[\Sd\Delta]$. Then by the first displayed equation in the proof of Lemma~\ref{lem:cells-generate}, we get
\begin{equation}\label{eq:put-in-omega}
y_w^2y_\beta = \omega_1 z_{w\beta},
\end{equation}
where $\omega_1 = z_v + z_w = y_v + y_w = \gamma_1$, so the problem is reduced to obtaining an expression for $z_{w\beta}$. The corresponding cell $w\beta\in \widehat P(\Sd\Delta)$ has label set $J_{w\beta}=\{1,2\}$, the entire label set, and its facet vector $\bfv_{w\beta}^{\Sd\Delta}$ is $(0,0,0,1)$. Consulting the facet vectors corresponding to the basis $\{1=z_\varnothing,z_v,z_\alpha,z_{v\alpha}\}$ computed in Example~\ref{ex:running-ex-basis}, we get the representation
\[
(0,0,0,1) = (1,1,1,1) - (1,0,1,0) - (1,1,0,0) + (1,0,0,0),
\]
or
\[
\bfv_{w\beta}^{\Sd\Delta} = \bfv_\varnothing^{\Sd\Delta} - \bfv_v^{\Sd\Delta} - \bfv_\alpha^{\Sd\Delta} + \bfv_{v\alpha}^{\Sd\Delta}.
\]
By Lemma~\ref{lem:what-is-the-facet-vector}, this expression tells us how to represent $z_{w\beta}$ as a $\kk[\Omega]$-linear combination of the basis:
\begin{equation}\label{eq:cell-representation}
z_{w\beta} = \omega_1\omega_2 - \omega_2 z_v - \omega_1 z_\alpha + z_{v\alpha}.
\end{equation}
We used the facet vectors in $\Sd\Delta$ because the label set of our target cell $w\beta$ is the entire label set $\{1,2\}$, but this whole computation with facet vectors would be done in the label-selected subcomplex $\Sd\Delta_S$ for a cell with given label set $S$.

Substituting \eqref{eq:cell-representation} into \eqref{eq:put-in-omega}, we get
\[
y_w^2y_\beta = \omega_1^2\omega_2 - \omega_1\omega_2 z_v - \omega_1^2 z_\alpha + \omega_1 z_{v\alpha},
\]
and translating the right side back into the familiar language of $\gamma$'s and $y$'s via $\gamma_j = \omega_j$ ($j=1,2$) and $z_{v\alpha}=y_vy_\alpha$, we recover the expression given in Example~\ref{ex:proof-illlustration}.

\section*{Acknowledgements}

This paper is dedicated to the memory of Adriano Garsia, whose work \cite{garsia, baclawski1981combinatorial, garsia-stanton} on combinatorial methods in Cohen--Macaulay rings was the source of our approach. The authors wish to thank Victor Reiner, Ashleigh Adams, Alexandra Pevzner, Laura Escobar, and Christopher Manon for useful conversations. We owe a particular debt of gratitude to Reiner for generously corresponding about many questions, including the discussion that led to Section~\ref{sec:nonconstructive}. We also thank Reiner and Ayah Almousa for encouraging us to work on the question addressed here. The authors are grateful to SLMath for their hospitality during research visits in July 2024 and July 2025, supported by NSF DMS-1928930. BBS was partially supported by Soledad Villar's NSF CAREER award, NSF CAREER 
2339682. SM is an associate at NITheCS (National Institute for Theoretical and Computational Sciences) in South Africa and would like to thank the institute for its ongoing support of her research.

\section*{Statement on AI use}

There was no use of generative AI at any stage of the research for or preparation of this manuscript.

\bibliographystyle{alpha}
\bibliography{biblio}

\end{document}